\documentclass[reqno,11pt]{amsart}
\usepackage{txfonts}
\usepackage{mathpazo}
\usepackage[utf8]{inputenc}
\usepackage[title]{appendix}
\usepackage{tikz}
\usepackage{tikz-cd}
\tikzcdset{scale cd/.style={every label/.append style={scale=#1},
    cells={nodes={scale=#1}}}}
\usepackage[default]{lato}
\usepackage{bold-extra}
\usefont{T1}{ptm}{b}{it}
\usepackage{graphicx}
\usepackage{color}
\usepackage{amssymb}
\usepackage{epsfig}
\usetikzlibrary{positioning}
\usepackage[colorlinks=true,citecolor=blue,urlcolor=blue,linkcolor=blue]{hyperref}
\usepackage{url}
\usepackage{setspace}
\usepackage{txfonts}
\usepackage{geometry}
\usepackage{graphics}
\usepackage{color}
\usepackage{setspace}
\usepackage{color,soul}
\usepackage{color}
\usetikzlibrary{matrix}
\usepackage{amsmath,amsthm,amscd,amsfonts,amssymb,euscript,amsthm,amsfonts,amssymb,amscd}
\usepackage{pdfpages}
\usepackage{mathtools}
\usepackage[widespace]{fourier}
\usepackage{color}
\usepackage{fancyhdr}
\usepackage{fullpage}
\usepackage{curves}

\usepackage{pdfpages}
\usepackage{mathtools}

\usepackage{latexsym}
\usepackage{color}

\theoremstyle{plain}
\newtheorem{exam}{Example}

\numberwithin{equation}{section}
\newcommand{\spec}{{\rm Spec}\,}
\newtheorem{thm}[subsection]{Theorem}
\newtheorem{prop}[subsection]{Proposition}

\newtheorem{lema}[subsection]{Lemma}
\newtheorem{coro}[subsection]{Corollary}
\theoremstyle{definition}
\newtheorem{rema}[subsection]{Remark}
\newtheorem{defe}[subsection]{Definition}

\newenvironment{claim}[1]{\par\noindent\underline{Claim:}\space#1}{}
\newenvironment{claimproof}[1]{\par\noindent\underline{Proof:}\space#1}{\leavevmode\unskip\penalty9999 \hbox{}\nobreak\hfill\quad\hbox{$\blacksquare$}}

\title{Relative Log-Symplectic structure on a semi-stable degeneration of moduli of Higgs bundles}
\author{Sourav Das}
\address{Department of Mathematics, University of Haifa, Israel}
\email{sdas6565@gmail.com}

\date{}

\begin{document}
\maketitle
\section*{Abstract}
In a recent paper \cite{3}, a semi-stable degeneration of moduli space of Higgs bundles on a curve has been constructed. In this paper, we show that there is a relative log-symplectic form on this degeneration, whose restriction to the generic fibre is the classical symplectic form discovered by Hitchin. We compute the Poisson ranks at every point and describe the symplectic foliation on the closed fibre. We also show that the closed fibre, which is a variety with normal crossing singularities, acquires a structure of an algebraically completely integrable system.

\section{\textbf{Introduction}}
Degenerations of moduli spaces of bundles on curves have had several interesting applications usually combined with the induction on the genus of the curve. For instance, proof of Newstead-Ramanan conjecture and proof of the factorisation theorem. It is, therefore, a natural question to ask whether there is a semi-stable degeneration of the moduli space of stable Higgs bundles of rank $n(\geq 2)$ on a smooth projective curve of genus $g(\geq 2)$. Such a degeneration has been recently constructed by Balaji et al. in \cite{3}, extending the techniques of Gieseker \cite{12} and Nagaraj-Seshadri \cite{25}. We recall that the construction begins with a choice of a degeneration of the smooth curve, i.e., a flat family of curves $\mathcal X$ over a complete discrete valuation ring $S$, whose generic fibre is a smooth projective curve of genus $\geq 2$ and the closed fibre is an irreducible nodal curve with a single node. Then one can construct a flat family of varieties over the discrete valuation ring $S$ such that
\begin{enumerate}
\item the generic fibre is isomorphic to the moduli of stable Higgs bundles (vector bundles) on the generic curve,
\item the total space of the family is regular over $\mathbb C$, and the closed fibre is a normal crossing divisor,
\item the closed fibre has a modular description; namely, the objects are certain admissible \eqref{Admi2021} Higgs bundles (admissible vector bundles) on certain semi-stable models of the nodal curve.
\end{enumerate}

Let us denote the degeneration of moduli of vector bundles by $\mathcal M_{GVB,S}$ and the closed fiber by $\mathcal M_{GVB}$. Let us denote the degeneration of moduli of Higgs bundles by $\mathcal M_{GHB,S}$ and the closed fibre by $\mathcal M_{GHB}$.

In \cite{14}, Hitchin proved that the moduli space of stable Higgs bundles on a compact Riemann surface has a natural holomorphic symplectic structure. The existence of the symplectic form can be seen in the following way. The cotangent bundle of the moduli space of vector bundles is a dense open subset of the moduli of Higgs bundles. Moreover, its complement has co-dimension two. Therefore the naturally occurring symplectic form on the cotangent bundle extends to the moduli space of Higgs bundles. Later in \cite{7}, Biswas and Ramanan and in \cite{8}, Bottacin studied the algebraic version of the symplectic form on the moduli of vector bundles on a smooth projective curve. 

Let $\mathcal M_S\rightarrow S$ be a semi-stable degeneration of a holomorphic-symplectic variety. We call it a \textbf{good degeneration} if there exists a relative log-symplectic form (definition \ref{RelLogSym}) on $\mathcal M_S$, whose restriction to the generic fibre is the given symplectic form.

We began with the following question. Is the semi-stable degeneration of moduli of Higgs bundles good in the above sense? To prove this, we compute the relative log-cotangent and log-tangent space of $\mathcal M_{GVB,S}$ at a given point in terms of the first-order infinitesimal logarithmic deformations of the objects of the moduli. Then we observe that the relative log-cotangent bundle $\Omega_{\mathcal M_{GVB,S/S}}(\mathcal M_{GVB})$ of $\mathcal M_{GVB, S}$ is a dense open subset of $\mathcal M_{GHB,S}$. By generality, $\Omega_{\mathcal M_{GVB,S/S}}(\mathcal M_{GVB})$ has a relative log symplectic form. Using the explicit description of the log-tangent space, we show that there is a skew-symmetric, non-degenerate bilinear form on the relative log-tangent space at any given point. We show that this form also coincides with the classical symplectic form on the generic fibre and also with the natural relative log symplectic form on $\Omega_{\mathcal M_{GVB,S/S}}(\mathcal M_{GVB})$. We can summarise the above discussion in the following theorem from \S 5.

\begin{thm}
There is a relative logarithmic-symplectic form on $\mathcal M_{GHB, S}$, whose restriction to the generic fibre is the classical symplectic form.
\end{thm}

For any variety $Z$, let us denote its singular locus by $\partial Z$. We see that the closed fibre of our degeneration has the following natural stratification.
\begin{equation}\label{Whit}
\mathcal M_{GHB}\supset \partial \mathcal M_{GHB}\supset \partial^2 \mathcal M_{GHB}\supset \dots
\end{equation}
 The log-symplectic form induces a Poisson structure on the closed fibre and every successive singular locus. In \S6 and \S7, we compute the Poisson rank at every point of $\mathcal M_{GHB}$ and show that the stratification by Poisson ranks coincides with the stratification given by the successive singular loci. To compute the Poisson rank, we first show that every smooth stratum is isomorphic, as a Poisson scheme, to a torus-quotient of a smooth variety equipped with an equivariant symplectic form (Corollary \ref{TQ1}). Then we compute the drop in the Poisson rank because of the torus-quotient (Lemma \ref{rank}). 

\begin{thm}\label{Foliation11}
The stratification of the Poisson variety $\mathcal M_{GHB}$ given by the successive degeneracy loci of the Poisson structure is the same as the stratification given by the successive singular loci. Moreover, $\partial^r \mathcal M_{GHB}\setminus \partial^{r+1}\mathcal M_{GHB}$ is a smooth Poisson sub-variety of dimension $2(n^2(g-1)+1)-r$ with constant Poisson rank $2(n^2(g-1)+1)-2r$. In particular, the most singular locus is a smooth Poisson variety of dimension $2(n^2(g-1)+1)-n$ with constant Poisson rank $2(n^2(g-1)+1)-2n$.
\end{thm}

Let $\widetilde{\mathcal M_{GHB}}$ denote the normalisation of the closed fibre $\mathcal M_{GHB}$ and $\partial \widetilde{\mathcal M_{GHB}}$ denote the inverse image of $\partial \mathcal M_{GHB}$. The pullback form equips $\widetilde{\mathcal M}_{GHB}$ with a log-symplectic structure. In \cite{24}, Matviichuk et al. showed that the log-cotangent bundle of a variety with a normal-crossing divisor has many natural log-symplectic forms other than the tautological one and any such form differs by a bi-residue, called the magnetic term. Moreover, any log-symplectic manifold is stably equivalent to the log-cotangent bundle of a normal-crossing divisor. We show that all the magnetic terms of the log-symplectic form on $\widetilde{\mathcal M}_{GHB}$ are zero. As a consequence, we obtain the local normal form of the Poisson structure on $\widetilde{\mathcal M}_{GHB}$.

In section \S 8, we describe the Casimir functions of the symplectic leaves of every strata $\partial^{r,o} \mathcal M_{GHB}:=\partial^r \mathcal M_{GHB}\setminus \partial^{r+1}\mathcal M_{GHB}$ of $\mathcal M_{GHB}$. For the notation, we refer to \S8. We show that every such strata $\partial^{r,o} \mathcal M_{GHB}$ is a free torus quotient of a smooth variety equipped with an equivarient symplectic form. Moreover, the latter variety has an equivariant momentum map. In this case, the momentum map descends to the stratum $\partial^{r,o} \mathcal M_{GHB}$ because the co-adjoint action of any torus is trivial. The Casimir functions of the strata $\partial^{r,o} \mathcal M_{GHB}$ are precisely the coordinate functions of the descended map. The following is the precise statement.

\begin{thm}\label{Casimir}
\begin{enumerate}
\item The map 
\begin{equation}
\mu_r: \mathcal M^{\chi, n,\epsilon, ad}_{HB, X_r}\rightarrow (T_{A_r,e})^{\vee} 
\end{equation}
defined by 
\begin{equation}
\mu_r(\mathcal E, \phi)(X_{\psi})=\lambda(\imath(X_{\psi}))=Trace(\phi\circ \imath(X_{\psi})), ~~\text{for}~~X_{\psi}\in H^0(X_r, T_{X_r})
\end{equation}
is a momentum map, where 
\begin{enumerate}
\item $\imath: H^0(X_r, T_{X_r})\hookrightarrow \mathbb H^1(\mathcal C_{\bullet})$ denotes the differential of the orbit map $A_r\rightarrow \mathcal M^{\chi, n,\epsilon, ad}_{HB, X_r}$ at the point $(\mathcal E,\phi)$.

\item $\lambda$ denotes the symplectic potential on $\mathcal M^{\chi, n,\epsilon, ad}_{HB, X_r}$ (\ref{Liou440} and remark \ref{Liou441}).
\end{enumerate} 

\item $\mu_r(\mathcal E,\phi)=(Trace ~~\phi|_{\mathcal O_{R[r]_1}(1)^{\oplus a_1}}, \dots, Trace ~~\phi|_{\mathcal O_{R[r]_r}(1)^{\oplus a_r}})$, where $\mathcal E|_{R[r]_i}\cong \mathcal O_{R[r]_i}(1)^{\oplus a_i}\oplus \mathcal O_{R[r]_i}^{\oplus b_i}$ for every $i=1,\dots, r$. 

\item The coordinate functions of $\mu_r$ are the Casimir functions of $\mathcal M^{a_{\bullet}}_{GHB}$ \eqref{admi445}. In particular, the variety $\mu^{-1}_r(0)\cap \mathcal M^{a_{\bullet}}_{GHB}$ is a symplectic leaf of $\mathcal M^{a_{\bullet}}_{GHB}$ containing $\Omega_{\mathcal M^{a_{\bullet}}_{GVB}}$. Moreover, it consists of triples $(X_r, \mathcal E, \phi)$ such that the trace of $\phi|_{\mathcal O_{R[r]_i} (1)^{\oplus a_i}}: \mathcal O_{R[r]_i} (1)^{\oplus a_i}\rightarrow \mathcal O_{R[r]_i} (1)^{\oplus a_i}$ is zero for all $i=1,\dots , r$.
\end{enumerate}
\end{thm}

In \cite{3}, Balaji et al. showed that there exists a proper Hitchin map $h: \mathcal M_{GHB}\rightarrow B:=\oplus_{i=1}^n H^0(X_0, \omega_{X_0}^{\otimes i})$ on the moduli space of stable Gieseker-Higgs bundles. In the final section of this article, we recall the definition of an algebraically completely integrable system structure (ACIS) on a variety with normal-crossing singularities. Following the strategy of \cite{23} and \cite{22}, we prove the following.

\begin{thm}
The general fibre $h^{-1} (\xi)$ corresponding to a spectral vine curve ramified outside the nodes is Lagrangian in a symplectic leaf for the log-symplectic structure on $\mathcal M_{GHB}$. Therefore the Hitchin map $h: \mathcal M_{GHB}\rightarrow B$ is an algebraically completely integrable system (\ref{ACIS}).
\end{thm}

\subsection{Acknowledgement} I am grateful to Vikraman Balaji for explaining the paper \cite{3} and for asking this question. I thank Najmuddin Fakhruddin, Rohith Varma and Suratno Basu for several discussions and comments, which helped write this article. 

\subsection{Notation and convention}
\begin{enumerate}\label{ConNot}
\item[$\bullet$] $k:=\mathbb C$= the field of complex numbers.
\item[$\bullet$] $k[\epsilon]$:= the ring of dual numbers over $k$.
\item[$\bullet$] $\mathbb N:=$ monoid of all positive integers with respect to "$+$",
\item[$\bullet$] $\overline{\mathbb N}= \mathbb N\cup \{0\}$, monoid with respect to "$+$"
\item[$\bullet$]{Standard Log Point:} A monoid structure on $\spec k$ given by a morphism of monoids $\overline{\mathbb N}\rightarrow k$ which maps $0\mapsto 1$ and $n\mapsto 0$ for all $n\neq 0$.
\item[$\bullet$] $S:=\{\eta, \eta_0\}$ denotes the spectrum of a complete discrete valuation ring, where $\eta$ denotes the generic point and $\eta_0$ denotes the closed point.
\item[$\bullet$] $r$:=rank, $d:=$ the degree and $\chi:=$ the Euler characteristic of the vector bundles. We will assume throughout that $(r,d)=1$ or equivalently $(r, \chi)=1$.
\item[$\bullet$] $\mathcal X$ denotes a flat family of curves whose generic fibre is smooth projective and the closed fibre is a nodal curve with a single node. We denote the nodal curve by $X_0$ and the node by $x$. We denote its normalisation by $q:\tilde X_0\rightarrow X_0$ and the two preimages of the node $x$ by $\{x^+,x^{-}\}$. 
\item[$\bullet$] for a category ${F}$ fibered in groupoids over the category of schemes $\underline{{Schs}}$ and for any scheme $U$ the notation ${F}_{_{U}}$ denotes the category whose objects are $U$-spaces and the morphisms are $U$-morphisms \cite[Definition 2.1]{21}
\item[$\bullet$] for an algebraic stack $\mathcal L$ we denote by $T\mathcal L$ the tangent stack of $\mathcal L$.
\item[$\bullet$] for any scheme/algebraic stack $\mathcal L$ we denote by $\mathcal L(k)$ the set of all $\spec k$-valued points of $\mathcal L$ and by $\mathcal L(k[\epsilon])$ the underlying vector space of all $\spec k[\epsilon]$-valued points of $\mathcal L$.

\item[$\bullet$] for any local ring $\mathcal O$ we will denote by $\mathcal O^{{h}}$ the Henselization at the maximal ideal.

\end{enumerate}

\section{\textbf{Preliminaries}}

\subsection{On Moduli of Gieseker-Higgs bundles}

Let $X_0$ be a projective irreducible nodal curve of genus $g\geq 2$ with a single node $x$. Let $q:\tilde{X}_0\rightarrow X_0$ be the normalisation and $q^{-1}(x)=\{x^+, x^-\}$. 

\begin{defe} The dualising sheaf of the nodal curve $X_0$ is the kernel of the following morphism of $\mathcal O_{X_0}$-modules.
\begin{equation}
q_*\Omega_{\tilde X_0}(x^++x^-)\rightarrow \mathbb C_x, 
\end{equation}
where 
\begin{enumerate}
\item $\mathbb C_x$ denotes the sky-scraper sheaf at the point $x$.
\item the map $q_*\Omega_{\tilde X_0}(x^++x^-)\rightarrow \mathbb C_x$ is given by 
\begin{equation}\label{res908}
s\mapsto Res(s; x^+)+Res(s;x^-)
\end{equation}
\end{enumerate}
We denote it by $\omega_{X_0}$. Here, $Res(s; x)$ denotes the residue of a form $s$ at a point $x$.
\end{defe}

\begin{rema}
Notice that the fibres $\omega_{\tilde X_0}(x^++x^-)_{x^+}$ and $\omega_{\tilde X_0}(x^++x^-)_{x^-}$ can be identified with $\mathbb C$, using Poincare adjunction formula. More precisely, for any coordinate function $z^+$ around $x^+$ with $z^+(x^+)=0$, the image of $\frac{dz^+}{z^+}$ in $\omega_{\tilde X_0}(x^++x^-)_{x^+}$ is independent of the choice of the coordinate function; the above identification between $\omega_{\tilde X_0}(x^++x^-)_{x^+}$ and $\mathbb C$ sends this independent image to $1\in \mathbb C$. Similarly, at $x^-$. Therefore the map \eqref{res908} makes sense. 
\end{rema}

\begin{rema}
The dualising sheaf can be defined for any nodal curve similarly. To be more precise, let $C$ be a nodal curve and $D$ denote the set of nodes. Let $q: \tilde{C}\rightarrow C$ denote the normalisation and $\tilde D$ denote the preimage $q^{-1}(D)$. Then the dualising sheaf $\omega_C$ is the kernel of the map 
\begin{equation}
q_*\Omega_{\tilde C}(\tilde D)\rightarrow \oplus_{x\in D} \mathbb C_x,
\end{equation}
where the map is constituted out of the maps \eqref{res908} at every point $x\in D$.  
\end{rema}

\begin{defe}\label{Higgs440}
Let $C$ be a nodal curve. A Higgs bundle on $C$ is a pair $(\mathcal E, \phi)$, where
\begin{enumerate}
\item $\mathcal E$ is a vector bundle on $C$, and 
\item $\phi: \mathcal E\rightarrow \mathcal E\otimes \omega_C$ any $\mathcal O_C$-module homomorphism.
\end{enumerate} 
\end{defe}

\begin{defe}\label{gis} Let $r$ be a positive integer.
\begin{enumerate}
\item A chain of projective lines is a scheme $R[r]$ of the form
$\cup_{i=1}^{^l} R[r]_{i}$ such that 
\begin{enumerate}
\item $R[r]_{i}\cong \mathbb{P}^{^1}$, 
\item for any $i<j$, $R[r]_{i}\cap R[r]_{j}$ consists of a single point $p_j$ if $j=i+1$ and empty otherwise. 
\end{enumerate}

We call $r$ the length of the chain $R[r]$. Let us choose and fix two smooth points $p_1$ and $p_{r+1}$ on $R[r]_1$ and $R[r]_{r}$, respectively. 

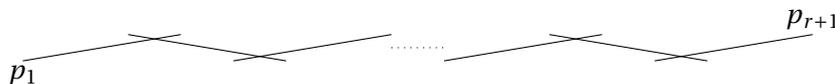
\vspace{1cm}
\begin{tikzpicture}[overlay, xshift= 1cm,scale=0.70]
\draw node[yshift=-1ex]{$p_1$}(0,0) -- (3,.5) node[yshift=1ex]{}; \draw (2,.5) -- (5,0)node[yshift=-1ex]{}; \draw (4,0) -- (7,.5); \draw[dotted] (7,.25) -- (8,.25);\draw (8,0) -- node[yshift=1.5ex]{}(11,.5); \draw (10,.5) -- (13,0)node[yshift=-1ex]{}; \draw (12,0) -- (15,.5) node[yshift=1ex]{$p_{r+1}$};
\end{tikzpicture}
\vspace{1cm}

\item A Gieseker curve $X_r$ is the categorical quotient of the disjoint union of the curves $\tilde{X}_0$ and $R[r]$ obtained by identifying $x^+$ with $p_1$ and $x^-$ with $p_{r+1}$.
\end{enumerate}
\end{defe}

\begin{rema}\label{Rem101} There is a natural morphism $\pi_r: X_r\rightarrow X_0$ that contracts the chain $R[r]$ to the node $x$ and that is isomorphism outside. It is easy to see that the pullback of the dualising sheaf $\omega_{X_0}$ to a Gieseker curve $X_r$ is isomorphic to the dualising sheaf $\omega_{X_r}$ of $X_r$. The sheaf $\omega_{X_r}$ can be constructed by gluing $\omega_{\tilde X_0}(x^++x^-)$ and $\omega_{X_r}|_{R[r]}\cong \mathcal O_{R[r]}$ by the following identifications

\begin{equation}\label{IDEN}
P_{x^+}:\omega_{\tilde X_0}(x^++x^-)_{x^+}\xrightarrow{\frac{dz^+}{z^+}\mapsto \mathbb 1} \mathcal O_{R[r], p_1}~~~\hspace{3cm}~~\text{and}~~~\hspace{3cm}~~P_{x^-}:\omega_{\tilde X_0}(x^++x^-)_{x^-}\xrightarrow{\frac{dz^-}{z^-}\mapsto -\mathbb 1} \mathcal O_{R[r], p_{r+1}}
\end{equation}

Such a curve is called a semi-stable model of the stable curve $X_0$. In the literature, a semi-stable curve is also referred to as a pre-stable curve. 
\end{rema}

\subsubsection{\textbf{Choice of a degeneration of curves.}}\label{DegeCourbes} Let us choose a flat family of projective curves $\mathcal X\rightarrow S$, such that 
\begin{enumerate}
\item the generic fibre $\mathcal X_{\eta}$ is a smooth curve of genus $g\geq 2$, 
\item the closed fibre is the nodal curve $X_0$, and
\item the total space $\mathcal X$ is regular over $\spec \mathbb C$. 
\end{enumerate}

The existence of such a family follows from \cite[Theorem B.2 and Corollary B.3, Appendix B]{22}. Let us denote the relative dualising sheaf by $\omega_{\mathcal X/S}$.

Moreover, it follows from \cite[17.16.3 (ii)]{Gro II} that there exists an etale neighbourhood $S'\rightarrow S$ of $\eta_0$ (the closed point of $S$) such that the morphism $\mathcal X': \mathcal X\times_S S'\rightarrow S'$ has a section $\sigma: S'\rightarrow \mathcal X'$ which passes through smooth points of the morphism $\mathcal X'\rightarrow S'$.

\begin{defe}For every $S$-scheme $T$, a modification is a commutative diagram
\begin{equation}
\begin{tikzcd}
\mathcal X^{mod}_T\arrow{dr}{p_T}\arrow{rr}{\pi_T}&& \mathcal X_T:=\mathcal X\times_S T\arrow{dl}\\
& T
\end{tikzcd}
\end{equation}
such that
\begin{enumerate}
\item $p_T : \mathcal X^{mod}_T \rightarrow T$ is flat,
\item the horizontal morphism is finitely presented which is an isomorphism when $(\mathcal X_T )_t$ is smooth,
\item over each closed point $t\in T$ over $\eta_0\in S$, we have $(\mathcal X^{mod}_T )_t \cong X_r$ for some integer $r$ and the horizontal morphism restricts to the morphism which contracts the $\mathbb P^1$’s on $X_r$.
\end{enumerate}
\end{defe}

We will also alternatively call such modifications as Gieseker curves. We call two such modifications $\mathcal X^{mod}_T$ and $\mathcal X'^{mod}_T$ isomorphic if there exists an isomorphism $\sigma_T: \mathcal X^{mod}_T\rightarrow \mathcal X'^{mod}_T$ such that the following diagram commutes
\begin{equation}
\begin{tikzcd}
\mathcal X^{mod}_T\arrow{dr}{\pi_T} \arrow{rr}{\sigma_T} && \mathcal X'^{mod}_T\arrow{dl}{\pi'_T}\\
& \mathcal X_T
\end{tikzcd}
\end{equation}

\begin{rema}\label{prestable}
By definition, a modification is a pre-stable curve over the base $T$. From remark \ref{Rem101}, it follows that the pullback of the relative dualising sheaf of $\mathcal X_T/T$ is isomorphic to the relative dualising sheaf of $\mathcal X^{mod}_T/T$. We denote it by $\omega_{\mathcal X^{mod}_T/T}$.
\end{rema}

\begin{defe}\label{SS440}\label{Admi2021} A vector bundle $\mathcal E$ of rank $n$ on $X_r$ with $r\geq 1$ is called a Gieseker vector bundle if
 \begin{enumerate}
\item $\mathcal E|_{R[r]}$ is a strictly standard vector bundle on $X_r$, i.e., for each $i=1,\dots, r $, $\exists$ non-negative integers $a_i$ and $b_i$ such that $\mathcal E|_{R[r]_i}\cong \mathcal O^{\oplus a_i}\oplus \mathcal O(1)^{\oplus b_i}$, and 
\item the direct image $(\pi_r)_*(E)$ is a torsion-free $\mathcal O_{X_0}$-module.
\end{enumerate}
Any vector bundle on $X_0$ is called a Gieseker vector bundle. In the literature, a Gieseker vector bundle is also called an admissible vector bundle.

A Gieseker vector bundle on a modification $\mathcal X^{mod}_T$ is a vector bundle such that its restriction to each $(\mathcal X^{mod}_{T})_t$ is a Gieseker vector bundle.
\end{defe}

\begin{defe} A Gieseker–Higgs bundle on $\mathcal X^{mod}_T$ is a pair
$(\mathcal E_T , \phi_T )$, where $\mathcal E_T$ is a vector bundle on $\mathcal X^{mod}_T$, and $\phi_T : \mathcal E_T \rightarrow \mathcal E_T \otimes \omega_{\mathcal X^{mod}_T/T}$ is an $\mathcal O_{\mathcal X^{mod}_T}$ -module homomorphism satisfying the following
\begin{enumerate}
\item $\mathcal E_T$ is a Gieseker vector bundle on $\mathcal X^{mod}_T$,
\item for each closed point $t \in T$ over $\eta_0 \in S$, the direct image $(\pi_t)_*(\mathcal E_t)$ is a torsion-free sheaf on $X_0$ and $(\pi_t)_*\phi_t: (\pi_t)_*(\mathcal E_t)\rightarrow (\pi_t)_*(\mathcal E_t)\otimes \omega_{X_0}$ is an $\mathcal O_{X_0}$-module homomorphism. We refer to such a pair $((\pi_t)_*(\mathcal E_t), (\pi_t)_*\phi_t)$ as a torsion-free Higgs pair on the nodal curve $X_0$.
\end{enumerate}
\end{defe}
\begin{defe}
A Gieseker–Higgs bundle $(\mathcal E_T, \phi_T )$ is called stable if the direct image $(\pi_T)_*(\mathcal E_T, \phi_T )$ is a family
of stable torsion-free Higgs pairs on $\mathcal X_T$ over $T$.
\end{defe}

We define
\begin{equation}\label{Aut2025}
Aut(X_r/X_0)=
\left\{
\begin{array}{@{}ll@{}}
\text{automorphisms of}~~ X_r,~~\text{which commute}\\
\text{with the projection morphism to} ~~X_0.
\end{array}\right\}
\end{equation}

Notice that $Aut(X_r/X_0)$ is also the subgroup of $Aut(X_r)$, which consists of all the automorphisms, which are the identity morphism on the sub curve $\tilde{X_0}$.

\begin{defe}\label{GisEq}
\begin{enumerate}
\item Two Gieseker vector bundles $(\mathcal X^{mod}_T, \mathcal E_T)$ and $(\mathcal X'^{mod}_T, \mathcal E'_T)$ are called equivalent if there exists an isomorphism $\sigma_T: \mathcal X^{mod}_T\rightarrow \mathcal X'^{mod}_T$ such that $\sigma_T$ commutes with the projection map $\pi_T$ and $\sigma^*_T\mathcal E'_T$ is isomorphic to $\mathcal E_T$ as vector bundles over $\mathcal X^{mod}_T$. 

\item Two Gieseker-Higgs bundles $(\mathcal X^{mod}_T, \mathcal E_T,\phi_T)$ and $(\mathcal X^{mod}_T, \mathcal E'_T, \phi'_T)$ are called equivalent if there exists an isomorphism $\sigma_T: \mathcal X^{mod}_T\rightarrow \mathcal X'^{mod}_T$ such that $\sigma_T$ commutes with the projection map $\pi_T$ and $(\sigma^*_T\mathcal E'_T, \sigma^*_T\phi'_T)$ is isomorphic to $(\mathcal E_T, \phi_T)$ as Higgs bundles over $\mathcal X^{mod}_T$. 
\end{enumerate}
\end{defe}

\begin{defe}\cite[Definition 3.4, 3.6, 3.8]{3}\hspace{2cm}\label{Definition101}
\begin{enumerate}
\item \textsf{Functor of Gieseker curves $F_{GC,S}$:}
We define the functor of Gieseker curves 
$$
F_{GC,S}: {Sch}/S\rightarrow {Sets} $$ 

\begin{equation}
T
\mapsto \left\{
\begin{array}{@{}ll@{}}
\text{Isomorphism classes of}\\
\text{modifications}~~\mathcal X^{mod}_T\rightarrow \mathcal X_T
\end{array}\right\}
\end{equation}

\item \textsf{Functor of Gieseker vector bundles $F_{GVB,S}$:}
We define the functor of Gieseker vector bundles 

$$F_{GVB,S}: {Sch}/S\rightarrow {Sets}$$  
\begin{equation}
T
\mapsto \left\{
\begin{array}{@{}ll@{}} 
\text{Gieseker-equivalent classes of families of Gieseker}\\
\text{ vector bundles i.e., pairs} ~~(\mathcal X^{mod}_T, \mathcal E_T), \text{where}~~ \mathcal X^{mod}_T\\
 \text{ is a family of Gieseker curves and} \mathcal E_T~~\text{ a family of }\\
\text{ Gieseker vector bundles on}~~ \mathcal X^{mod}_T
\end{array}\right\}
\end{equation}

\item \textsf{Functor of Gieseker-Higgs bundles $F_{GHB,S}$:}
We define the functor of Gieseker-Higgs bundles 

$$F_{GHB,S}: {Sch}/S\rightarrow {Sets}$$ 
\begin{equation}
T
\mapsto \left\{
\begin{array}{@{}ll@{}} 
\text{Gieseker-equivalent classes of families of Gieseker-Higgs bundles i.e., triples}\\
(\mathcal X^{mod}_T, \mathcal E_T, \phi_T: \mathcal E_T\rightarrow \mathcal E_T\otimes \omega_{\mathcal X^{mod}_T/T}), \text{where}~~ \mathcal X^{mod}_T~~ \text{ is a family of Gieseker curves} \\
\text{and}~~\mathcal E_T~~\text{ a family of Gieseker vector bundles on}~~\mathcal X^{mod}_T~~\text{ and }~~\phi_T: \mathcal E_T\rightarrow \mathcal E_T\otimes \omega_{\mathcal X^{mod}_T/T}\\
\text{any}~~ \mathcal O_{\mathcal X^{mod}_T}-\text{module homomorphism}
\end{array}\right\}
\end{equation}

\end{enumerate}
\end{defe}

Let us denote by $F^{st}_{GVB,S}$ and $F^{st}_{GHB,S}$ the open subfunctors of stable Gieseker vector bundles and stable Gieseker-Higgs bundles, respectively. Now we recall few results from \cite{3} and \cite{25}, which are necessary for further discussion.

\begin{enumerate}
\item \cite[Theorem 2]{25} Assume $(n,d)=1$. The functor of stable Gieseker vector bundles $F^{st}_{GVB,S}$ is represented by a scheme $\mathcal M_{GVB,S}$ which is projective and flat over $S$. Let us denote the closed fibre by $\mathcal M_{GVB}$. The variety $\mathcal M_{GVB,S}$ is regular as a scheme over $k$, and the closed fibre $\mathcal M_{GVB}$ is a normal crossing divisor.

\item \cite[Theorem 1.1]{3} Assume $(n,d)=1$. The functor of stable Gieseker-Higgs bundles $F^{st}_{GHB,S}$ is represented by a scheme $\mathcal M_{GHB,S}$ which is quasi-projective and flat over $S$. Let us denote the closed fibre by $\mathcal M_{GHB}$. The variety $\mathcal M_{GHB,S}$ is regular as a scheme over $k$, and the closed fibre $\mathcal M_{GHB}$ is a normal crossing divisor. Moreover, there is a Hitchin map $h_S: \mathcal M_{GHB,S}\rightarrow B_S$ to an affine space over $S$. Moreover, the map $h$ is proper. 
\end{enumerate}

\subsubsection{\textbf{Construction of the moduli Gieseker vector bundles and Gieseker-Higgs bundles}} Let us briefly recall the constructions of moduli of Gieseker vector bundles (Gieseker-Higgs bundles) from \cite[Section 3]{25} (\cite[Section 5.3]{3}).

Let us choose a relatively ample line bundle $\mathcal O_{\mathcal X/S}(1)$ for the family of curves $\mathcal X/S$. The set of all flat families of stable torsion-free sheaves (Higgs pairs) of degree $d$ and rank $n$ over $\mathcal X$ forms a bounded family. Therefore we can choose a large integer $m$ such that given any family of stable torsion-free Higgs pairs $(\mathcal F_S, \phi_S)$, the sheaf $\mathcal F_s\otimes \mathcal O_{\mathcal X_s}(m)$ is generated by global sections and $H^1(\mathcal X_s, \mathcal F_s\otimes \mathcal O_{\mathcal X_s}(m))=0$ for every geometric point $s\in S$. Set $N:=H^0(\mathcal X_s, \mathcal F_s\otimes \mathcal O_{\mathcal X_s}(m))$ for any geometric point $s\in S$. We denote by $Grass(N,n)$ the Grassmannian of $n$ dimensional quotient vector spaces of $\mathbb C^N$.

\begin{defe}
Let $\mathcal G_S:{Sch}/S\rightarrow {Sets}$ be the functor defined as follows:
\begin{equation}
\mathcal G_S(T)=\{(\Delta_T, V_T)\},
\end{equation}
where
\begin{equation}
\Delta_T\subset \mathcal X\times_S T\times {Grass}(N,n)
\end{equation}
is a closed subscheme and $V_T$ is a vector bundle on $\Delta_T$ such that
\begin{enumerate}
\item the projection $j: \Delta_T\rightarrow T\times {Grass}(N,n)$ is a closed immersion,
\item the projection $\Delta_T\rightarrow \mathcal X\times_S T$ is a modification,
\item the projection $p_T:\Delta_T\rightarrow T$ is a flat family of Gieseker curves,
\item Let $\mathcal V$ be the tautological quotient bundle of rank $n$ on ${Grass}(N,n)$ and $\mathcal V_T$
its pullback to $T\times {Grass}(N,n)$. Then
\begin{equation}
V_T:=j^*(\mathcal V_T)
\end{equation}
be such that $V_T$ is a Gieseker vector bundle on the modification $\Delta_T$ of rank $n$ and degree
$d':=N+n(g-1)$.
\item for eact $t\in T$, the quotient $\mathcal O_{\Delta_t}^N\rightarrow V_t$ induces an isomorphism
\begin{equation}
H^0(\Delta_t, \mathcal O_{\Delta_t}^N)\cong H^0(\Delta_t, V_t)
\end{equation}
and $H^1(\Delta_t, V_t)=0$.
\end{enumerate}
\end{defe}

We denote by $P$ the Hilbert polynomial of the closed subscheme $\Delta_s$ of $\mathcal X_s\times Grass(N,n)$ for any geometric point $s\in S$ with respect to the polarisation $\mathcal O_{\mathcal X_s}(1)\boxtimes \mathcal O_{Grass(N,n)}(1)$, where $\mathcal O_{Grass(N,n)}(1)$ is the line bundle $det~~\mathcal V$.

\begin{rema}\label{GIT420}
It is shown in \cite[Proposition 8]{25} that the functor $\mathcal G_S$ is represented by a $PGL(N)$-invariant open subscheme
$\mathcal Y_S$ of the Hilbert scheme $\mathcal H_S:={Hilb}^P(\mathcal X\times {Grass}(N,n))$. Moreover, the subfunctor $\mathcal G^{st}_S$
of stable Gieseker vector bundles is represented by an open subscheme $\mathcal Y^{st}_S$ of $\mathcal Y_S$. The moduli of Gieseker vector bundles 
\begin{equation}
\mathcal M_{GVB,S}:=\mathcal Y^{st}_S\parallelslant PGL(N)
\end{equation}

is the GIT quotient. Moreover, the
action of $PGL(N)$ is free; therefore $\mathcal Y^{st}_S\rightarrow \mathcal M_{GVB,S}$ is a principal $PGL(N)$-bundle.
\end{rema}

\begin{rema}\label{Cur300}
Let $\Delta_{\mathcal Y_S}$ be the universal object defining the functor $\mathcal G^{st}_S$. By definition, we have the following closed immersion
\begin{equation}
\begin{tikzcd}
\Delta_{\mathcal Y^{st}_S}\arrow[hook]{r}\arrow{dr} & \mathcal Y^{st}_S\times {Grass}(N,n)\arrow{d}\\
& \mathcal Y^{st}_S
\end{tikzcd}
\end{equation}
More precisely, $\Delta_{\mathcal Y^{st}_S}=\{(y, x)\in \mathcal Y^{st}_S\times Grass(N,n)| ~~y\in \mathcal Y^{st}_S ~~\&~~x\in \Delta_{y}\}$. Here $\Delta_y$
denotes the fiber of the morphism $\Delta_{\mathcal Y^{st}_S}\rightarrow \mathcal Y^{st}_S$ over the point $y\in \mathcal Y^{st}_S$. Using this description, it is clear that the action of $PGL(N)$ on $\mathcal Y^{st}_S\times Grass(N,n)$ restricts to an action on the subscheme $\Delta_{\mathcal Y^{st}_S}$ such that the morphism is equivariant under the action of $PGL(N)$. Since the action of $PGL(N)$ is free on $\mathcal Y^{st}_S$ the action is also free on $\Delta_{\mathcal Y^{st}_S}$.
\end{rema}

\begin{defe}
We define a functor
\begin{equation}
\mathcal G^{H}_S: {Sch}/\mathcal Y_S\rightarrow {Groups}
\end{equation}
which maps $$T\rightarrow H^0(T, (p_T)_*(\mathcal End~~\mathcal V_T\otimes \omega_{\Delta_T/T})),$$

where $p_T: \Delta_T:=\Delta_{_{\mathcal Y_S}}\times_{_{\mathcal Y_S}}
T\rightarrow T$ is the projection, and $\omega_{\Delta_T/T}$ denotes the relative dualising sheaf of the family of curves $p_T$.
\end{defe}

\begin{rema}\label{GIT421}
Since $\mathcal Y_S$ is a reduced scheme the functor $\mathcal G^{H}_S$ is representable i.e., there exists a linear $\mathcal Y_S$-
scheme $\mathcal Y^{H}_S$ which represents it. For a $S$-scheme $T$, a point in $\mathcal G^{H}_S(T)$ is given by
$(V_T , \phi_T)$, where
\begin{enumerate}
\item $V_T \in \mathcal G_S(T)$, and
\item $(V_T, \phi_T)$ is a Gieseker–Higgs bundle.
\end{enumerate}
The subfunctor $\mathcal G^{H,st}_S$ of stable Gieseker-Higgs bundles is represented by an open subscheme $\mathcal Y^{H,st}_S$ of
$\mathcal Y^{H}_S$. The moduli of Gieseker-Higgs bundles 
\begin{equation}
\mathcal M_{GHB,S}:=\mathcal Y^{H,st}_S\parallelslant PGL(N)
\end{equation}
is the GIT quotient. As before, the action of $PGL(N)$ is free and therefore $\mathcal Y^{H,st}_S\rightarrow \mathcal M_{GHB,S}$ is a principal $PGL(N)$-bundle. Let us pullback the universal curve $\Delta_{\mathcal Y_S}$ via the morphism $\mathcal Y^{H,st}_S \rightarrow \mathcal Y^{st}_S$ and denote it by $\Delta_{\mathcal Y^{H}_S}$. The action of $PGL(N)$ lifts to an action on $\Delta_{\mathcal Y^H_S}$. The morphism $\Delta_{\mathcal Y^{H}_S}\rightarrow \mathcal Y^{H,st}$ is equivariant under the action of $PGL(N)$.
\end{rema}

\subsection{On Poisson structures on schemes}
Let $X$ be a scheme over $\mathbb C$. For any positive integer $k$, we write $\mathcal X^k_{X}:=(\Omega^k_X)^{\vee}$, the dual of the $\mathcal O_X$-module $\Omega^k_X$. This is the $\mathcal O_X$-module of alternating $k$-multiliear forms on $\Omega^1_X$. The natural map $\wedge^k T_X\rightarrow \wedge^k \mathcal X^k_X$ is an isomorphism for $k=1$ but need not be isomorphism in the higher degrees. We refer to the sections of $\mathcal X^k_X$ as $k$-derivations.

\begin{defe}\cite[Definition 1]{13}
A Poisson scheme is a pair $(X,\sigma)$, where $X$ is a scheme and $\sigma\in H^0(X, \mathcal X^2_X)$ is a $2$-derivation such that the $\mathbb C$-bilinear morphism 
\begin{equation}
\{\cdot, \cdot\}: \mathcal O_X\times \mathcal O_X\rightarrow \mathcal O_X,
\end{equation}
$$(g,h)\mapsto \sigma(dg\wedge dh)$$
 defines a Lie algebra structure on $\mathcal O_X$. This Lie bracket is the Poisson bracket. 
 
Using the Hom-Tensor duality we get, $Hom(\Omega_X, T_X)\cong H^0(X, \mathcal X^2_X)$. Therefore, the $2$-derivation $\sigma$ induces an $\mathcal O_X$-linear map (the anchor map) $\sigma^{\flat}: \Omega^1_X\rightarrow T_X$ defined by 
\begin{equation}
\sigma^{\flat}(\alpha)(\beta)=\sigma(\alpha\wedge \beta)
\end{equation}
for all $\alpha, \beta\in \Omega^1_X$. We say that $(X,\sigma)$ is a smooth Poisson scheme if the underlying scheme $X$ is smooth.
\end{defe}

\begin{defe}\cite[Definition 2]{13}\label{PoissonMorphism}
Let $(X,\sigma)$ and $(Y, \eta)$ be Poisson schemes with corresponding brackets $\{\cdot, \cdot\}_X$ and $\{\cdot, \cdot\}_Y$. A morphism $f: X\rightarrow Y$ is a Poisson morphism if it preserves the Poisson brackets, i.e., the pull-back morphism $f^*: O_Y\rightarrow f_*\mathcal O_X$ satisfies 
\begin{equation}
f^*\{g,h\}_Y=\{f^*g, f^*h\}_X
\end{equation}
for all $g,h\in \mathcal O_Y$. Equivalently, $f$ is a Poisson morphism if the following diagram is commutative.
\begin{equation}
\begin{tikzcd}
\Omega_X \arrow{d}{\sigma^{\flat}_X} & f^*\Omega_Y\arrow{d}{f^*\sigma^{\flat}_Y}\arrow{l}{(df)^*}\\
T_X\arrow{r}{df}& f^*T_Y
\end{tikzcd}
\end{equation}
\end{defe}

\begin{defe}
Let $(X,\sigma_X)$ be a Poisson scheme. We say the Poisson scheme $(Y, \sigma_Y)$ is a Poisson subscheme of $(X,\sigma_X)$ if $Y$ is a subscheme of $X$ and the embedding $i: Y\rightarrow X$ is a Poisson morphism.  
\end{defe}

Here we recall, from \cite[section 3]{13}, few examples of natural Poisson subschemes of a Poisson scheme $(X, \sigma)$.

\begin{exam} An open embedding is a Poisson subscheme in a unique way. A closed subscheme $Y$ of $X$ admits the structure of a Poisson subscheme if and only if $\{I_Y, \mathcal O_X\}\subset I_Y$. Note that the condition is necessary and sufficient for $\{\cdot, \cdot\}$ to descent to a Poisson bracket on $\mathcal O_Y=\mathcal O_X/I_Y$. In this case, the induced Poisson structure on $Y$ is unique. We denote it by $\sigma|_Y$ \cite[Proposition 2]{13}.
\end{exam}

\begin{exam}\label{Ex2}
The irreducible components of $X$ are Poisson subvarieties. Similarly, the singular locus of $X$ is a Poisson subscheme \cite[Lemma 3]{13}.
\end{exam}

\begin{defe}\cite[Definition 5]{13} Let $(X, \sigma)$ be a Poisson scheme. The degeneracy loci $D_{2k}(\sigma)$ of $\sigma$ is the locus where the morphism $\sigma^{\flat}: \Omega^1_X\rightarrow T_X$ has rank at most $2k$. It is the closed subscheme whose ideal sheaf is the image of the morphism 
\begin{equation}
\Omega^{2k+1}_X\xrightarrow{\sigma^{k+1}} \mathcal O_X
\end{equation}
where 
\begin{equation}
\sigma^{k+1}:=\underbrace{\sigma\wedge\cdots\wedge \sigma}_{k+1~~\text{times}}\in H^0(X, \mathcal X^{2k+2}_X).
\end{equation}
\end{defe}
\begin{exam}\label{exam12}
From \cite[Proposition 6]{13}, it follows that for $0\leq 2k\leq dim~~ X$, the degeneracy loci $D_{2k}(\sigma)$ are Poisson subschemes of $X$. Notice that $D_{2k}(\sigma)\setminus D_{2k-2}(\sigma)$ is a subscheme of $X$ consisting of points where the rank of the morphism $\sigma^{\flat}$ is exactly equal to $2k$. From \cite[Lemma 5]{13}, it follows that, If $(X,\sigma)$ is a Poisson scheme, and $Y$ is a Poisson subscheme, then $D_{2k}(\sigma)\cap Y=D_{2k}(\sigma|_Y)$. In particular, it implies that if $Y$ is a Poisson subscheme of $X$ and $y$ is any point of $Y$, then the Poisson rank of $\sigma$ at $y$ is the same as the Poisson rank of $\sigma|_Y$ at $y$. We get a natural stratification of $X$ by closed Poisson subschemes
\begin{equation}\label{Pstrat}
D_{dim ~~X}(\sigma):=X\supseteq D_{dim ~~X-2}(\sigma)\supseteq \cdots\supseteq D_{2k}(\sigma)\supseteq D_{2k-2}(\sigma)\supseteq \cdots \supseteq D_{0}(\sigma).
\end{equation}

We refer to it as the stratification by Poisson ranks.
\end{exam}

\subsection{On log-symplectic and relative log-symplectic structure}

\begin{defe} Let $S$ be a discrete valuation ring and $f: Y_S\rightarrow S$ be a scheme over $S$. Let us denote the closed fibre by $Y$. We call $Y_S$ a flat degeneration over $S$ if it satisfies the following conditions
\begin{enumerate}
\item $Y_S$ is regular as a scheme over $k$,
\item the generic fibre of $f: Y_S\rightarrow S$ is smooth, and
\item $Y$ is a normal crossing divisor in $Y_S$.
\end{enumerate}  
\end{defe}

Let $t$ be a uniformising parameter of $S$. Then the divisor $Y$ is the vanishing locus of the function $t\circ f$ on $Y_S$. 

\begin{defe}\cite[Definition 1.1, 1.2]{30.5}
The sheaf of differentials on $Y_S$ with logarithmic poles along $Y$ is defined by
\begin{equation}
\Omega_{Y_S}(log~~Y):=\{\text{meromorphic forms}~~ \omega~~\text{on}~~Y_S~~|~~t\cdot \omega ~~\text{and}~~t\cdot d\omega~~\text{are both regular differential forms} \}
\end{equation}
\end{defe}

Since any two uniformising parameter of $S$ differs by an unit, the definition does not depend on the choice of the uniformising parameter. By a local calculation \cite[Properties 2.2, (c)]{10} it follows that in our case $\Omega_{Y_S}(log~~Y)$ is a locally free sheaf. We call it the log-cotangent bundle. We call the dual of this vector bundle the log-tangent bundle and denote it by $T_{Y_S}(-log~~Y)$. 

A similar local calculation also shows that $Coker(f^*\Omega_S(\eta_0)\rightarrow \Omega_{Y_S}(log~~Y))$ is a vector bundle on $Y_S$. We call it the relative log-cotangent bundle and denote it by $\Omega_{Y_S/S}(log~~Y)$. We call the dual vector bundle the relative log-tangent bundle and denote it by $T_{Y_S/S}(-log~~Y)$. We call the restriction of the vector bundle $\Omega_{Y_S/S}(log~~Y)$ to $Y$ the log-cotangent bundle of $Y$ and denote it by $\Omega_{Y}(log~~\partial Y)$ and we call its dual the log tanegnt bundle of $Y$ and denote it by $T_{Y}(-log~~\partial Y)$. Here we denote by $\partial Y$ the singular locus of $Y$.

Let $q: X\rightarrow Y$ be the normalization. We denote by $\partial X$ the preimage $q^{-1}(\partial Y)$. It follows that $(X, \partial X)$ is a normal crossing divisor. 

\begin{lema}\label{LogExt}
$q^* \Omega_{Y}(log~~\partial Y)\cong \Omega_{X}(log~~\partial X)$.
\end{lema}
\begin{proof}
From \cite[Theorem 3.2]{11}, we have the following inclusion of sheaves
\begin{equation}
\Omega_{Y}(log~~\partial Y)\hookrightarrow q_*\Omega_{X}(log~~\partial X)
\end{equation}

It induces the following inclusion.

\begin{equation}
q^* \Omega_{Y}(log~~\partial Y)\hookrightarrow \Omega_{X}(log~~\partial X)
\end{equation}

The support of the cokernel is $\partial X$. Let $\partial^2 X:=\text{singular locus of}~~\partial X$. We claim that the morphism is an isomorphism over $X\setminus \partial^2 X$.

Assuming the claim, we see that the morphism of two vector bundles is isomorphic outside co-dimension $2$. Therefore the map must be an isomorphism.

The proof of the claim follows from the description \cite[Equation 3.1.1, (3.1.2)']{11}.

\end{proof}

\begin{defe}\label{RelLogSym}
A relative log-symplectic form on $Y_S$ is a relative non-degenerate $2$-form $\omega_S\in H^0(Y_S, \Omega^2_{Y_S/S}(log~~Y))$ such that $\omega$ is non-degenerate over $T_{Y_S/S}(-log~~Y)$ and $d\omega_S=0$, where $d$ is the relative exterior derivative. 
\end{defe}

\begin{thm}\label{log-sym1}
There is a natural relative log-symplectic structure on the relative log-cotangent bundle $\Omega_{_{Y_S/S}} ({log}~~Y)$.
\end{thm}
\begin{proof}
Let $\tilde f:\Omega_{Y_S/S} ({log}~~Y)\rightarrow Y_S$ denote the projection map. The vector bundle $\tilde f^*\Omega_{Y_S/S} ({log}~~Y)\cong \Omega_{Y_S/S} ({log}~~Y)\times_{Y_S} \Omega_{Y_S/S} ({log}~~Y)$ has a diagonal section $\lambda:\Omega_{Y_S/S} ({log}~~Y)\rightarrow \Omega_{Y_S/S} ({log}~~Y)\times_{Y_S} \Omega_{Y_S/S} ({log}~~Y)$. But $\tilde f^*\Omega_{Y_S/S} ({log}~~Y)$ is a sub-bundle of the log cotangent bundle of $\Omega_{Y_S/S} ({log}~~Y)$, where the polar divisor of $\Omega_{Y_S/S} ({log}~~Y)$ is the inverse image $\tilde f^{-1}(Y)$. Therefore we have a logarithmic $1$-form $\lambda$ over $\Omega_{Y_S/S} ({log}~~Y)$. We now define a two form $\omega:=-d\lambda$ by taking the exterior derivative. It is clearly a closed two form. By a local calculation, it follows that $\omega$ is non-degenerate on $T_{Y_S/S}(-log~~Y)$.
\end{proof}

\begin{defe}
 A log-symplectic form on $Y$ is a $2$-form $\omega\in H^0(Y, \Omega^2_{Y}(log~~\partial Y))$ such that $\omega$ is non-degenerate over $T_Y(-log~~\partial Y)$ and $d\omega=0$.
\end{defe}

Inverting $\omega$, we obtain a Poisson bivector 
\begin{equation}
\sigma\in H^0(Y, \mathcal X^2_Y(-log ~~\partial Y)).
\end{equation}

\begin{rema}Given a log symplectic form $\omega$ on $Y$ the pullback $\tilde \omega:=q^*\omega$ is a log symplectic form on the normal crossing divisor $(X, \partial X)$. 
\end{rema}

Let us now discuss a prototype example of a variety with a log-symplectic form.

\begin{exam}\label{cotex}
Consider the smooth variety $\mathbb C^k$ with coordinates $y_1,\dots, y_k$. Consider the normal crossing divisor given by the equation $y_1\cdots y_k=0$. Then the coordinates of the log-cotangent bundle are $\{y_1,\dots, y_k, p_1, \dots, p_k\}$, where $p_j:=y_j\partial_{y_j}$, for every $j=1,\dots, k$. Notice that the log cotangent bundle is a smooth variety isomorphic to $\mathbb C^{2k}$ which has a natural normal crossing divisor given by the equation $y_1\cdots y_k=0$. There is a tautological logarithmic one form (Liouville $1$-form) on the log-cotangent bundle which is given by 
\begin{equation}
\lambda:=\sum^k_{j=1} p_j\cdot \frac{dy_j}{y_j}
\end{equation}
The exterior derivative 
\begin{equation}
\omega:=-d\lambda=\sum^k_{j=1} dp_j\wedge \frac{dy_j}{y_j}
\end{equation}
is a logarithmic symplectic form. The corresponding Poisson bivector is 
\begin{equation}
\sigma:=\sum^k_{j=1} y_j\partial_{y_j}\wedge \partial_{p_j}
\end{equation} 
Now, for any skew-symmetric matrix $(B_{ij})\in \mathbb C^{k\times k}$, consider the $2$-form 
\begin{equation}
B=\sum_{1\leq i<j\leq k} B_{ij} \frac{dy_i}{y_i}\wedge\frac{dy_j}{y_j}
\end{equation}

We can define a new $2$-form 
\begin{equation}\label{Magnetic1}
\omega':=\omega+B
\end{equation}

The form $\omega'$ is again a log-symplectic form with the same polar divisor as $\omega$. The corresponding Poisson bivector
\begin{equation}\label{Magnetic2}
\sigma':=\sigma+\sum_{1\leq i<j\leq k} B_{ij} \partial_{p_i}\wedge \partial_{p_j}
\end{equation}

\end{exam}

\subsection{On the symplectic structure on the moduli of Higgs bundles on a curve}\label{SymCur} Let $X$ be a smooth projective curve. Let $\mathcal M_{VB} (\mathcal M_{HB})$ denote the moduli space of stable vector bundles (Higgs bundles) of rank $n$ and degree $d$, $(n,d)=1$. In this subsection, we recall few results from \cite{7} and \cite{8} about the symplectic structure on $\mathcal M_{HB}$. The following results will be used in the subsequent sections of this paper.

\subsubsection{\textbf{Symplectic form on a cotangent bundle}} Let $Z$ be a smooth variety. Let us denote by $f$ the projection map $\Omega_Z\rightarrow Z$. We have a natural morphism 
\begin{equation}\label{Lou}
\Omega_Z\xrightarrow{\Delta} \Omega_Z\times_Z \Omega_Z\cong f^*\Omega_Z\rightarrow \Omega_{\Omega_Z},
\end{equation}

where $\Delta$ denotes the diagonal map. The above section induces a $1-$form on the cotangent bundle $\Omega_Z$, known as the tautological $1$-form or the Liouville $1$-form. We denote it by $\lambda$. 

The negative of the exterior derivative i.e., $-d\lambda$ is a symplectic form on $\Omega_Z$. We denote it by $\omega$. Let $(z, w)$ be an element of $\Omega_Z$ over a point $z$. Let $v\in T_{\Omega_Z,(z,w)}$. Then by definition \eqref{Lou} we have 
\begin{equation}\label{lamb}
\lambda(v)=w(df(v)).
\end{equation} 
and for two elements $v_1,v_2\in T_{\Omega_Z,(z,w)}$, 
\begin{equation}
\omega(v_1,v_2)=-d\lambda(v_1, v_2).
\end{equation}

\subsubsection{\textbf{The tangent space and cotangent space of $\mathcal M_{VB}$ and $\mathcal M_{HB}$}} The tangent space of $\mathcal M_{VB}$ at a point $\mathcal E$ is naturally isomorphic to the space of first-order infinitesimal deformations of the vector bundle $\mathcal E$. It is well-known that the latter space is isomorphic to $H^1(X, \mathcal End \mathcal E)$. The cotangent space of $\mathcal M_{VB}$ at a point $\mathcal E$ is isomorphic to $H^1(X, \mathcal End \mathcal E)^{\vee}\cong Hom (\mathcal E, \mathcal E\otimes \Omega_X)$. It follows that $\Omega_{\mathcal M_{VB}}$ is an open subset of $\mathcal M_{HB}$ whose complement has codimension $2$. Therefore the natural Liouville form $\lambda$ and the symplectic form 
$\omega$ on $\Omega_{\mathcal M_{VB}}$ extends over $\mathcal M_{HB}$. We will now describe the forms $\lambda$ and $\omega$ on the tangent space of $\mathcal M_{HB}$.

Given a Higgs bundle $(\mathcal E, \phi)$, we denote by $\mathcal C_{\bullet}$ the following complex.

\begin{equation}
0\rightarrow \mathcal End \mathcal E\xrightarrow{[\bullet, \phi]} \mathcal End \mathcal E\otimes \Omega_X\rightarrow 0,
\end{equation}

where $[\bullet, \phi]$ is the morphism of $\mathcal O_X$-modules which maps $s\mapsto (s\otimes \mathbb 1)\circ \phi-\phi\circ s$.

We denote by $\mathcal C^{\vee}_{\bullet}$ the dual complex.

\begin{equation}
0\rightarrow \mathcal End \mathcal E\xrightarrow{[\phi, \bullet]} \mathcal End \mathcal E\otimes \Omega_X\rightarrow 0,
\end{equation}

where $[\phi, \bullet ]:=-[\bullet, \phi]$. 

Let $X=\cup_{i\in \Lambda} U_i$ be an open cover such that $\mathcal E$ and $\Omega_X$ are trivial over $U_i$ for each $i\in \Lambda$. For any $i,j,k\in \Lambda$ (all distinct), we set $U_{ij}:=U_i\cap U_j$ and $U_{ijk}:=U_i\cap U_j\cap U_k$. Let $\{A_{ij}\}_{i,j\in \Lambda}$ denote the transition functions of $\mathcal E$ with respect to the open cover $\{U_i\}_{i\in \Lambda}$. Then we have $A_{ij}A_{jk}=A_{ik}$. Let $\phi_i\in \Gamma(U_i, \mathcal End \mathcal E\otimes \Omega_X)$ denote the Higgs fields on each open sets in the open cover which glue to give the Higgs field $\phi$. In other words, $A_{ij}\circ \phi_j\circ A_{ij}^{-1}=\phi_i$ for all $i,j\in \Lambda$.

\begin{prop}\label{prop6565}
\begin{enumerate}
\item The tangent space of $\mathcal M_{HB}$ at $(\mathcal E, \phi)$ is isomorphic to $\mathbb H^1(X, \mathcal C_{\bullet})$,

\item The cotangent space is isomorphic to $\mathbb H^1(X, \mathcal C^{\vee}_{\bullet})$. 
\end{enumerate}
\end{prop}
\begin{proof}
By Kodaira-Spencer theory, the tangent space of $\mathcal M_{HB}$ at a point $(\mathcal E,\phi)$ is the space of infinitesimal deformations of the Higgs bundle. The elements of this space can be expressed as pairs $(s_{ij}, t_i)\in \Gamma(U_{ij}, \mathcal End \mathcal E)\times \Gamma(U_i, \mathcal End \mathcal E\otimes \Omega_X)$ such that 
\begin{enumerate}
\item $s_{ij}A_{jk} +A_{ij}s_{jk}=s_{ik}$,
\item $t_iA_{ij}-A_{ij}t_j=s_{ij}\phi_j-\phi_i s_{ij}$.
\end{enumerate}

Therefore from \cite[proof of theorem 2.3]{7} and \cite[Proposition 3.1.2]{8}, it follows that $(\{s_{ij}\}, \{t_i\})$ defines an element of $\mathbb H^1(\mathcal C_{\bullet})$ and the space of infinitesimal first-order deformations of the Higgs bundle is isomorphic to $\mathbb H^1(\mathcal C_{\bullet})$. 
  
Using duality of hypercohomologies, one can show that the cotangent space of $\mathcal M_{HB}$ is isomorphic to $\mathbb H^1(X, \mathcal C^{\vee}_{\bullet}(\mathcal E, \phi))$ whose elements can be expressed as pairs $(s_{ij}, t_i)\in \Gamma(U_{ij}, \mathcal End \mathcal E)\times \Gamma(U_i, \mathcal End \mathcal E\otimes \Omega_X)$ such that 
such that 
\begin{enumerate}
\item $s_{ij}A_{jk}+A_{ij}s_{jk}=s_{ik}$ as elements of $\Gamma(U_{ijk}, \mathcal End \mathcal E)$,
\item $t_iA_{ij}-A_{ij}t_j=-s_{ij}\phi_j+\phi_i s_{ij}$ as elements of $\Gamma(U_{ij}, \mathcal End \mathcal E\otimes \Omega_X)$.

\end{enumerate}
\end{proof}

\begin{rema}\label{co}
By a different interpretation, as in \cite[proof of Theorem 2.3]{7} and \cite[Proposition 3.1.2]{8}, the elements of the tangent space $\mathbb H^1(\mathcal C_{\bullet})$ can be equivalently described as pairs of co-cycles $(s_{ij}, t_i)$ satisfying the following condition
\begin{enumerate}
\item $s_{ij}+s_{jk}=s_{ik}$,
\item $t_i-t_j=s_{ij}\phi-\phi s_{ij}$.
\end{enumerate}
\end{rema}

Now, consider the short exact sequence of complexes
\begin{equation}
0\rightarrow \mathcal End \mathcal E\otimes \Omega_X[-1] \rightarrow \mathcal C_{\bullet}\rightarrow \mathcal End \mathcal E\rightarrow 0
\end{equation}

The following is the long exact sequence of hypercohomologies of the above short exact sequence.
\begin{equation}\label{LongEx}
0\rightarrow \mathbb H^0(\mathcal C_{\bullet})\rightarrow H^0(\mathcal End \mathcal E)\rightarrow H^0(\mathcal End \mathcal E\otimes \Omega_X)\rightarrow \mathbb H^1(\mathcal C_{\bullet})\rightarrow H^1(\mathcal End \mathcal E)\rightarrow H^1(\mathcal End \mathcal E\otimes \Omega_X)\rightarrow \mathbb H^2(\mathcal C_{\bullet})\rightarrow 0
\end{equation}

There is a natural forgetful morphism from the functor of Higgs bundles on $X$ to the functor of the vector bundles on $X$. Let us denote the morphism by $f$. In the long exact sequence \eqref{LongEx}, the map 

\begin{equation}\label{pro123}
\mathbb H^1(\mathcal C_{\bullet})\rightarrow H^1(\mathcal End \mathcal E)
\end{equation}

is the differential $df$ of the forgetful map $f$. 

\subsubsection{\textbf{Description of the symplectic potential}} \label{Liou440}Consider the morphism of vector spaces 

\begin{equation}\label{Liou1}
\lambda: \mathbb H^1(\mathcal C_{\bullet})\rightarrow \mathbb C
\end{equation}
given by 
$$v\mapsto \phi(df(v)).$$ 
where $df$ is the morphism $\eqref{pro123}$.

If $(s_{ij}, t_i)$ represents the tangent vector $v$ (Proposition \ref{prop6565} and remark \ref{co}), then 
\begin{equation}\label{eq109}
\phi(df(v))=Trace(\phi\circ s_{ij})
\end{equation} 

Notice that $\{Trace(\phi\circ s_{ij})\}_{i\in \Lambda}$ is an $1$-cocycle of $\Omega_X$ and hence an element of $H^1(X, \Omega_X)\cong \mathbb C$. The equality on the right in $\eqref{eq109}$ can be seen using the Trace paring on the sheaf $\mathcal End \mathcal E$. By definition \eqref{lamb}, $\lambda$ is the extension of the Liouville $1$-form $\Omega_{\mathcal M_{VB}}$.

\subsubsection{\textbf{Description of the symplectic form}} 
Consider the morphism of complexes $\mathcal C^{\vee}_{\bullet}\rightarrow \mathcal C_{\bullet}$.

\begin{equation}\label{Pair}
\begin{tikzcd}
\mathcal End \mathcal E \arrow{d}{[\phi, \bullet]}\arrow[" \mathbb 1"]{r} & \mathcal End \mathcal E\arrow{d}{-[\phi, \bullet]}\\
\mathcal End \mathcal E\otimes \Omega_X \arrow["-\mathbb 1\otimes \mathbb 1"]{r} & \mathcal End \mathcal E\otimes \Omega_X
\end{tikzcd}
\end{equation}

It induces a skew-symmetric pairing 
\begin{equation}
\mathbb H^1(\mathcal C^{\vee}_{\bullet})\rightarrow \mathbb H^1(\mathcal C_{\bullet})
\end{equation}

From the diagram \ref{Pair}, it follows that the above morphism can be expressed as 
\begin{equation}\label{omEga12}
(s_{ij}, t_i)\mapsto (s_{ij}, -t_i)
\end{equation} 
in terms of the co-cycle descriptions (Proposition \ref{prop6565} and remark \ref{co}). It induces a bilinear skew-symmetric pairing 

\begin{equation}
\omega: \mathbb H^1(\mathcal C_{\bullet})\times \mathbb H^1(\mathcal C_{\bullet})\rightarrow \mathbb C
\end{equation} 

Alternatively, the pairing can be described as the following composition
\begin{equation}\label{Pyaar1}
\mathbb H^1(\mathcal C_{\bullet})\times \mathbb H^1(\mathcal C_{\bullet})\rightarrow \mathbb H^2(\mathcal C_{\bullet}\otimes \mathcal C_{\bullet})\xrightarrow{Trace} \mathbb H^2(\Omega_X[-1])\cong H^1(X, \Omega_X)\cong \mathbb C
\end{equation}
given by 
\begin{equation}
((s_{ij},t_i),(s'_{ij},t'_i))\mapsto s_{ij}\otimes t'_j-t_i\otimes s'_{ij}\mapsto Trace(s_{ij}\circ t'_j-t_i\circ s'_{ij})
\end{equation}
Therefore, in terms of co-cycle, the pairing $\omega$ is given by 
\begin{equation}\label{Pyaar2}
\omega((s_{ij},t_i),(s'_{ij}, t'_i))=Trace(s_{ij}\circ t'_j-t_i\circ s'_{ij}).
\end{equation}

Using this descriptions, one can show that 
\begin{equation}
\omega=-d\lambda
\end{equation}
For further detail, we refer to \cite{7} and \cite{8}.

\begin{rema}\label{Liou441}
The results discussed in this subsection hold for the moduli of Higgs bundles (or more generally Hitchin pairs) on any curve (not necessarily smooth). This is because the arguments in the proof of \cite[Theorem 2.3]{7} does not require the curve to be smooth. In \S 6, we will show that the moduli of Higgs bundles on a fixed Gieseker curve has a natural symplectic potential and a symplectic form. Moreover, the symplectic potential and the corresponding symplectic form can be described in terms of the co-cycles similarly as in the case of a smooth curve.
\end{rema}
 

\section{\textbf{Functorial log structures on the moduli spaces}}

This section aims to define two natural logarithmic structures on the moduli space $\mathcal M_{GHB, S}$ and to show that they are isomorphic. We refer to \cite{18, 19} for basic definitions and results on log-geometry. 

\subsection{\textbf{Existence of an universal family}} To begin with, we show that there is a universal family over the moduli spaces $\mathcal M_{GVB,S}$ and $\mathcal M_{GHB, S}$. The proof is an easy adaptation of \cite[proof of Theorem 3.2.1]{31}. 

\begin{prop}\label{Universal}
Let $\mathcal X\rightarrow S$ be a family of curves with a section $\sigma$, as in \ref{DegeCourbes}. There exists a universal family of Gieseker vector bundles (Higgs bundles) on the moduli space $\mathcal M_{GVB, S}$ ($\mathcal M_{GHB, S}$). The varieties $\mathcal M_{GVB, S}$ and $\mathcal M_{GHB, S}$ are fine moduli spaces.
\end{prop}

\begin{proof}
First let us recall that the moduli space $\mathcal M_{GVB,S}$ is a GIT quotient of $\mathcal Y_S$ by the action of $PGL(N)$ (remark \ref{GIT420}). There is a universal curve $\Delta_S\subset \mathcal Y_S\times {Grass}(N,n)$ (remark \ref{Cur300}). More precisely, 
\begin{equation}
\Delta_S=\big\{(h, x)|~~h\in \mathcal Y_S, x\in \Delta_h\big\}.
\end{equation}
For an element $h\in \mathcal Y_S$, let $\Delta_h$ denote the fibre of $\Delta_S\rightarrow \mathcal Y_S$ over the point $h$. It is the image of the morphism $X_r\hookrightarrow {Grass}(N,n)$ corresponding to the element $h$. From this description of the universal curve, $\Delta_S$ it follows that it is stable under the action of $PGL(N)$. Consider the $PGL(N)$-equivariant polarisation 
\begin{equation}
\mathcal O_{\mathcal Y_S}(s)\otimes \mathcal O_{Grass(N,n)}(t) 
\end{equation}
over $\mathcal Y_S\times Grass(N,n)$, where $\mathcal O_{\mathcal Y_S}(1)$ and $\mathcal O_{Grass(N,n)}(1)$ are the natural polarizations on $\mathcal Y_S$ (remark \ref{GIT420}) and ${Grass}(N,n)$, respectively and $s/t$ is sufficiently large. We denote by $\mathcal O_{\Delta_S}(s/t)$ the restriction of this polarisation to $\Delta_S$. Using this polarisation, we construct the GIT quotient $\Delta^{ss}_{S}\parallelslant PGL(N)\rightarrow \mathcal Y^{st}_S\parallelslant PGL(N)$. Because of the assumption $g.c.d (\text{rank}, \text{deg})=1$, we have stable$=$semistable on $\mathcal Y_S$. Since $s/t$ is sufficiently large, therefore on $\Delta_S$ we also have stable $=$ semi-stable. In fact, the pre-image of $\mathcal Y^{st}_S$ under the morphism $\Delta_S\rightarrow \mathcal Y_S$ is precisely the set of semistable points in $\Delta_S$. We denote by $\mathcal X^{univ}_S$ the GIT quotient $\Delta^{st}_{S}\parallelslant PGL(N)$ and refer to it as the universal curve over $\mathcal M_{GVB,S}$. Since the action of $PGL(N)$ is free on $\mathcal Y^{st}_S$, from the description of the universal curve $\Delta_S$ it follows that the action of $PGL(N)$ is also free on $\Delta^{st}_S$. Therefore we have the following cartesian square
\begin{equation}
\begin{tikzcd}
\Delta^{st}_S\arrow{r}\arrow{d} & \mathcal Y^{st}_S\arrow{d}\\
\mathcal X^{univ}_S \arrow{r} & \mathcal M_{GVB,S}
\end{tikzcd}
\end{equation}
where the vertical morphisms are $PGL(N)$- principal bundles. 

Now let us discuss the descent of the universal vector bundle. Notice that there exists a universal bundle $U$ over $\Delta^{st}_S$, which is the pullback of the universal bundle over ${Grass}(N,n)$. Let us choose a line bundle over $\mathcal X$ such that the restriction of the line bundle on each fibre of $\mathcal X\rightarrow S$ is of degree one, e.g. the section $\sigma$ gives such a line bundle. Let us denote this line bundle by $\mathcal O_{\mathcal X/S}(1)$. With a choice of such a line bundle of relative degree one,  the rest of the proof follows from similar arguments from \cite[Lemma 5.11]{26}.

The relative moduli of Gieseker-Higgs bundles is a GIT quotient of $\mathcal Y^{H}_S$ by the action of $PGL(N)$ (remark \ref{GIT421}). Over $\mathcal Y^{H}_S$, we can pull back the family of curves $\Delta_S$ and the universal vector bundle $U$ by the forgetful morphism $\mathcal Y^{H}_S\rightarrow \mathcal Y_S$. Let us denote the curve by $\Delta^{H}_S$ and the vector bundle by $U'$. By similar arguments, we can show that the curve and the vector bundle descend to the moduli space $\mathcal M_{GHB, S}$. Notice that we have a tautological section $\phi$ of the vector bundle $\Gamma(\Delta^{H,st}_S, \mathcal End ~~U'\otimes p^*_{\mathcal X} \omega_{\mathcal X/S})$, where $p_{\mathcal X}$ denotes the composite morphism $\Delta_S\rightarrow \mathcal X\times_S \mathcal Y_S\rightarrow \mathcal X$. Also, $\phi$ is $PGL(N)$-equivariant. Therefore it descends to the GIT quotient.
\end{proof}

\begin{rema}\label{Formal1}
The functors $F^{st}_{GVB,S}$ and $F^{st}_{GHB,S}$ are represented by the varieties $\mathcal M_{GVB,S}$ and $\mathcal M_{GHB,S}$. We have natural forgetful morphisms
\begin{equation}
F^{st}_{GHB,S}\rightarrow F^{st}_{GVB,S}\rightarrow F_{GC,S}.
\end{equation}
The natural transformations $F_{GHB,S}\rightarrow F_{GC,S}$ and $F_{GVB,S}\rightarrow F_{GC,S}$ are formally smooth.  (\cite[Appendix: Local theory, I]{25} and \cite[Proposition 5.12]{3}).
\end{rema}

\subsection{\textbf{Log structures on $\mathcal M_{GVB,S}$ and $\mathcal M_{GHB,S}$}} Any normal crossing divisor of a smooth variety induces a natural log structure on the variety. Let us consider the logarithmic structure on the discrete valuation ring $S$ induced by the closed point $\eta_0$. For any log scheme $S$, we denote by $\mathcal Log_S$ the algebraic stack classifying fine log-structures on schemes over the log scheme $S$ (\cite[section 4]{27}). We have the following two natural log structures on $\mathcal M_{GVB,S}$.

\begin{enumerate}
\item{$(\mathcal M_{GVB,S}, \mathcal X^{univ}_S)$:} the curve $\mathcal X^{univ}_S\rightarrow \mathcal M_{GVB,S}$ is a prestable curve (remark \ref{prestable}). So it induces a log-structure on $\mathcal X^{univ}_S$ and $\mathcal M_{GVB,S}$ such that the projection morphism is a morphism of logarithmic schemes \cite[Global construction]{17}. By \cite[section 4]{27}, this log structure induces a morphism $f_{Cur}:\mathcal M_{GVB,S}\rightarrow \mathcal Log_{S}$.

\item{$(\mathcal M_{GVB,S}, \mathcal M_{GVB})$:} the normal crossing divisor $\mathcal M_{GVB}\subset \mathcal M_{GVB,S}$ induces a log structure on $\mathcal M_{GVB,S}$. Similarly, this also induces a morphism $f_{Div}: \mathcal M_{GVB,S}\rightarrow \mathcal Log_S$.
\end{enumerate}

Similarly, we have two log-structures on $\mathcal M_{GHB, S}$.

\begin{rema}
Notice that for the first log-structure to exist we need the universal curve, the existence of which is ensured by the assumption that the curve $\mathcal X\rightarrow S$ has a section $\sigma$, as in \ref{DegeCourbes} and Proposition \ref{Universal}. From now on in this section we will work with this assumption.
\end{rema}

\begin{prop}\label{logs}
The two log-structures on $\mathcal M_{GVB,S}$ ($\mathcal M_{GHB,S}$) are isomorphic.
\end{prop}
\begin{proof}
Let $(X_r, \mathcal E)$ be any $\spec k$-valued point of $\mathcal M_{GVB, S}$, where $X_r$ is a Gieseker curve with a chain of rational curves of length $r$ and $\mathcal E$ is a Gieseker vector bundle on $X_r$. Let us denote by $A$ the Henselian local ring of $\mathcal M_{GVB, S}$ at the point $(X_r, \mathcal E)$. Denote the maximal ideal by $m_{_A}$. We have the following diagram
\begin{equation}
\begin{tikzcd}
X_r\arrow[hook]{r}\arrow{d} & \mathcal X^{univ}_{A} \arrow{r} \arrow{d} & \mathcal X^{univ}_S\arrow{d}\\
\spec \frac{A}{m_A}\arrow[hook]{r} & \spec A \arrow{r} & \mathcal M_{GVB,S} 
\end{tikzcd}
\end{equation}

Both the squares are cartesian. Let $D$ be the closed subscheme of $\mathcal X^{univ}_{S}$ defined by the first Fitting ideal ${Fitt}^1 (\Omega_{\mathcal X^{univ}_S/\spec S})$.

We claim the following:

\begin{enumerate}
\item $V({Fitt}^1 \Omega_{\mathcal X^{univ}_A/\spec A})=D_A=\coprod_{i=1}^{r+1} D_{A, i}$, where $D_A:=D\times_{\mathcal M_{GVB, S}} \spec A$ and $D_{A, i}$ are the connected components of $D_A$.
\item around $D_i$,
\begin{equation}
R_i:=\mathcal O^h_{\mathcal X^{univ}_A,D_i}\cong \frac{A[x,y]}{xy-t_i}, ~~\text{for some}~~ t_i\in m_{A},
\end{equation}
\item Set $A_0:=A\otimes \frac{\mathcal O_S}{m_S}$, where $m_S$ denotes the maximal ideal of $\mathcal O_S$. Then $A_0=A/(t_1\cdots t_{r+1})$, i.e., in the Henselian local ring, the normal crossing divisor is the vanishing locus of $(t_1\cdots t_{r+1})$.
\end{enumerate}

The first claim follows from the functoriality of the construction of the Fitting ideals. Second claim follows from the definition of a family of pre-stable curves. To prove the third claim note that the map
\begin{equation}
A\rightarrow \frac{A[x,y]}{xy-t_i}
\end{equation}
is not smooth only over $V(t_i)$. Therefore the map $\mathcal X^{univ}_A\rightarrow \spec A$ is not smooth exactly over $V(t_1\cdots t_{r+1})$.

Let us denote the point $(X_r, \mathcal E)$ of $\mathcal M_{GVB,S}$ by $p$. For simplicity, let us denote $\mathcal M_{GVB,S}$ by $\mathcal M_S$. So $A={\mathcal O}^h_{\mathcal M_{S},p}$. We have the following inclusions
\begin{equation}
\mathcal O_{S,s}\rightarrow \mathcal O_{\mathcal M_S,p}\rightarrow \mathcal O^{h}_{\mathcal M_S,p}\rightarrow \hat{\mathcal O}_{\mathcal M_S,p}
\end{equation}

For our purpose, we can assume that $S=\spec k[|t|]$. The functor of Artin rings $F_{GC,S}$ has a versal deformation space given by $W:=\spec k[|z_1,\dots, z_{r+1}|]$ and a versal family of Gieseker curves $B$ over $W$ (\cite[Lemma 4.2]{12} and \cite["Appendix: Local Theory ", II. (a),(b), (c),(d)]{25}). There is a morphism $W\rightarrow S$ given by $t\mapsto z_1\cdots z_{r+1}$ and the fiber over $t=0$ is the versal space for the absolute functor $F_{GC}$. From the construction of $B$ and $W$, it follows that the fiber over $t=0$ is precisely the locus in $W$ over which the morphism $B\rightarrow W$ is not smooth.

Now the restriction of the universal modification $\mathcal X^{univ}_S$ is a modification on $\spec \hat{\mathcal O}_{\mathcal M_S,p}$. By the versality property, there exists a formally smooth morphism $v:\spec \hat{\mathcal O}_{\mathcal M_S,p}\rightarrow W$ such that $\mathcal X^{univ}_S\cong v^*B$. It follows from \cite[Proposition 4.5]{12} that $t\cdot \hat{\mathcal O}_{\mathcal M_S,p}=(t_1\cdots t_{r+1})\hat{\mathcal O}_{\mathcal M_S,p}$.
Since ${\mathcal O}_{\mathcal M_S,p}^{h}\rightarrow \hat{\mathcal O}_{\mathcal M_S,p}$ is faithfully flat, therefore $t\cdot {\mathcal O}_{\mathcal M_S,p}^h=(t_1\cdots t_{r+1}){\mathcal O}_{\mathcal M_S,p}^h$. This proves the claim $(3)$. Since the local equations of the divisor coincide with the local equation of the nodes of the universal curve, the two log structures are isomorphic. The proof for $\mathcal M_{GHB,S}$ is similar.

\end{proof}

\begin{rema}
Let us consider the logarithmic structure on the discrete valuation ring $S$ induced by the closed point $\eta_0$. Since the morphisms $\mathcal M_{GVB, S}\rightarrow S$ and $\mathcal M_{GHB, S}\rightarrow S$ are semistable degenerations; therefore, they are log-smooth morphisms. In other words, the induced morphisms $\mathcal M_{GVB, S}\rightarrow \mathcal Log_S$ and $\mathcal M_{GHB, S}\rightarrow \mathcal Log_S$ are smooth morphisms of algebraic stacks \cite[3.7 (Verification of (1.1 (iii))).]{28}.
\end{rema}

\section{\textbf{Relative Log-tangent space}}

\subsubsection{\textbf{Relative Log-tangent space}}
In this subsection we want to compute the relative tangent space and relative log-tangent space of $\mathcal M_{GVB,S}\rightarrow S$ and $\mathcal M_{GHB,S}\rightarrow S$ using log-deformation theory \cite{18},\cite{27}. 

\begin{rema}
For this purpose, it is enough to concentrate on the special fibres $\mathcal M_{GVB}\rightarrow \spec k$ and $\mathcal M_{GHB}\rightarrow \spec k$ instead of the relative case over $S$. Therefore, we see that to compute the relative log-tangent space we can replace $S$ by a suitable etale neighbourhood of the closed point of $S$ as in \ref{DegeCourbes}.
\end{rema}

Let $\pi_r: X_r\rightarrow X_0$ be a $\spec k$-valued point of $F_{GC}$. From \cite[ " Global Construction", Proposition 2.1]{17}, it follows that there are canonical induced log-structures on $X_r$ and $X_0$ and as well as on $\spec k$ such that the arrows $X_r\rightarrow \spec k$ and $X_0\rightarrow \spec k$ are log-smooth. It is straightforward to check that the log structure on $\spec k$ induced by the curve $X_r\rightarrow \spec k$ is the same as the log-structure induced by the following pre-log structure.

\begin{equation}\label{Log1001}
\oplus_{i=1}^{i=r+1} \overline{\mathbb N}\rightarrow k  
\end{equation}
which sends
\begin{center}
$e_i\mapsto 0$,
\end{center}
where $e_i$ is the $i$-th basis element of $\oplus_{i=1}^{i=r+1} \overline{\mathbb N}$.

\begin{lema}\label{Log101}
Let $\mathcal X$ be a family of Gieseker curves over $\spec k[\epsilon]$ such that the fiber over the closed subscheme $\spec k$ is the Gieseker curve $\pi_r: X_r\rightarrow X_0$. If the induced log structure (\cite[Proposition 2.1]{17}) on $\spec k[\epsilon]$ is isomorphic to the pull back of the log structure of $\spec k$, defined above in \eqref{Log1001} under the natural projection map $\spec k[\epsilon]\rightarrow \spec k$, then the deformation is trivial i.e., $\mathcal X\cong X_r\times \spec k[\epsilon]$.
\end{lema}
\begin{proof}
The vanishing locus of the first Fitting ideal of the relative cotangent sheaf of the morphism $\mathcal X\rightarrow \spec k[\epsilon]$ has $r+1$ components. Let us denote them by $D_1, \dots, D_{r+1}$. Let us denote by $U^{-}_i$ the complement of the node $D_i$ and by $U^{+}_i$ the complement of the closed subset $\coprod_{_{j\neq i}} D_j$. Then $\{U^{+}_i, U^{-}_i\}$ is a Zariski-open covering of $\mathcal X$. Etale locally, around $D_i$ we have
\begin{equation}\label{logform}
\mathcal O_{U^+_i}\cong \frac{k[x_i, y_i, \epsilon]}{x_iy_i-\lambda_i\epsilon}. 
\end{equation}
We attach a log structure defined by a pre-log structure $\overline{\mathbb N}^2\rightarrow \frac{k[x_i, y_i, \epsilon]}{x_iy_i-\lambda_i\epsilon}$ which sends $(1,0)$ to $x_i$ and $(0,1)$ to $y_i$. On $U^{-}_i$ we consider the log structure defined by the pre-log structure $\overline{\mathbb N}\rightarrow \mathcal O_{_{U^{-}_i}}$ which maps $1\rightarrow \lambda_i\epsilon$; these log structures can be glued along the intersection by using the diagonal homomorphism $\overline{\mathbb N}\rightarrow \overline{\mathbb N}^2$. We denote the resulting log structure by $\mathcal M_i$. The induced log structure on $\spec k[\epsilon]$ is the log structure defined by the pre-log structure $\alpha_i:\overline{\mathbb N}\rightarrow k[\epsilon]$ given by $1\rightarrow \lambda_i\epsilon$. Let us denote the log structure by $\mathcal L_i$. Then we see that $\mathcal L_i\cong \overline{\mathbb N}\otimes_{\alpha_i^{-1}(k[\epsilon]^*)} k[\epsilon]^*\cong \overline{\mathbb N}\oplus k[\epsilon]^*$.

Finally the induced log structure on $\spec k[\epsilon]$ is the amulgumated sum
\begin{equation}
\mathcal L_{k[\epsilon]}:=\mathcal L_1\otimes_{k[\epsilon]^*} \cdots \otimes_{k[\epsilon]^*} \mathcal L_{r+1}
\end{equation}

It is isomorphic to the log structure associated with the prelog structure
\begin{equation}
\oplus_{i=1}^{r+1} \overline{\mathbb N}\rightarrow k[\epsilon]
\end{equation}
given by
\begin{center}
$e_i:=(0,\dots, \underbrace{1}_{i-\text{th position}} ,\dots,0)\mapsto \lambda_i\epsilon$
\end{center}
Therefore it is isomorphic to the pull back of the log structure on $\spec k$, defined above in eq.\ref{Log1001} under the natural projection map $\spec k[\epsilon]\rightarrow \spec k$ if and only if $\lambda_i=0$ for all $i=1,\dots,r+1$.

The space of infinitesimal deformations of the nodal curve $X_r$ is isomorphic to $Ext^1(\Omega_{X_r}, \mathcal O_{X_r})$. Using Local-to-global spectral sequence \cite[eq. 1.2, page 169]{1}, we get

\begin{equation}
0\rightarrow H^1(X_r, \mathcal Hom(\Omega_{X_r}, \mathcal O_{X_r})) \rightarrow Ext^1(\Omega_{X_r}, \mathcal O_{X_r})\rightarrow H^0(X_r, \mathcal Ext^1(\Omega_{X_r}, \mathcal O_{X_r}))\cong \oplus_{i=1}^{i=r+1} Ext^1(\Omega_{X_r,p_i}, \mathcal O_{X_r,p_i})\rightarrow 0
\end{equation}

Since $\lambda_i=0$ for all $i=1,\dots,r+1$, from \eqref{logform}, it follows that the infinitesimal deformation $\mathcal X$ is an element of $H^1(X_r, \mathcal Hom(\Omega_{X_r}, \mathcal O_{X_r}))$. Now from \cite[Corollary 4.4]{12} we have the following inclusion
\begin{equation}
H^1(X_r, \mathcal Hom(\Omega_{X_r}, \mathcal O_{X_r}))\hookrightarrow H^1(X_0, \mathcal Hom(\Omega_{X_0}, \mathcal O_{X_0}))
\end{equation}
Moreover, the image of any $1$-cocycle/infinitesimal first order deformation $\mathcal X'$ of $X_r$ under this inclusion is the obstruction to extending the map $X_r\rightarrow X_0$ to a map $\mathcal X'\rightarrow X_0\times \spec k[\epsilon]$. But since the deformation $\mathcal X$, by definition comes with a morphism to $X_0\times \spec k[\epsilon]$, therefore the image under the inclusion is $0$. Therefore $\mathcal X\cong X_r\times \spec k[\epsilon]$.
\end{proof}

\begin{rema}\cite[Example 2.5,(2)]{19}\label{Log102}
Since our base field $k$ is algebraically closed, there is a bijection between the following two sets:
\begin{equation}
\left\{
\begin{array}{@{}ll@{}}
\text{Isomorphism classes of integral}\\
\text{ log-structures on}~~\spec k
\end{array}\right\}
\leftrightarrow\left\{
\begin{array}{@{}ll@{}}
\text{Isomorphism classes of integral monoids}\\
\text{having no invertible elements other than }~~0
\end{array}\right\}
\end{equation}
given by the following:

\begin{enumerate}
\item Given an integral monoid $P$ such that $P^{*}=\{0\}$, the corresponding log structure is $k^*\oplus P$ with
\begin{equation}
k^*\oplus P\rightarrow k
\end{equation}
$$
(\lambda, p)\mapsto
\begin{cases}
\lambda ~~\text{if}~~p=0,\\
0~~\text{otherwise}
\end{cases}
$$
\item Given an integral log-structure $\alpha:\mathcal P\rightarrow k$, the correspoding integral monoid is $P:=\mathcal P/\alpha^{-1}(k^*)$. Consider the zero map
\begin{equation}
P\rightarrow k
\end{equation}
Then the associated log structure $P\oplus k^*\cong \mathcal P$.
\end{enumerate}
\end{rema}
\begin{lema}\label{Log103}
Let $Q_{_0}$ be an integral monoid with $0$ as its only unit. Let $\mathcal X\rightarrow \spec k[\epsilon]$ be a family such that the fiber over $\spec k$ is isomorphic to $X_r$. The log structure on $\spec k[\epsilon]$ induced by the family $\mathcal X$ is isomorphic to the log structure associated with the pre-log structure
\begin{equation}
\beta: Q_{_0}\rightarrow k[\epsilon]
\end{equation}
$$
q\mapsto
\begin{cases}
1 ~~\text{if}~~q=0,\\
0~~\text{otherwise}
\end{cases}
$$
if and only if $Q_{_0}\cong \overline{\mathbb N}^{r+1}$ and the pre-log structure
\begin{equation}
\beta: \overline{\mathbb N}^{r+1}\rightarrow k[\epsilon]
\end{equation}
is given by
$$
e_i\mapsto 0 ~~\hspace{3cm}~~ \text{for all} ~~i=1,\dots, r+1
$$
\end{lema}
\begin{proof}
We leave it to the reader.
\end{proof}

Consider the following morphism of pre-log structures
\begin{equation}
\begin{tikzcd}
& e_i\arrow{r} & 0 ~~\text{for all}~~i=1,\cdots, r+1\\
e_1+..+e_{r+1} & \overline{\mathbb N}^{r+1}\arrow{r} & k[\epsilon]\\
e\arrow{u} & \overline{\mathbb N} \arrow{u}\arrow{r} & k\arrow{u}\\
& e\arrow{r} & 0
\end{tikzcd}
\end{equation}

Let us denote the log structure induced on $\spec k[\epsilon]$ by $M$. We denote the log scheme $(\spec k[\epsilon], M)$ by $X$, the log scheme $(\spec k, \overline{\mathbb N})$ by $Y$ and the log morphism above by $f:X\rightarrow Y$.
\begin{lema}\label{Auto101}
The group of the automorphisms $\phi$ of the log structure $M$ on $\spec k[\epsilon]$ satisfying the following conditions
\begin{enumerate}
\item the automorphism of the underlying scheme $\spec k[\epsilon]$ is identity,
\item the restriction $\phi_0$ of $\phi$ on the closed log-subscheme $(\spec k, j^*M)$ is the identity morphism, where $j$ is the closed immersion $\spec k\hookrightarrow \spec k[\epsilon]$,
\item the automorphism $\phi$ commutes with the log morphism $f$
\end{enumerate}
has the structure of a $k$-vector space of dimension $r$. We denote this group by ${Aut}^{^{{Inf}}}_{Y}(M)$.
\end{lema}
\begin{proof}
Notice that the log structure $M$ is isomorphic to $\overline{\mathbb N}^{r+1}\oplus k[\epsilon]^*$, where the monoid product is given by addition in the first component and multiplication in the second component. The monoid morphism $M\rightarrow k[\epsilon]$ is given by
$(0, a\epsilon+b)\mapsto (0, a\epsilon+b)$ and $(e_i, 1)\mapsto 0$ for all $i=1,\cdots, r+1$.
Similarly, the log structure on $\spec k$ associated to the prelog structure $\overline{\mathbb N}\xrightarrow{0} k$ is isomorphic to $\overline{\mathbb N}\oplus k^*$ and the monoid morphism is the morphism $\overline{\mathbb N}\oplus k^*\rightarrow k$ which sends $(0, \lambda)\mapsto \lambda$ and $(e, 1)\mapsto 0$.

Since the restriction $\phi_0$ of the automorphism

$$\overline{\mathbb N}^{r+1}\oplus k[\epsilon]^*\xrightarrow{\phi}\overline{\mathbb N}^{r+1}\oplus k[\epsilon]^*$$

to the reduced log subsceme $(\text{Spec}~~k, \overline{\mathbb N}^{r+1}\oplus k^*)$ is the identity morphism. Therefore $\phi_0(e_i,1)=(e_i,1)$ for $i=1,\cdots, r+1$. Therefore the first factor of $\phi(e_i, a+b\epsilon)$ also must be $e_i$ for every $i=1,\cdots,r+1$.

Since $\phi$ is a monoid isomorphism therefore $\phi((0,1))=(0,1)$, because $(0,1)$ is the identity element in this monoid. Therefore $\phi(0, a+b\epsilon)=a+b\epsilon$.

Now notice that the monoid $\overline{\mathbb N}^{r+1}\oplus (k[\epsilon])^*$ is generated by the elements of the following forms $\{(e_i, 1)|~~i=1,\cdots, r+1\}$ and $\{(0, a+b\epsilon)|~~a+b\epsilon\in k[\epsilon]^*\}$. Therefore the images of these generators under $\phi$ determine the automorphism $\phi$. Since the first factor of $\phi(e_i, 1)$ must be $e_i$, therefore $\phi(e_i, 1)=(e_i, c_i\cdot (1+\lambda_i\epsilon))$ for some $c_i\in k^*$ and $\lambda_i\in k$. But since $\phi_0$ is the identity morphism therefore we see that $c_i=1$ for all $i=1,\cdots,r+1$. Therefore $\phi(e_i, 1)=(e_i, 1+\lambda_i\epsilon)$, for some $\lambda_i\in k$.

Since the isomorphism $\phi$ commutes with the log morphism $f$, we must have $\phi(e_1+\cdots+e_{r+1}, 1)=(e_1\cdots+e_{r+1}, 1)$. Since $\phi$ is automorphism of a monoid we have $\phi(e_1+\cdots+e_{r+1}, 1)=(e_1+\cdots+e_{r+1}, (1+\lambda_1\epsilon)\cdots(1+\lambda_{r+1}\epsilon))=(e_1+\cdots+e_{r+1}, 1+(\lambda_1+\cdots+\lambda_{r+1})\epsilon)$. Therefore the sum $\lambda_1+\cdots+\lambda_{r+1}=0$.

Therefore the group of such automorphism is isomorphic to the underlying additive group of the vector space $k^{r}$. If $\phi\in Aut^{{Inf}}_Y(M)$ such that $\phi(e_1,1)=(e_1, 1+\lambda_i\epsilon )$ and $\lambda \in k$ any scalar then we define $(\lambda\star \phi) (e_i,1):=(e_i, 1+\lambda\cdot \lambda_i\cdot \epsilon)$ for all $i=1,\cdots , r+1$. Therefore the group has the structure of a vector space and is isomorphic to $k^{r}$.
\end{proof}

The following two lemmas follow from \cite[Lemma 4.6, 4.8]{12}. Nevertheless, we include slightly different proofs, which are more suitable for our purpose.

\begin{lema} \label{Equ1011}
\begin{enumerate}
\item Let $\mathcal E$ be a stable Gieseker vector bundle on the curve $X_r$ and $\psi\in {Aut}_{X_0}(X_r)$. Then $\mathcal E\ncong \psi^*\mathcal E$.
\item Let $(\mathcal E,\phi)$ be a stable Gieseker-Higgs bundle on the curve $X_r$ and $\psi\in {Aut}_{X_0}(X_r)$. Then $\mathcal E\ncong \psi^*\mathcal E$.
\end{enumerate}
\end{lema}
\begin{proof}
\textsf{proof of (1)} Suppose that there exists a Gieseker vector bundle $\mathcal E$ on $X_r$and $\psi \in {Aut}_{X_0}(X_r)$ such that $\psi^*\mathcal E\cong \mathcal E$. Then there exists an automorphism $\tilde{\psi}: \mathcal E\rightarrow \mathcal E$ such that the following diagram commutes
\begin{equation}
\begin{tikzcd}
\mathcal E\arrow{r}{\tilde{\psi}}\arrow{d} &\mathcal E\arrow{d}\\
X_r\arrow{r}{\psi} & X_r
\end{tikzcd}
\end{equation}
Notice that $X_r=\tilde{X_0}\cup R$, where $\tilde{X_0}\rightarrow X_0$ is the normalization, $R$ is the rational chain of length $r$ and $\tilde{X_0}\cap R=\{p_1,p_{r+1}\}$. Consider the push-forward $\pi_*\mathcal E\xrightarrow{\pi_*\tilde{\psi}} \pi_*\mathcal E$. Since $\pi_*\mathcal E$ is a stable torsion-free sheaf the morphism $\pi_*\tilde{\psi}=\lambda\cdot \text{Identity},$ where $\lambda$ is a non-zero scalar.

Restricting the above diagram on $R$ we get
\begin{equation}
\begin{tikzcd}
\mathcal E|_{R}\arrow{r}{\tilde{\psi}}\arrow{d} &\mathcal E|_{R}\arrow{d}\\
R\arrow{r}{\psi} & R
\end{tikzcd}
\end{equation}
The restriction of $\tilde{\psi}$ at the two points are $\lambda\cdot\text{Identity}$. Now notice that the vector bundle $\mathcal E|_{R}$ is globally generated. Given a global section $\sigma\in \Gamma(R, \mathcal E|_{R})$, we get a new section $\tilde{\psi}^{-1}\circ \sigma\circ \psi$. Notice that $(\tilde{\psi}^{-1}\circ \sigma\circ \psi-\lambda\cdot \sigma)(p_i)=0$ for $i=1$ and $i=r+1$. Since $\mathcal E|_{R}$ is strictly standard (definition \ref{SS440}), we conclude that $\tilde{\psi}^{-1}\circ \sigma\circ \psi=\lambda\cdot \sigma$. Therefore the induced morphism $H^0(R, \mathcal E|_{R})\rightarrow H^0(R, \mathcal E|_{R})$ is multiplication by $\lambda^{-1}$. Since $\mathcal E|_{R}$ is globally generated therefore the morphism $\tilde{\psi}=\lambda^{-1}\cdot \text{Identity}$.

This is not possible because $\mathcal E|_{R}$ is strictly standard. To see this first notice that we can decompose $\mathcal E|_{R}\cong L_1\oplus \cdots\oplus L_n$, such that $L_i|_{R_i}\cong \mathcal O(1)$ and $L_i|_{R_j}\cong \mathcal O$ for all $i\neq j$. For $i\neq j$, the induced morphism $L_i\xrightarrow{\tilde{\psi}} L_j$ cannot be multiplication by a scalar becuase there is no $\psi-$equivariant homomorphism from $\mathcal O(1)\rightarrow \mathcal O$. Therefore the induced morphism $L_i\rightarrow L_j$ is $0$ for any $i\neq j$. Therefore $\tilde{\psi}(L_i)\cong L_i$ for all $i=1,\cdots, n$.

Since $\psi$ is a nontrivial automorphism of the curve $X_r$, it is nontrivial on at least one rational curve in the chain $R$. Without loss of generality, let us assume that $\psi|_{R_i}\neq 1$. Moreover, let $\psi|_{R_i}$ is given by the multiplication by a scalar $\mu\notin \{1,-1\}$ i.e., ${\psi}([x:y])=[\mu\cdot x: \frac{1}{\mu}y]$. Notice that the morphism $L_i\xrightarrow{\tilde{\psi}} L_i$ given by the multiplication by a scalar $\lambda^{-1}$ cannot be $\psi$-equivariant for any scalar $\lambda^{-1}$. To see this let us focus on the rational curve $R_i$. We have $L_i|_{R_i}\cong \mathcal O(1)$ and therefore $L^*_i|_{R_i}\cong \mathcal O(-1)$. We have the following commutative square
\begin{equation}
\begin{tikzcd}
L^*\arrow{d}\arrow{r}{\tilde{\psi}^*} & L^*\arrow{d}\\
R_i\arrow{r}{\psi^{-1}} & R_i
\end{tikzcd}
\end{equation}
Notice that the map $\tilde{\psi}^*$ is multiplication by the scalar $\lambda$.

The total space of $\mathcal O(-1)$ is the following subvariety of $R_i\times \mathbb C^2$

\begin{equation}
\big\{([x:y], (\gamma\cdot x, \gamma \cdot y)|~~\gamma\in \mathbb C\big\}
\end{equation}

Consider the diagram
\begin{equation}
\begin{tikzcd}
([x:y], (\gamma\cdot x, \gamma \cdot y))\arrow{rr}&& ([ \frac{1}{\mu}\cdot x:\mu\cdot y],(\lambda\cdot \gamma\cdot x, \lambda \cdot \gamma \cdot y)), ~~\forall \gamma\in \mathbb C\\
\mathcal O(-1)\arrow{rr}\arrow{d} && \mathcal O(-1)\arrow{d}\\
R_i\arrow{rr} && R_i \\
{[x:y]}\arrow{rr} && {[ \frac{1}{\mu}\cdot x :\mu \cdot y]}
\end{tikzcd}
\end{equation}

But if $\mathcal O(-1)$ has to be equivariant then $\lambda\cdot \gamma=\frac{1}{\mu} \cdot \nu$ and $\lambda\cdot \gamma=\mu\cdot \nu$ for some $\nu\in \mathbb C$. But this implies that $\mu^2\cdot \nu=\nu$. Since $\lambda\neq 0$ the scalar $\nu \neq 0$. Therefore $\mu^2=1$ i.e, $\mu\in \{1,-1\}$. But the multiplications by $\pm 1$ induce the identity morphism on $R_i$, which is a contradiction.

\textsf{proof of (2)} The proof of the second statement follows similarly using the fact that induced torsion-free Higgs pair $(\pi_*\mathcal E, \pi_*\phi)$ is stable and therefore the automorphism $\pi_*\tilde \psi =\lambda\cdot \text{Identity},$ where $\lambda$ is a non-zero scalar.
\end{proof}

The vector space $H^0(X_r, T_{X_r})$ parametrises the automorphisms of the variety $X_r\times \spec k[\epsilon]$, which commute with the projection to $\spec k[\epsilon]$ and whose restriction on the closed fiber is the identity morphism. Let us denote $\spec k[\epsilon]$ by $\mathbb D$. Now notice given an infinitesimal automorphism $X_r\times \mathbb D\xrightarrow{\psi} X_r\times \mathbb D$ and a vector bundle $\mathcal E_{\mathbb D}$ over $X_r\times \mathbb D$ such that the restriction to the closed fiber is a Gieseker vector bundle $\mathcal E$, we can pull back the vector bundle $\mathcal E_{\mathbb D}$ by the morphism $\psi$. We define the following action
\begin{equation}
H^0(X_r, T_{X_r})\times H^1(X_r, \mathcal End \mathcal E)\rightarrow H^1(X_r, \mathcal End \mathcal E)
\end{equation}
given by $(\psi, \mathcal E_{\mathbb D})\mapsto \psi^*\mathcal E_{\mathbb D}$.

Similarly, given an infinitesimal automorphism $X_r\times \spec k[\epsilon]\xrightarrow{\psi} X_r\times \spec k[\epsilon]$ and a Higgs bundle $(\mathcal E_{\mathbb D},\phi_{\mathbb D})$ over $X_r\times \spec k[\epsilon]$ such that the restriction to the closed fiber is the Higgs bundle $(\mathcal E, \phi)$, we can pullback the Higgs field $\phi_{\mathbb D}$ by the morphism $\psi$. We define the following action
\begin{equation}
H^0(X_r, T_{X_r})\times \mathbb H^1(X_r, \mathcal C_{\bullet})\rightarrow \mathbb H^1(X_r, \mathcal C_{\bullet})
\end{equation}
given by $(\psi, \mathcal E_{\mathbb D}, \phi_{\mathbb D})\mapsto (\psi^*\mathcal E_{\mathbb D}, \psi^*\phi_{\mathbb D})$.

\begin{lema}\label{InfAct}\label{Tesimal}
\begin{enumerate}
\item Let $\mathcal E$ be a stable Gieseker vector bundle on $X_r$. The action of the group $H^0(X_r, T_{X_r})$ of infinitesimal automorphisms of $X_r$ on the space $H^1(X_r, \mathcal End \mathcal E)$ of all first order infinitesimal deformations of the vector bundle $\mathcal E$ is free.

\item Let $(\mathcal E, \phi)$ be a stable Gieseker-Higgs bundle on $X_r$. The action of the group $H^0(X_r, T_{X_r})$ of infinitesimal automorphisms of $X_r$ on the space $\mathbb H^1(X_r, \mathcal C_{\bullet})$ of all first order infinitesimal deformations of the Higgs bundle $(\mathcal E, \phi)$ is free.
\end{enumerate}
\end{lema}
\begin{proof}
\textsf{proof of (1)} Notice that if there exists $\psi\in H^0(X_r, T_{X_r})$ and $\mathcal E_{\mathbb D}\in H^1(X_r, \mathcal End \mathcal E)$ such that there exists an isomorphism $\tilde{\psi}: \mathcal E_{\mathbb D}\rightarrow \psi^*\mathcal E_{\mathbb D}$, then we have a following cartesian square
\begin{equation}
\begin{tikzcd}
\mathcal E_{\mathbb D}\arrow{rr}{\tilde{\psi}}\arrow{d}&& \mathcal E_{\mathbb D}\arrow{d}\\
X_r\times \spec k[\epsilon] \arrow{rr}{\psi} && X_r\times \spec k[\epsilon]
\end{tikzcd}
\end{equation}
In other words, $\mathcal E_{\mathbb D}$ is a $\psi$-equivariant bundle on $X_r\times \spec k[\epsilon]$. Let us denote by $\pi$ the morphism $X_r\rightarrow X_0$. Then we have an induced automorphism $X_0\times \spec k[\epsilon]\xrightarrow{(\pi\times \mathbb 1)_*\psi} X_0\times \spec k[\epsilon]$ such that it commutes with the projection to $\spec k[\epsilon]$ and the induced automorphism on the special fiber is the identity. But since $X_0$ is a stable curve therefore $(\pi\times \mathbb 1)_*\tilde \psi=\text{Identity}$. We also have an induced automorphism of the torsion-free sheaf $(\pi\times \mathbb 1)_*\tilde{\psi}: (\pi\times \mathbb 1)_*\mathcal E_{\mathbb D}\xrightarrow{} (\pi\times \mathbb 1)_*\mathcal E_{\mathbb D}$ such that induced morphism on the closed fiber $\pi_*\mathcal E\rightarrow \pi_*\mathcal E$ is the identity. Since the morphism $(\pi\times \mathbb 1)_*\tilde{\psi}$ is an $\mathcal O_{X_0}[\epsilon]$ module homomorphism and it is the Identity morphism modulo $\epsilon$, therefore the morphism is multiplication(on the left) by $I+\epsilon \Psi_0$, where $\Psi_0: \pi_*\mathcal E\rightarrow \pi_*\mathcal E$ an $\mathcal O_{X_0}$ module homomorphism. Therefore if $\sigma_1+\epsilon\sigma_2$ is a local section of $(\pi\times \mathbb 1)_*\mathcal E_{\mathbb D}$ then $(\pi\times \mathbb 1)_*\tilde{\psi}(\sigma_1+\epsilon\sigma_2)=\sigma_1+\epsilon\sigma_2+\epsilon \Psi_0(\sigma_1)$. Since the torsion free sheaf $\pi_*\mathcal E$ is stable therefore the morphism $\Psi_0$ must be multiplication by some scalar $\lambda$. Therefore on $\tilde{X_0}\times \spec k[\epsilon]$ also the restriction of $\tilde{\psi}$ is given by
\begin{equation}
\tilde{\psi}(\sigma_1+\epsilon\sigma_2)=\sigma_1+\epsilon\cdot\sigma_2+\epsilon \lambda \sigma_1
\end{equation}
By restricting the morphism $\tilde{\psi}$ over $R[\epsilon]:=R\times \spec k[\epsilon]$ we get
\begin{equation}
(\mathcal E_{\mathbb D})|_{R[\epsilon]}\xrightarrow{\tilde{\psi}} (\mathcal E_{\mathbb D})|_{R[\epsilon]}
\end{equation}
But $\tilde{\psi}|_{p_i\times \spec k[\epsilon]}(\sigma_1+\epsilon \sigma_2)=\sigma_1+\epsilon \sigma_2+\epsilon \lambda \sigma_1$ for $i=1, r+1$. Since the morphism $\tilde{\psi}$ is $\psi$-equivariant $\mathcal O_{R}[\epsilon]$-module homomorphism and is the Identity morphism modulo $\epsilon$, therefore $\tilde{\psi}$ is multiplication(on the left) by $I+\epsilon \Psi$, where $\Psi: \mathcal E|_{R}\rightarrow \mathcal E|_R$ is $\mathcal O_R$-module. Therefore $\tilde{\psi}(\sigma_1+\epsilon \sigma_2)=\sigma_1+\epsilon \sigma_2+\epsilon \Psi(\sigma_1)$. Now notice that at the two extremal points $p_1$ and $p_{r+1}$, the morphism $\Psi$ is multiplication by the scalar $\lambda$. Since the vector bundle $\mathcal E|_R$ is a strictly standard vector bundle, therefore $\Psi=\lambda\cdot I$.

But this is not possible unless the infinitesimal automorphism $\psi$ is trivial. To see this, notice that $\psi$ is given by
\begin{equation}
\mathcal O_R[\epsilon]\rightarrow \mathcal O_R[\epsilon]
\end{equation}
which maps $f+\epsilon \cdot g\mapsto f+\epsilon (g+X_{\psi}(df))$, where $X_{\psi}$ is the vector field on $R$ corresponding to the infinitesimal automorphism $\psi$. Over a $\psi$-equivariant trivialization $U[\epsilon]$ of $\mathcal E_{\mathbb D}$, where $U$ is an open subset of $R$, we have
\begin{equation}
(\mathcal O_{U})^{\oplus 2}[\epsilon]\xrightarrow{\tilde{\psi}} (\mathcal O_{U})^{\oplus 2}[\epsilon]
\end{equation}
such that $\tilde{\psi}(f_1+\epsilon g_1, f_2+\epsilon g_2)=\tilde{\psi}((f_1,f_2)+\epsilon(g_1,g_2))=\tilde{\psi}(f_1\cdot (1,0))+\tilde{\psi}(f_2\cdot(1,0))+\tilde{\psi}(\epsilon g_1\cdot(0,1))+\tilde{\psi}(\epsilon g_2\cdot(0, 1))=\psi(f_1)\cdot(1,0)+\psi(f_2)\cdot(1,0)+\psi(\epsilon g_1)(0,1)+\psi(\epsilon g_2)(0,1)=(f_1+\epsilon X_{\psi}(f_1))(1,0)+(f_2+\epsilon X_{\psi}(f_2))(1,0)+(\epsilon g_1)(0,1)+(\epsilon g_2)(0,1)=(f_1+\epsilon(g_1+X_{\psi}(df_1)), f_2+\epsilon(g_2+X_{\psi}(df_2)))$.
But then we must have $X_{\psi}(df_1)=\lambda$ and $X_{\psi}(df_2)=\lambda$ for all local functions $f_1$ and $f_2$, which is only possible when $\lambda=0$ and the vector field $X_{\psi}$ is trivial, i.e., the infinitesimal automorphism $\psi$ is trivial.

\textsf{proof of (2)} The proof of the second statement follows similarly using the fact that the induced infinitesimal torsion free Higgs pair $(\pi_*\mathcal E_{\mathbb D}, \pi_{\mathbb D}*\phi)$ is stable and therefore the $\Psi_0=\lambda\cdot \text{Identity}$, where $\lambda$ is a scalar.
\end{proof}
\begin{rema}\label{Free}
From the above lemma it follows that if $0\neq \psi\in H^0(X_r, T_{X_r})$ and $\mathcal E_{\mathbb D}$ is the trivial infinitesimal deformation of $\mathcal E$ over $X_r\times \spec k[\epsilon]$, then $\psi^* \mathcal E_{\mathbb D}\ncong \mathcal E_{\mathbb D}$. Therefore we conclude that $H^0(X_r, T_{X_r})$ is a subspace of $H^1(X_r, \mathcal End \mathcal E)$. Similarly, we can show that $H^0(X_r, T_{X_r})$ is a subspace of $\mathbb H^1(\mathcal C_{\bullet})$.
\end{rema}
\begin{prop}\label{RelTan}
\begin{enumerate}
\item The relative tangent space of $f_{Cur}: \mathcal M_{GVB}\rightarrow \mathcal Log_{(\spec k, \overline{\mathbb N})}$ at a point $(\pi_r: X_r\rightarrow X_0, \mathcal E)$ is isomorphic to $H^1(X_r, \mathcal End \mathcal E)$.
\item The relative tangent space of $f_{Cur}: \mathcal M_{GHB}\rightarrow \mathcal Log_{(\spec k, \overline{\mathbb N})}$ at a point $(\pi_r: X_r\rightarrow X_0, \mathcal E,\phi: \mathcal E\rightarrow \mathcal E\otimes \pi^*_r\omega_{X_0})$ is isomorphic to $\mathbb H^1(\mathcal C_{\bullet})$, where $\mathcal C_{\bullet}$ is the complex
\begin{equation}\label{complex1}
0\rightarrow {\mathcal E}nd~~\mathcal E\xrightarrow{-[\phi,\bullet]} {\mathcal E}nd~~\mathcal E\otimes \pi^*_r\omega_{X_0}\rightarrow 0,
\end{equation}
and the map $[\phi,\bullet](s)=\phi\circ s-(\mathbb 1\otimes s)\circ \phi$.
\end{enumerate}
\end{prop}
\begin{proof}
Since $\mathcal M_{GVB}\rightarrow (\spec k, \overline{\mathbb N})$ is a base-change of the log smooth morphism $\mathcal M_{GVB,S}\rightarrow S$, it is also a log-smooth. Therefore the morphism $f_{Cur}: \mathcal M_{GVB}\rightarrow \mathcal Log_{(\spec k, \overline{\mathbb N})}$ is smooth. By definition \cite[Definition 17.14.2]{21}, the relative tangent space is the fiber product (for the notations see \ref{ConNot})

\begin{equation}
\begin{tikzcd}
\mathcal M_{GVB}(k[\epsilon])\arrow{d} & \mathcal M_{GVB}(k[\epsilon])\times_{T\mathcal Log_{(\spec k, \overline{\mathbb N})}}( k) \mathcal Log_{(\spec k, \overline{\mathbb N})}( k)\arrow{d}\arrow{l}\\
T\mathcal Log_{(\spec k, \overline{\mathbb N})}( k) & \mathcal Log_{(\spec k, \overline{\mathbb N})}( k)\arrow{l}
\end{tikzcd}
\end{equation}

The vertical morphism on the left is the natural projection, and the horizontal morphism below is the null-section \cite[17.11.4]{21}. The fibre product in the diagram is a vector space.

It is well-known that the isomorphism classes of first-order infinitesimal deformations of a vector bundle $\mathcal E$ over a projective curve $C$ are parametrised by the vector space $H^1(C, \mathcal End \mathcal E)$. Therefore we have the following.

\begin{equation}
\left\{
\begin{array}{@{}ll@{}}
\text{Isomorphism classes of stable Gieseker vector bundles}\\
\mathcal E ~~\text{over}~~ X_r\times \spec k[\epsilon]~~\text{such that the}\\
\text{restriction over}~~ X_r ~~\text{is}~~ \mathcal E
\end{array}\right\}\cong H^1(X_r, \mathcal End \mathcal E)
\end{equation}

and from lemma \ref{Log101} it follows that we have a surjective morphism of vector spaces
\begin{equation}\label{Final}
H^1(X_r, \mathcal End \mathcal E) \rightarrow \mathcal M_{GVB}(k[\epsilon])\times_{T\mathcal Log_{(\spec k, \overline{\mathbb N})}( k)} \mathcal Log_{(\spec k, \overline{\mathbb N})}( k).
\end{equation}

From remark \ref{Log102} and lemma \ref{Log103}, it follows that the elements which lies in the image of the morphism $\mathcal M_{GVB}(k[\epsilon])\times_{T\mathcal Log_{(\spec k, \overline{\mathbb N})}( k)} \mathcal Log_{(\spec k, \overline{\mathbb N})}( k)\rightarrow \mathcal M_{GVB}( k[\epsilon])$ are the Gieseker-equivalent classes (definition \ref{GisEq}) of families of stable Gieseker vector bundles $(\mathcal X^{mod}, \mathcal E)$ over $\spec k[\epsilon]$ such that the induced logarithmic structure is on $\spec k[\epsilon]$ is isomorphic to the pull back of the log structure of $\spec k$, defined above in \eqref{Log1001}, under the natural projection map $\spec k[\epsilon]\rightarrow \spec k$. Therefore from lemma \ref{Log101}, it follows that the image of the morphism $\mathcal M_{GVB}(k[\epsilon])\times_{T\mathcal Log_{(\spec k, \overline{\mathbb N})}(k)} \mathcal Log_{(\spec k, \overline{\mathbb N})}( k)\rightarrow \mathcal M_{GVB}( k[\epsilon])$ is isomorphic to the quotient vector space
\begin{equation}
\frac{H^1(X_r, \mathcal End \mathcal E)}{H^0(X_r, T_{X_r})}.
\end{equation}

In lemma \ref{InfAct}, we have shown that the action of $H^0(X_r, T_{X_r})$ on $H^1(X_r, \mathcal End \mathcal E)$ is free and in \ref{Free}, we have remarked that $H^0(X_r, T_{X_r})$ is a vector subspace of $H^1(X_r, \mathcal End \mathcal E)$. Also, notice that the Gieseker equivalence on the infinitesimal families of Gieseker vector bundles is precisely the equivalence induced by the action of $H^0(X_r, T_{X_r})$ on $H^1(X_r, \mathcal End \mathcal E)$.

It follows from the definition of the fibre product of algebraic stacks that the fibre of the surjective morphism of vector spaces.
\[
\mathcal M_{GVB}(k[\epsilon])\times_{T\mathcal Log_{(\spec k, \overline{\mathbb N})}( k)} \mathcal Log_{(\spec k, \overline{\mathbb N})}(k)\rightarrow \frac{H^1(X_r, \mathcal End \mathcal E)}{H^0(X_r, T_{X_r})}
\]
is isomorphic to the vector space ${Aut}^{Inf}_{Y}(M)$ which is isomorphic to $ k^r$ (by lemma \ref{Auto101}). Also notice that $H^0(X_r, T_{X_r})\cong k^r$. Since the morphism \ref{Final} is a surjective morphism between two vector spaces of the same dimension, it has to be an isomorphism. Therefore the relative log tangent space is isomorphic to $ H^1(X_r, \mathcal End \mathcal E)$.

Similarly, if $(X_r, \mathcal E, \phi:\mathcal E\rightarrow \mathcal E\otimes \pi^*_r\omega_{X_0})$ is a Gieseker-Higgs bundle then the relative tangent space of $\mathcal M_{GHB}\rightarrow \mathcal Log_{(\spec k, \overline{\mathbb N})}$ at the point $(X_r, \mathcal E, \phi:\mathcal E\rightarrow \mathcal E\otimes \pi^*_r\omega_{X_0})$ is isomorphic to
\begin{equation}
\left\{
\begin{array}{@{}ll@{}}
\text{Isomorphism classes of Higgs bundles}\\
(\mathcal E, \tilde{\phi}) ~~\text{over}~~ X_r\times \spec k[\epsilon]~~\text{such that the}\\
\text{restriction over}~~ X_r ~~\text{is}~~ (\mathcal E,\phi)
\end{array}\right\}\cong \mathbb H^1(\mathcal C_{\bullet}), (\text{subsection}~~ \ref{SymCur} ~~\text{and}~~ \text{remark}~~ \ref{Liou441})
\end{equation}
where $\mathcal C_{\bullet}$ is the complex \eqref{complex1}.\end{proof}

\begin{thm}\label{Tann}
Let $(X_r,\mathcal E,\phi)$ be any Gieseker-Higgs bundle. Then
\begin{equation}\label{Trel1}
T_{\mathcal M_{GVB,S}/S}(-{log} \mathcal M_{GVB})|_{(X_r, \mathcal E)}\cong H^1(X_r, \mathcal End \mathcal E), T_{\mathcal M_{GHB,S}/S}(-{log} \mathcal M_{GHB})|_{(X_r, \mathcal E,\phi)}\cong \mathbb H^1(\mathcal C_{\bullet}).
\end{equation}
\end{thm}
\begin{proof}
Since the two log structures $f_{Div}: \mathcal M_{GVB,S}\rightarrow \mathcal Log_{S}$ and $f_{Cur}: \mathcal M_{GVB,S}\rightarrow \mathcal Log_{S}$ are the same therefore the relative tangent bundles are also isomorphic. Therefore the theorem follows from \cite[3.8]{28}, proposition \ref{RelTan} and proposition \ref{logs}.
\end{proof}

\begin{rema}
It is easy to see that although the universal bundle $U$ on $\Delta^{st}_S$ may not, in general, descend to the quotient $\mathcal X_S$, the vector bundle $\mathcal End ~~U$ descends to $\mathcal X_S$. It is because the $GL(N)-$ action at a stable bundle has stabilizer $\mathbb C^*$ and the action of the stabilizer on the bundle $\mathcal End ~~U$ is trivial and therefore the bundle $\mathcal End ~~U$ is in fact $PGL(N)$ equivariant. Therefore, the vector bundle $\mathcal End~~U $ descends. We denote it by $\mathcal End ~~\mathcal E$. Similarly, the vector bundle $R^1\pi_*\mathcal End~~U$ (and $\mathbb R^1\pi_*\mathcal C_*$) also descends to $\mathcal M_{GVB,S}$ (and $\mathcal M_{GHB,S}$). 
\end{rema}

\section{\textbf{Relative Log symplectic structure on $\mathcal M_{GHB,S}$}}

In this section, we will show that there is a relative log symplectic structure on $\mathcal M_{GHB, S}\rightarrow S$ and also describe it functorially.

Consider the following composite morphism.

\begin{equation}
\begin{tikzcd}
\wedge^2 \Omega_{\mathcal M_{GHB,S/S}} ({log} \mathcal M_{GHB})\arrow{r} & \mathcal M_{GHB,S}\arrow{r} & S
\end{tikzcd}
\end{equation}

Now over the generic point $\eta$ of $S$, the first projection map has a natural section which corresponds to the symplectic form on moduli of Higgs bundles over the generic curve (\cite[Section 8]{14},\cite[Section 4]{7}, \text{subsection}~~ \ref{SymCur} ~~\text{and}~~ \text{remark}~~ \ref{Liou441}). Let us denote the generic fibre of $\mathcal M_{GHB, S}\rightarrow S$ by $\mathcal M_{HB}$. It is an open subset of $\mathcal M_{GHB, S}$.

Consider the relative logarithmic cotangent bundle.

\begin{equation}
\begin{tikzcd}
\Omega_{\mathcal M_{GVB,S/S}}({log} \mathcal M_{GVB})\arrow{r} & \mathcal M_{GVB,S}\arrow{r} & S
\end{tikzcd}
\end{equation}

From theorem \ref{Trel1}, it follows that for any stable Gieseker vector bundle $(X_r, \mathcal E)$
\begin{equation}
\Omega_{\mathcal M_{GVB,S/S}} ({log} \mathcal M_{GVB})|_{(X_r, \mathcal E)}\cong H^1(X_r, \mathcal End \mathcal E)^*
\end{equation}

\subsubsection{\textbf{Few properties of the map $\pi_r$}}\label{ProPro123} Let us remind here that $\pi_r$ is the projection morphism $X_r\rightarrow X_0$ and recall the following facts from \cite[Proposition 3]{25}
\begin{enumerate}
\item $\pi_r^* \omega_{X_0}\cong \omega_{X_r}$,
\item $R^i(\pi_r)_* \mathcal O_{X_r}=0,$ for all $i>0$,
\item $(\pi_r)_*\mathcal O_{X_r}\cong \mathcal O_{X_0}$.
\end{enumerate}

Now, using Serre duality for nodal curves and the properties above, we see that
\begin{equation}
H^1(X_r, \mathcal End \mathcal E)^*\cong H^0(X_r, \mathcal End \mathcal E\otimes \pi^*_r\omega_{X_0})\cong Hom(\mathcal E, \mathcal E\otimes \pi^*_r\omega_{X_0})
\end{equation}

Therefore we have a morphism $\Omega_{\mathcal M_{GVB,S/S}} ({log} \mathcal M_{GVB})\rightarrow \mathcal M_{GHB,S}$, which is clearly injective. The objects of $\Omega_{\mathcal M_{GVB,S/S}} ({log} \mathcal M_{GVB})$ are precisely those Gieseker-Higgs bundles whose underlying Gieseker vector bundle is stable. By the openness of stability of Gieseker vector bundles it follows that $\Omega_{\mathcal M_{GVB,S/S}} ({log} \mathcal M_{GVB})\rightarrow \mathcal M_{GHB,S}$ is an open immersion.

There is a natural relative log-symplectic structure $\omega$ on $\Omega_{\mathcal M_{GVB,S/S}} ({log} \mathcal M_{GVB})$ (theorem \ref{log-sym1}). Now consider the union of the two open subvarieties $\mathcal M_{HB}\cup \Omega_{\mathcal M_{GVB,S/S}} ({log} \mathcal M_{GVB})$. The symplectic structure on $\mathcal M_{HB}$ and the relative log-symplectic structure on $\Omega_{\mathcal M_{GVB,S/S}} ({log} \mathcal M_{GVB})$ agree on the intersection $\mathcal M_{HB}\cap \Omega_{\mathcal M_{GVB,S/S}} ({log} \mathcal M_{GVB})=\Omega_{\mathcal M_{VB}}$. Therefore there is a relative log-symplectic structure on the union of the two open subsets $\mathcal M_{HB}\cup \Omega_{\mathcal M_{GVB,S/S}} ({log} \mathcal M_{GVB})$. The following lemma shows that $\mathcal M_{HB}$ is a dense open subset of $\mathcal M_{GHB,S}$.

\begin{lema}\label{Dense1}
Let $(X_r,\mathcal E,\phi: \mathcal E\rightarrow \mathcal E\otimes \pi^*_r\omega_{X_0})$ be any stable Gieseker-Higgs bundle. There exists a family of stable Gieseker-Higgs bundles $(\mathcal X^{mod}_S,\mathcal E_S, \phi_S: \mathcal E_S\rightarrow \mathcal E_S\otimes \omega_S)$ over a complete discrete valuation ring $S$, whose generic fiber is a stable Higgs bundle over the smooth curve $\mathcal X_{\eta}$ and the special fiber is $(X_r,\mathcal E,\phi)$.
\end{lema}
\begin{proof}
The proof follows from the openness of a flat morphism and the fact that the relative moduli space $\mathcal M_{GHB, S}\rightarrow S$ is flat over $S$ \cite[Theorem 1.1]{3}.
\end{proof}

\begin{thm}\label{Main1}
There is a relative logarithmic-symplectic form on $\mathcal M_{GHB, S}\rightarrow S$, which is the classical symplectic form on the generic fibre and is a log-symplectic form on the special fibre.
\end{thm}
\begin{proof}
Given a Gieseker-Higgs bundle $(X_r, \mathcal E,\phi)$ consider the following complex $\mathcal C_{\bullet}$:
\begin{equation}
0\rightarrow \mathcal End \mathcal E\xrightarrow{-[\phi,\bullet]} \mathcal End \mathcal E\otimes \pi^*_r\omega_{X_0}\rightarrow 0,
\end{equation}
where $[\phi,\bullet](s)=\phi\circ s-(\mathbb 1\otimes s)\circ \phi$. The log-cotangent space of the moduli $\mathcal M_{GHB}$ at $(X_r, \mathcal E,\phi)$ is isomorphic to $\mathbb H^1(\mathcal C_{\bullet})$ (\text{subsection}~~ \ref{SymCur} ~~\text{and}~~ \text{remark}~~ \ref{Liou441}). The dual of this complex (let us denote by $\mathcal C^{\vee}_{\bullet}$) is
\begin{equation}
0\rightarrow \mathcal End \mathcal E\xrightarrow{[\phi,\bullet]} \mathcal End \mathcal E\otimes \pi^*_r\omega_{X_0}\rightarrow 0
\end{equation}

We have the following isomorphism of complexes:

\begin{equation}\label{Pair}
\begin{tikzcd}
\mathcal End \mathcal E \arrow{d}{[\phi, \bullet]}\arrow[" \mathbb 1"]{r} & \mathcal End \mathcal E\arrow{d}{-[\phi, \bullet]}\\
\mathcal End \mathcal E\otimes \pi^*_r\omega_{X_0} \arrow["-\mathbb 1\otimes \mathbb 1"]{r} & \mathcal End \mathcal E\otimes \pi^*_r\omega_{X_0}
\end{tikzcd}
\end{equation}

This induces an isomorphism 

\begin{equation}
(\sigma')^{\flat}:\mathbb H^1(\mathcal C_{\bullet}^{\vee})\rightarrow \mathbb H^1(\mathcal C_{\bullet})
\end{equation}

Therefore, we have an isomorphism 
\begin{equation}\label{Pair2}
(\sigma')^{\flat}:\Omega_{\mathcal M_{GHB,S/S}}({log}~\mathcal M_{GHB})|_{(X_r,E,\phi)}\xrightarrow{\cong} T_{ \mathcal M_{GHB,S}} (-{log}~~\mathcal M_{GHB})|_{(X_r,E,\phi)}
\end{equation}

In other words, this gives a non-degenerate relative logarithmic two-form. We denote it by $\omega'$. From the choice of the sign in the diagram \ref{Pair}, it follows that $\omega'$ is skew symmetric. Therefore, this gives a section $\omega':\mathcal M_{GHB,S}\rightarrow \wedge^2 \Omega_{\mathcal M_{GHB,S/S}} ({log}~~\mathcal M_{GHB})$. Following theorem \ref{log-sym1}, we have another section $\omega$ over $\mathcal M_{HB}\cup \Omega_{\mathcal M_{GVB,S/S}} ({log} \mathcal M_{GVB})$, and from \cite[Section 4]{7}, \cite[Theorem 4.5.1.]{8},  and \cite[Section 8]{14} it follows that these two sections are the same on the open subset $\mathcal M_{HB}$ (see also  \text{subsection}~~ \ref{SymCur} ~~\text{and}~~ \text{remark}~~ \ref{Liou441}). Therefore these two $2$-forms are the same and hence $\omega'$ is the extension of the relative log-symplectic form $\omega$ discussed in theorem \ref{log-sym1}. 

\end{proof}

\begin{rema}\label{Reducible}
Suppose that the special fiber of the surface $\mathcal X\rightarrow S$ is a reducible curve of the form $C_1\cup C_2$, where $C_1$ and $C_2$ are smooth curves transversally intersecting at a point $C_1\cap C_2$. We can similarly construct the moduli of Gieseker-Higgs bundles in this case. If we concentrate on the case $(\chi, r)=1$, we can ensure that all the semistable objects are stable for a generic choice of polarisation. As a result, the moduli is a variety with normal crossing singularity (see \cite[remark 9.3]{3} and \cite{4} for details). It follows similarly that there is a relative log-symplectic form on this degeneration. A particular case interesting for many computations is when the rank is $2$, and $\chi$ is odd. For a generic choice of polarisation, one can show that the moduli of stable torsion-free Hitchin pairs on the nodal curve $C_1\cup C_2$ consists of $(\mathcal F, \phi)$, where $\mathcal F$ has local types either $\mathcal O\oplus \mathcal O$ or $\mathcal O\oplus m$. Since we can avoid the local type $m\oplus m$, the moduli of stable torsion-free Higgs pair coincides with the moduli of Gieseker-Higgs bundles (\cite[remark 9.3]{3} and \cite{4}). We will show in \S 7 that in this special case the special fibre $\mathcal M_{GHB}$ is union of two smooth log-symplectic manifolds transversally intersecting along a divisor.
\end{rema}

\section{\textbf{Moduli space of Higgs bundles on a fixed Gieseker curve}}
In this section, we discuss about the moduli of Higgs bundles of rank $n$ and Euler-characteristic $\chi$ on a fixed Gieseker curve $X_r$ such that $g.c.d(n,\chi)=1$. In \cite{22}, Kiem and Li introduced and studied semi-stability of vector bundles on a fixed Gieseker curve with respect to a special polarisation on the curve whose degree on every rational component is sufficiently smaller (in ratio) than the degree on the curve $\tilde{X_0}$. We will refer to this notion of semistability as the $\epsilon$-semistability. Moreover, they show that $\epsilon$-semistable vector bundles are quasi-Gieseker vector bundles, i.e., they are standard bundles \eqref{stan450} and the push-forward is a torsion-free sheaf on the nodal curve $X_0$. Later in \cite{33}, Sun introduced the notion of $0$-semi-stability of vector bundles on a fixed Gieseker curve and showed that when the Euler characteristic of the bundle is positive and co-prime to the rank, the two notions, namely $\epsilon$-semistability and $0$-semistabilty coincide. In this section, we adapt these notions for the Higgs bundles on a fixed Gieseker curve.

To motivate this, we first notice that the moduli of Gieseker-Higgs bundles $\mathcal M_{GHB}$ has the following Whitney stratification given by the successive singular locus.
\begin{equation}\label{strat2055}
\mathcal M_{GHB}\supset \partial^1\mathcal M_{GHB}\supset \dots\supset \partial^n\mathcal M_{GHB} \supset \cdots
\end{equation}

The purpose of the discussion in this section is to show that each stratum is a torus quotient of the moduli of $\epsilon$-semistable Higgs bundles on a fixed Gieseker curve. 

Let $0\leq \epsilon<<1$ be an arbitrarily small non-negative number. Since the purpose is to describe the stratum, as mentioned above, it is safe to \textbf{assume that $\chi(\mathcal E)>0$} (see also remark \ref{chi1}).

\begin{defe}
A Higgs bundle $(X_r, \mathcal E,\phi)$ is called $\epsilon$-semistable ($\epsilon$ stable) if for all $\phi$-invariant subsheaf $\mathcal F$ of $\mathcal E$ we have
\begin{equation}
\chi(\mathcal F)\leq (<) \frac{\chi(\mathcal E)}{n}rk_{\epsilon}(\mathcal F),
\end{equation}
where $rk_{\epsilon}(\mathcal F):=(1-r\epsilon)\text{rank}~~ \mathcal F|_{\tilde{X_0}}+\epsilon\sum^{r}_{i=1}(\text{rank}~~ \mathcal F|_{R_i})$
\end{defe}

\begin{defe}
A Higgs bundle $(X_r, \mathcal E,\phi)$ is called $0$-semistable if for all $\phi$-invariant subsheaf $\mathcal F$ of $\mathcal E$ we have
\begin{equation}
\chi(\mathcal F)\leq \frac{\chi(\mathcal E)}{n}rk_{0}(\mathcal F),
\end{equation}
 
and it is called $0$-stable if it is $0$-semistable and 

\begin{equation}
\chi(\mathcal F)\leq \frac{\chi(\mathcal E)}{n}rk_{0}(\mathcal F), ~~~\text{when}~~~ rk_0(\mathcal F)\neq 0.
\end{equation}
\end{defe}

The proof of the following two lemmas are elementary; we leave it to the reader.

\begin{lema}
$(X_r, \mathcal E,\phi)$ is $\epsilon$-stable if and only if it is $0$-stable.  
\end{lema}

\begin{lema}
Assume $(\chi(\mathcal E), n)=1$. Then
\begin{enumerate}
\item $(X_r, \mathcal E, \phi)$ is $\epsilon$-semistable if and only if it is $\epsilon$-stable. 
\item $(X_r, \mathcal E, \phi)$ is $0$-semistable if and only if it is $0$-stable.
\end{enumerate}
\end{lema}

\begin{lema}\label{stan}
If $(X_r, \mathcal E,\phi)$ is $\epsilon$-stable then $\mathcal E$ satisfies the following two properties.
\begin{enumerate}\label{stan450}
\item $\mathcal E$ is a standard vector bundle i.e., 
\begin{equation}
\mathcal E|_{R[r]_i}\cong \mathcal O^{\oplus a_i}\oplus \mathcal O(1)^{\oplus b_i}~~\text{ for all}~~ i=1,\dots, r, ~~\text{and}
\end{equation} 
\item ${\pi_r}_*\mathcal E$ is a torsion-free sheaf. 
\end{enumerate}
\end{lema}
\begin{proof}
For each $i\in \{1,\dots, r\}$, we denote by $x^+_i, x^-_i$ the two marked points on the rational curve $R[r]_i$. We denote the Zariski-closure of the curve $X_r\setminus R[r]_i$ by $\tilde X_i$. We denote the two extremal points on $\tilde X_0$ by $x^+$ and $x^-$. 

Suppose $\exists$ an integer $i$ such that $\mathcal E|_{R[r]_i}$ has a negative degree line sub bundle as a direct summand. Let $m$ and $a$ be the largest positive integers such that $\mathcal O_{R[r]_i}(-m)^{\oplus a}$ is a direct summand of $\mathcal E|_{R[r]_i}$. Let $K$ be the kernel of the surjection $ \mathcal E\rightarrow \mathcal O_{R[r]_i}(-m)^{\oplus a}$. It follows that $K$ is a $\phi$-invariant subsheaf of $\mathcal E$. We have 
\begin{equation}
\chi(K)=\chi(\mathcal E)-\chi(\mathcal O_{R_i}(-m)^{\oplus a})=\chi(\mathcal E)+a(m-1).
\end{equation}
Since $\text{rank}~~ K|_{\tilde X_0}=n$, we have $\chi(\mathcal E)+a(m-1)<\chi(\mathcal E) $ which implies $a(m-1)<0$. This is a contradiction. Therefore $m\leq 0$.

Suppose $\exists$ an integer $i$ such that $\mathcal E|_{R[r]_i}$ has a positive degree line sub bundle as a direct summand. Let $m$ and $a$ be the largest positive integers such that $\mathcal O_{R[r]_i}(m)^{\oplus a}$ is a direct summand of $\mathcal E|_{R[r]_i}$. We have $\mathcal E|_{R[r]_i}=\mathcal O_{R[r]_i}(m)^{\oplus a}\oplus M$, where $M$ is a vector bundle on $R[r]_i$. We have a short exact sequence 
\begin{equation}
0\rightarrow \mathcal E|_{\tilde X_i}(-x^+_i-x^-_i)\rightarrow \mathcal E \rightarrow \mathcal E|_{R[r]_i}\rightarrow 0
\end{equation}
From the above short exact sequence it follows that $(\mathcal O_{R[r]_i}(m)^{\oplus a})\otimes \mathcal O_{R[r]_i}(-x^+_i-x^-_i)$ is a subsheaf of $\mathcal E$. Notice that $(\mathcal O_{R[r]_i}(m)^{\oplus a})\otimes \mathcal O_{R[r]_i}(-x^+_i-x^-_i)\cong \mathcal O_{R[r]_i}(m-2)^{\oplus a}$. It follows that $\mathcal O_{R[r]_i}(m-2)^{\oplus a}$ is a $\phi$-invariant subsheaf of $\mathcal E$. The $\epsilon$-stability implies
$$
\chi(\mathcal O_{R[r]_i}(m-2)^{\oplus a})=a(m-1) \leq 0\implies m\leq 1.
$$
Therefore $\mathcal E$ is a standard vector bundle.

The sheaf $\pi_*\mathcal E$ is torsion free if and only if $H^0(\mathcal E|_{R[r]}(-x^+-x^-))=0$. Suppose $ H^0(\mathcal E|_{R[r]}(-x^+-x^-))\neq 0$. Consider the sub-bundle $F$ of $\mathcal E|_{R[r]}(-x^+-x^-)$ generated by $H^0(\mathcal E|_{R[r]}(-x^+-x^-))$. It is a $\phi$-invariant sub-sheaf of $\mathcal E$. Since it is generically generated by global sections, $\chi(F)\geq 1$. But $\epsilon$-stability implies $\chi(F)\leq 0$, which is a contradiction. Therefore, $H^0(\mathcal E|_{R[r]}(-x^+-x^-))=0$.
\end{proof}

\begin{rema}\label{qGH}
If a Higgs bundle $(\mathcal E, \phi)$ on $X_r$ satisfies the condition $(1)$ and $(2)$ in Lemma \ref{stan}, we call it a quasi-Gieseker-Higgs bundle. 
\end{rema}

\begin{defe}
A generalised parabolic Higgs bundle (GPH) on $\tilde X_0$ is a triple $(E, \phi, F(E))$, where $E$ is a vector bundle, $\phi: E\rightarrow E\otimes \omega_{\tilde X_0}(x^++x^-)$ is a homomorphism and $F(E)\subseteq E_{x^+}\oplus E_{x^-}$ is any sub-vector space such that $(q_*\phi) (q_*F(E))\subseteq q_*(F(E))\otimes \omega_{X_0}$, where $q: \tilde X_0\rightarrow X_0$ is the normalisation morphism.
\end{defe}

Given a GPH $(E, \phi, F(E))$ we have the following torsion-free sheaf
\begin{equation}
\mathcal F:=Kernel~~\big(q_*E\rightarrow \frac{E_{x^+}\oplus E_{x^-}}{F(E)}\big)
\end{equation}
Since $(q_*\phi) (q_*F(E))\subseteq q_*(F(E))\otimes \omega_{X_0}$ the morphism $\phi$ induces a homomorphism $\phi_0:\mathcal F\rightarrow \mathcal F\otimes \omega_{X_0}$.
\begin{prop}
The GPH $(E,\phi, F(E))$ is semistable (stable) if and only if the induced torsion-free Higgs pair $(\mathcal F, \phi_0)$ is semi-stable(stable). 
\end{prop}
\begin{proof}
The proof is similar to \cite[Proposition 4.2]{Bh}.
\end{proof}

\begin{prop}
If a Higgs bundle $(X_r, \mathcal E, \phi)$ is $\epsilon$-stable, then $((\pi_r)*\mathcal E, (\pi_r)_*\phi)$ is stable torsion-free Higgs pair. If $\mathcal E|_R$ is positive and $((\pi_r)*\mathcal E, (\pi_r)_*\phi)$ stable then $(X_r, \mathcal E, \phi)$ is $\epsilon$-stable. 
\end{prop}
\begin{proof}
Set $\tilde E:=\mathcal E|_{\tilde X_0}$ and $\tilde F:=\mathcal E|_{R[r]}$. Since $(X_r, \mathcal E,\phi)$ is $\epsilon$-stable therefore from Lemma \ref{stan} it follows that $H^0(\tilde F(-x^+-x^-))=0$. Therefore the first map in the following sequence is injective
\begin{equation}
H^0(R[r], \tilde F)\xrightarrow{s\mapsto (s(x^+), s(x^-))} \tilde F_{x^+}\oplus \tilde F_{x^-}\xrightarrow{\theta_1\oplus \theta_2} \tilde E_{x^+}\oplus \tilde E_{x^-} ,
\end{equation}
where $\theta_1: \tilde F_{x^+}\rightarrow \tilde E_{x^+}$ and $\theta_2: \tilde F_{x^-}\rightarrow \tilde E_{x^-}$ are the gluing isomorphisms. Since $H^0(R[r],\tilde F)$ is $\phi$ invariant therefore we get a GPH $(\tilde E, (\theta_1\oplus \theta_2)(H^0(R[r], \tilde F)), \tilde \phi)$ whose induced torsion-free Higgs pair is $((\pi_r)_*\mathcal E, (\pi_r)_*\phi)$.

It is enough to show that the GPH is stable. The rest of the proof follows from the observation that for any $\tilde \phi $-invariant sub-sheaf $\tilde E'\subset \tilde E$ (on $\tilde X_0$), the sub-sheaf $E'$ of $\mathcal E$ constructed in \cite[Proposition 1.6]{33} is $\phi $-invariant. The converse also follows from similar arguments.
\end{proof}

Let $\chi$ be a positive integer such that $(\chi, n)=1$ and $\epsilon$ be a sufficiently small positive number. Let us denote by $\mathcal M^{\chi, n, \epsilon}_{HB, X_r}$ the moduli space of $\epsilon$-stable Higgs bundles $(\mathcal E, \phi)$ on the curve $X_r$ of rank $n$ and $\chi(\mathcal E)=\chi$. For the construction of the moduli space we refer to \cite{32} and \cite[Theorem B.12.]{Bh I}. Notice that if $(\mathcal E,\phi)\in \mathcal M^{\chi, n, \epsilon}_{HB, X_r}$, then from Lemma \ref{stan} it follows that it is a quasi-Gieseker-Higgs bundle (\ref{qGH}). Since $(\chi, n)=1$, one can show that the moduli space $\mathcal M^{\chi, n, \epsilon}_{HB, X_r}$ is a fine moduli space. The proof of the existence of a universal family is similar to the proof of proposition \ref{Universal}. Let us fix a universal family $(\mathcal E^{univ}, \phi^{univ})$ over $X_r\times \mathcal M^{\chi, n, \epsilon}_{HB, X_r}$. For every $i=1,\dots, r$, let us choose a smooth point $s_i$ of $R[r]_i$. Consider the map
\begin{equation}\label{map1101}
\mathcal M^{\chi, n, \epsilon}_{HB, X_r}\rightarrow \prod^r_{i=1}\{0,1,\dots, r\}
\end{equation}  
$$
\text{given by}~~[(\mathcal E, \phi)]\mapsto (dim~~H^0(\mathcal E|_{R[r]_1}\otimes \mathcal O_{R[r]_1}(-s_1) ), \cdots, dim~~H^0(\mathcal E|_{R[r]_r}\otimes \mathcal O_{R[r]_r}(-s_r) ))
$$

Notice that if $\mathcal E|_{R[r]_i}\cong \mathcal O_{R[r]_i}^{\oplus a}\oplus \mathcal O_{R[r]_i}(1)^{\oplus b}$, where $a+b=n$, then $dim~~H^0(E|_{R[r]_i}\otimes \mathcal O_{R[r]_i}(-s_i)) =b$. Since the codomain is discrete, we see that the inverse image of every element of the codomain is a disjoint union of some connected components of $\mathcal M^{\chi, n, \epsilon}_{HB, X_r}$.

\subsubsection{\textbf{Action of $Aut(X_r/X_0)$ on the moduli space $\mathcal M^{\chi, n, \epsilon}_{HB, X_r}$}}\label{Admi102} Let $(\mathcal E^{univ},\phi^{univ})$ be a universal family over $X_r\times \mathcal M^{\chi, n, \epsilon}_{HB, X_r}$. Given any $\gamma\in Aut(X_r/X_0)$ \eqref{Aut2025}, consider the pullback family $(\gamma^*\mathcal E^{univ}, \gamma^*\phi^{univ})$ over $X_r\times \mathcal M^{\chi, n, \epsilon}_{HB, X_r}$. Notice that given any Gieseker-Higgs bundle $(\mathcal E, \phi)$ the pullback $(\gamma^*\mathcal E, \gamma^*\phi)$ induces the same torsion-free Higgs pair i.e., $(\pi_r)_*\mathcal E\cong (\pi_r)_*(\gamma^*\mathcal E)$ and ${\pi_r}_*\phi={\pi_r}_*(\gamma^*\phi)$. Therefore $\chi(\mathcal E)=\chi(\gamma^*\mathcal E)=\chi(\mathcal F)$. Moreover, $(\mathcal E, \phi)$ is $\epsilon$-stable if and only if $(\gamma^*\mathcal E, \gamma^*\phi)$ is $\epsilon$-stable. Therefore, we see that $\gamma^*(\mathcal E^{univ}, \phi^{univ})$ is also a family of $\epsilon$-stable Gieseker-Higgs bundles. In other words, we have an action of $Aut(X_r/X_0)$ on the moduli space $\mathcal M^{\chi, n, \epsilon}_{HB, X_r}$. 

Given any $a_{\bullet}:=(a_1,\dots,a_r)\in \prod^r_{i=1}\{0,1,\dots, r\}$, let us denote by ${\mathcal M^{\chi, n, \epsilon, a_{\bullet}}_{HB, X_r}}$ the inverse image of $a_{\bullet}$ by the map \eqref{map1101}. Clearly the action of $Aut(X_r/X_0)$ on ${\mathcal M^{\chi, n, \epsilon}_{HB, X_r}}$ induces an action on ${\mathcal M^{\chi, n, \epsilon, a_{\bullet}}_{HB, X_r}}$. We call a tuple $a_{\bullet}:=(a_1,\dots, a_r)$ \textbf{admissible} if $a_i\geq 1$ for every $i=1,\dots, r$. From lemma \ref{Equ1011}, it follows that if $a_{\bullet}$ is admissible, the action of $Aut(X_r/X_0)$ on ${\mathcal M^{\chi, n, \epsilon, a_{\bullet}}_{HB, X_r}}$ is free. We define 

\begin{equation}\label{admi}
\mathcal M^{\chi, n, \epsilon, ad}_{VB, X_r}:=\bigcup_{a_{\bullet}~~ admissible} ~~\mathcal M^{\chi, n, \epsilon, a_{\bullet}}_{VB, X_r}
\end{equation}

and

\begin{equation}\label{admi}
\mathcal M^{\chi, n, \epsilon, ad}_{HB, X_r}:=\bigcup_{a_{\bullet}~~ admissible} ~~\mathcal M^{\chi, n, \epsilon, a_{\bullet}}_{HB, X_r}
\end{equation}
 
\begin{rema}\label{chi1}
We recall that we had fixed a positive integer $n$ which denotes the rank and an integer $d$ which denotes the degree satisfying $g.c.d(n,d)=1$. We denote by $\mathcal M_{GHB}$ the moduli of Gieseker-Higgs bundles of rank $n$ and degree $d$ on the nodal curve $X_0$. Notice that for a Gieseker-Higgs bundle $(\mathcal E,\phi)\in \mathcal M_{GHB}$, we have $\chi(\mathcal E)=\chi({\pi_r}_*\mathcal E)=d+n(1-g)$. So if $d<n(g-1)$, we see that $\chi(\mathcal E)$ is not positive. But if we choose a smooth point $x\in X_0$ and a positive integer $N$ such that $d+n\cdot N>n(g-1)$, then tensoring every Gieseker-Higgs bundle with $\mathcal O(N\cdot x)$ we get an isomorphism from the moduli space of Gieseker-Higgs bundles of rank $n$ and degree $d$ to the moduli of Gieseker-Higgs bundles of rank $n$ and degree $d+n\cdot N$. So we can safely assume that $\chi(\mathcal E)>0$. Therefore, we have a morphism $\mathcal M^{\chi, n, \epsilon, ad}_{HB, X_r}\rightarrow \mathcal M_{GHB}$, where 
\begin{enumerate}
\item $\chi=d+n(1-g)$, if $d+n(1-g)>0$  
\item $\chi=d+N\cdot n+n(1-g)$, for some sufficiently large positive integer $N$, if $d+n(1-g)\geq 0$. 
\end{enumerate}
\end{rema}

\begin{rema} \label{rema450}
Let 
\begin{equation}
\mathcal M_{GHB}\supset \partial^1 \mathcal M_{GHB}\supset \partial^2 \mathcal M_{GHB}\supset \cdots
 \end{equation}
 be the stratification of $\mathcal M_{GHB}$ given by the successive singular locus. We will show in Propostion \ref{Strat} that $\partial^l\mathcal M_{GHB}\setminus \partial^{l+1}\mathcal M_{GHB}$ is the locally closed subvariety of $\mathcal M_{GHB}$ consisting of the stable Gieseker-Higgs bundles $(X_l, \mathcal E, \phi)$. Therefore, the image of the morphism $\mathcal M^{\chi, n, \epsilon, ad}_{HB, X_r}\rightarrow \mathcal M_{GHB}$ is precisely $\partial^r\mathcal M_{GHB}\setminus \partial^{r+1}\mathcal M_{GHB}$. In fact, $\mathcal M^{\chi, n, \epsilon, ad}_{HB, X_r}\rightarrow \partial^r\mathcal M_{GHB}\setminus \partial^{r+1}\mathcal M_{GHB}$ is a principal $Aut(X_r/X_0)$-bundle. Let us denote the morphism $\mathcal M^{\chi, n, \epsilon, ad}_{HB, X_r}\rightarrow \mathcal M_{GHB}$ by $f_r$. Using the map \eqref{map1101}, we see that the stratum 
\begin{equation}\label{admi445}
\partial^r\mathcal M_{GVB}\setminus \partial^{r+1}\mathcal M_{GVB}=\bigcup_{a_{\bullet}~~\text{admissible}} \mathcal M^{a_{\bullet}}_{GVB}
\end{equation} 
and 
\begin{equation}\label{admi445}
\partial^r\mathcal M_{GHB}\setminus \partial^{r+1}\mathcal M_{GHB}=\bigcup_{a_{\bullet}~~\text{admissible}} \mathcal M^{a_{\bullet}}_{GHB}
\end{equation} 
where $\mathcal M^{a_{\bullet}}_{GVB}$ and $\mathcal M^{a_{\bullet}}_{GHB}$ are defined as in subsubsection \eqref{Admi102} using the map \eqref{map1101}. Moreover, $f_r^{-1}(\mathcal M^{a_{\bullet}}_{GVB})=\mathcal M^{\chi, n, \epsilon, a_{\bullet}}_{VB, X_r}
$ and $f_r^{-1}(\mathcal M^{a_{\bullet}}_{GHB})=\mathcal M^{\chi, n, \epsilon, a_{\bullet}}_{HB, X_r}
$.
\end{rema}

\subsubsection{\textbf{Symplectic structure on $\mathcal M^{\chi, n, \epsilon, ad}_{HB, X_r}$}} 

\begin{lema}\label{Tanagain}
\begin{enumerate}
\item The tangent space of the moduli space $\mathcal M^{\chi, n, \epsilon, ad}_{HB, X_r}$ at a point $(\mathcal E, \phi)$ is naturally isomorphic to $\mathbb H^1(X_r, \mathcal C_{\bullet}(\mathcal E, \phi))$, where $\mathcal C_{\bullet}(\mathcal E, \phi)$ is the following complex:
\begin{equation}
\mathcal End \mathcal E\xrightarrow{[\cdot, \phi]} \mathcal End \mathcal E\otimes \omega_{X_r}, 
\end{equation}
where $[\cdot, \phi](s):=[s, \phi]=s\circ \phi-\phi\circ s$.

\item The cotangent space of the moduli space $\mathcal M^{\chi, n, \epsilon, ad}_{HB, X_r}$ at a point $(\mathcal E, \phi)$ is naturally isomorphic to $\mathbb H^1(X_, \mathcal C^{\vee}_{\bullet}(\mathcal E, \phi))$, where $\mathcal C^{\vee}_{\bullet}(\mathcal E, \phi)$ is the following complex:
\begin{equation}
\mathcal End \mathcal E\xrightarrow{[\phi, \cdot]} \mathcal End \mathcal E\otimes \omega_{X_r}, 
\end{equation}
where $[\phi,\cdot ](s):=[\phi, s]=-s\circ \phi+\phi\circ s$.
\end{enumerate}
\end{lema}

\begin{proof}
The proof follows from subsection \ref{SymCur} and remark \ref{Liou441}. 
\end{proof}

Let $S$ be a complete discrete valuation ring. Let us fix $\mathcal X_r\rightarrow S$, a flat family of projective curves such that the generic fibre $\mathcal X_{r,\eta}$ is a smooth curve of genus $g$, the closed fibre is the nodal curve $X_r$ and the total space $\mathcal X_r$ is regular over $\spec~~ \mathbb C$. Again, the existence of such a family follows from \cite[Theorem B.2 and Corollary B.3, Appendix B]{22}. Let us denote by $\omega_{\mathcal X_r/S}$ the relative dualising sheaf. 

We can choose a line bundle $\mathcal O_{\mathcal X_r/S}(1)$, which has the following property

\begin{equation}\label{Lanebun}
\text{If}~~deg~~\mathcal O_{\mathcal X_r/S}(1)|_{\tilde X_0}=b_0~~\text{and}~~ deg~~\mathcal O_{\mathcal X_r/S}(1)|_{R_i}=b~~\text{for}~~ i=1,\dots, r, ~~\text{then}~~ b_0\neq 0, b\neq 0 ~~\text{and}~~ \frac{b}{b_0}=\frac{\epsilon}{1-r\epsilon}.
\end{equation}

To construct such a line bundle, we first choose a line bundle $\mathcal O_{X_r}(1)$ on $X_r$ satisfying \eqref{Lanebun}. Then a line bundle $\mathcal O_{\mathcal X_r/S}(1)$ can be constructed using a standard spreading-out argument (may be after replacing $S$ by an etale neighbourhood of the closed point of $S$). 
\begin{rema}
By Simpson's method \cite[Theorem 4.7]{32}, one can construct a relative moduli of Higgs bundles over the family of curves. To construct a total space and the GIT quotient, one must choose a relatively ample line bundle over the family $\mathcal X_r/S$. We choose the line bundle $\mathcal O_{\mathcal X_r/S}(1)$ \eqref{Lanebun} for this purpose. Then one can easily see from the definition that the relative moduli of Higgs bundle constructed using GIT with respect to this line bundle parametrises families of $\epsilon$-semistable Higgs bundles over the family $\mathcal X_r/S$.  
\end{rema}

\begin{prop}
There exists a family
\begin{equation}
\mathcal M^{\chi, n, \epsilon,ad}_{HB, \mathcal X_r}\rightarrow S
\end{equation}
of moduli of $\epsilon$-semistable Higgs bundles along the fibers of $\mathcal X_r/S$ with Euler characteristic $\chi$. Moreover, the morphism $\mathcal M^{\chi, n, \epsilon,ad}_{HB, \mathcal X_r}\rightarrow S$ is smooth. 
\end{prop}
\begin{proof}
We refer to \cite[Theorem 4.7]{32} for the construction of the family. The space of first order infinitesimal deformations of a Higgs bundle $(\mathcal E, \phi)$ is isomorphic to $\mathbb H^1(\mathcal C_{\bullet}(\mathcal E, \phi))$ and the space of obstructions to extend the Higgs bundle over a small thickenning is isomorphic to $\mathbb H^2(\mathcal C_{\bullet}(\mathcal E, \phi))$. Since $\epsilon $-stability implies $0$-stability, from \cite[Proposition 5.3]{3}, we have $dim~~\mathbb H^2(\mathcal C_{\bullet}(\mathcal E, \phi))=1$. Therefore the dimension of the relative tangent space i.e., $dim~~\mathbb H^1(\mathcal C_{\bullet})$ is constant and hence the morphism $\mathcal M^{\chi, n, \epsilon,ad}_{HB, \mathcal X_r}\rightarrow S$ is smooth.
\end{proof}

\begin{thm}\label{equisym}
There is a natural $Aut(X_r/X_0)$-equivariant symplectic form on $\mathcal M^{\chi, n, \epsilon, ad}_{HB, X_r}$. 
\end{thm}

\begin{proof}
As before ( \eqref{Pair} and \eqref{Pair2}), the following morphism of complexes induces a bi-linear pairing on the tangent space.

\begin{equation}
\mathcal C^{\vee}_{\bullet}(\mathcal E, \phi)\rightarrow \mathcal C_{\bullet}(\mathcal E, \phi)
\end{equation}

Skew-symmetricity and non-degeneracy of the above pairing follow from the description of the morphism of complexes. The closed-ness of the corresponding $2$-form follows from the fact that $\mathcal M^{\chi, n, \epsilon, ad}_{HB, X_r}$ is the closed fibre of the smooth family $\mathcal M^{\chi, n, \epsilon, ad}_{HB, \mathcal X_r}\rightarrow S$ and the fact that the above pairing is closed on the generic fibre.

Let $t\in Aut(X_r/X_0)$ be an automorphism $t:X_r\rightarrow X_r$. Then we have a commutative diagram of complexes
\begin{equation}
\begin{tikzcd}
\mathcal C^{\vee}_{\bullet}(\mathcal E, \phi)\arrow{r}\arrow{d}{t^*} & \mathcal C_{\bullet}(\mathcal E, \phi)\arrow{d}{t^*}\\
\mathcal C^{\vee}_{\bullet}(t^*\mathcal E, t^*\phi)\arrow{r} & \mathcal C_{\bullet}(t^*\mathcal E, t^*\phi)
\end{tikzcd} 
\end{equation}
The commutativity follows from the fact that $t^*[\phi, s]=[t^*\phi, t^*s]$. 

It induces the following commutative diagram of hypercohomologies
\begin{equation}
\begin{tikzcd}
\mathbb H^1(\mathcal C^{\vee}_{\bullet}(\mathcal E, \phi))\arrow{r}\arrow{d}{t^*} & \mathbb H^1(\mathcal C_{\bullet}(\mathcal E, \phi))\arrow{d}{t^*}\\
\mathbb H^1(\mathcal C^{\vee}_{\bullet}(t^*\mathcal E, t^*\phi))\arrow{r} & \mathbb H^1(\mathcal C_{\bullet}(t^*\mathcal E, t^*\phi))
\end{tikzcd} 
\end{equation}
Therefore, the symplectic form on $\mathcal M^{\chi, n, \epsilon, ad}_{HB, X_r}$ is $Aut(X_r/X_0)$-equivariant.
\end{proof}
\begin{coro}\label{TQ1}
The morphism $f_r: \mathcal M^{\chi, n, \epsilon, ad}_{HB, X_r}\rightarrow \mathcal M_{GHB}$ (remark \ref{rema450}) is a Poisson morphism.
\end{coro}
\begin{proof}
Using the descriptions (Theorem \ref{Tann} and Lemma \ref{Tanagain}) of the vector bundles in the following equations, we see that 
\begin{equation}
f^*_r\Omega_{\mathcal M_{GHB}}(log~~\partial \mathcal M_{GHB})\cong \Omega_{\mathcal M^{\chi, n, \epsilon, ad}_{HB, X_r}}, \hspace{.3cm} \text{and}~~f^*_rT_{\mathcal M_{GHB}}(-log~~\partial \mathcal M_{GHB})\cong T_{\mathcal M^{\chi, n, \epsilon, ad}_{HB, X_r}}
\end{equation}

The explicit descriptions given in \eqref{Pair}, \eqref{Pair2} and Theorem \ref{equisym}) of the morphisms $f^*_r\Omega_{\mathcal M_{GHB}}(log~~\partial \mathcal M_{GHB})\rightarrow f^*_r T_{\mathcal M_{GHB}}(log~~\partial \mathcal M_{GHB})$ and 
$\Omega_{\mathcal M^{\chi, n, \epsilon, ad}_{HB, X_r}}\rightarrow T_{\mathcal M^{\chi, n, \epsilon, ad}_{HB, X_r}}$ induced by the Poisson bi-vectors on $\mathcal M_{GHB}$ and $\mathcal M^{\chi, n, \epsilon, ad}_{HB, X_r}$, respectively clearly match at every point. Therefore the corollary follows.

\end{proof}

\section{\textbf{Stratification of $\mathcal M_{GHB}$ by Poisson ranks}}

Let us recall that we had chosen a degeneration of a smooth projective curve i.e., a family of curves $\mathcal X$ over a discrete valuation ring $S$ (\ref{DegeCourbes}). Then one can construct a family of varieties $\mathcal M_{GHB, S}$ over $S$ such that the fibre over the generic point is the moduli of Higgs bundles over the generic curve and the fibre over the closed point is the moduli of Gieseker-Higgs bundles on the nodal curve. Moreover, the closed fibre is a normal-crossing divisor in $\mathcal M_{GHB, S}$. It has a natural stratification given by its successive singular loci
\begin{equation}\label{strat2055}
\mathcal M_{GHB}\supset \partial^1\mathcal M_{GHB}\supset \dots\supset \partial^n\mathcal M_{GHB} \supset \partial^{n+1}\mathcal M_{GHB}:=\emptyset,
\end{equation}

By \cite[Lemma 3.1]{5}, the stratification has the following description.
\[
\text{ for every }~~0\leq r\leq n, ~~\partial^r\mathcal M_{GHB}=\{x\in \mathcal M_{GHB}~~| ~~\text{cardinality of the set }~~q^{-1}(x)\geq r+1\},
\]
where $q$ denotes the normalisation $\widetilde{\mathcal M}_{GHB}\rightarrow \mathcal M_{GHB}$.

\begin{prop}\label{Strat}
\begin{enumerate}
\item For every integer $0 \leq r\leq n$, $\partial^r\mathcal M_{GHB}$ is a closed Poisson sub-variety $\mathcal M_{GHB}$. The closed points of $\partial^r \mathcal M_{GHB}$ correspond to the equivalence classes of stable Gieseker-Higgs bundles $(X_k, \mathcal E, \phi)$, where $n\geq k\geq r$. 

\item the $r$-th stratum $\partial^{r,o}\mathcal M_{GHB}:=\partial^r\mathcal M_{GHB}\setminus \partial^{r+1}\mathcal M_{GHB}$ is a smooth locally-closed Poisson sub-scheme of $\mathcal M_{GHB}$. 

\item the most singular locus $\partial^n\mathcal M_{GHB}$ is a smooth Poisson variety of dimension $2n^2(g-1)+2-n$, whose closed points correspond to the equivalence classes of stable Gieseker-Higgs bundles $(X_n, \mathcal E,\phi)$ of rank $n$ and degree $d$. 
\end{enumerate}
\end{prop}
\begin{proof}
From \cite[Corollary 2.4]{29}, it follows that for every $r$, the variety $\partial^i\mathcal M_{GHB}$ is a closed Poisson subvariety of $\mathcal M_{GHB}$ of dimension $2(n^2(g-1)+1)-r$. In particular, $\partial^{r,o}\mathcal M_{GHB}$ is a smooth locally-closed Poisson subvariety.

There is a universal curve $\mathcal X^{univ}$ over $\mathcal M_{GHB}$, which is the restriction of the universal curve $\mathcal X^{univ}_S$. We define 
\begin{equation}
D:=V(Fitt^1 \Omega_{\mathcal X^{univ}/\mathcal M_{GHB}}),
\end{equation} 
where $Fitt^1 \Omega_{\mathcal X^{univ}/\mathcal M_{GHB}}$ denotes the first Fitting ideal of $\Omega_{\mathcal X^{univ}/\mathcal M_{GHB}}$.

\begin{claim}
$D$ is the normalisation of $\mathcal M_{GHB}$. 
\end{claim}

\begin{claimproof}
Since the fibres of $D\rightarrow \mathcal M_{GHB}$ are the singular locus of the morphism $\mathcal X^{univ}\rightarrow \mathcal M_{GHB}$, we see that the earlier map is finite, birational and surjective. So if we show that $D$ is smooth, then it follows that it is the normalisation of $\mathcal M_{GHB}$. 

The question is local. So, let us concentrate around a point $p$ representing an equivalence class of Gieseker-Higgs bundle $(X_r, \mathcal E, \phi)\in \mathcal M_{GHB}$, as in the proof of Proposition \ref{logs}. The Henselian local ring of $\mathcal M_{GHB}$ at $p$ is $A_0$ (see proposition \ref{logs}), whose local components $\{D_i\}^{r+1}_{i=1}$ are given by 

$$\spec A_{0,i}:=\spec \frac{A_0}{(t_i)} ~~\text{for}~~ i=1,\dots, r+1.$$

 Moreover, if $p_i$ denotes the $i$-th node of $X_r$, then 
 
 $$\mathcal O_{\mathcal X^{univ}_{A_0}, p_i}\cong \frac{A_0[x, y]}{xy-t_i}~~\text{ for every}~~ i=1,\dots, r+1.$$ 
 
Therefore, using the description of $A_0$ ((3) in the proof of Proposition \ref{logs}) we have 

$$\mathcal O_{D_i, p_i}\cong A_{0,i}, ~~\text{and}~~ D_i ~~\text{is smooth for every} ~~i=1,\dots, r+1.$$ 

This proves that the normalisation is isomorphic to the vanishing locus of the first Fitting ideal. 
\end{claimproof}

Since $\widetilde{\mathcal M}_{GHB}$ is the vanishing locus of the first Fitting ideal $Fitt^1 (\Omega_{\mathcal X^{univ}/\mathcal M_{GHB}})$, the fibre of the normalisation $\widetilde{\mathcal M}_{GHB}\rightarrow \mathcal M_{GHB}$ over a point $(X_r, \mathcal E,\phi)$ is $\{(X_r, \mathcal E,\phi, x)| x ~~\text{is a node of the curve}~~ X_r\}$.
\end{proof}

Since $\partial^{r,o}\mathcal M_{GHB}$ is a smooth locally closed Poisson sub-scheme, the Poisson bi-vector $\sigma$ induces a morphism $\sigma^{\flat}_r: \Omega_{\partial^{r,o}\mathcal M_{GHB}}\rightarrow T_{\partial^{r,o}\mathcal M_{GHB}}$. To compute the Poisson rank of $\sigma$ at a point of this stratum it is enough to compute the rank of the morphism $\sigma^{\flat}_r$ (see Example \ref{exam12}). Before computing the Poisson ranks, we need a preliminary lemma \ref{rank}.
Let us denote the torus $Aut(X_r/X_0)$ by $A_r$, for convenience. Consider the principal $A_r$-bundle $\mathcal M^{\chi, n, \epsilon, ad}_{HB, X_r}\rightarrow \partial^{r,o} \mathcal M_{GHB}$. 

Let $\tilde p$ denote an isomorphism class of a stable Gieseker-Higgs bundle $(X_k, \mathcal E,\phi)\in \mathcal M^{\chi, n, \epsilon, ad}_{HB, X_r}$ and $p$ denote the image of $\tilde p$ i.e., the Gieseker-equivalent class of $(X_k, E,\phi)$. Then we have the following diagram

\begin{equation}\label{B}
\begin{tikzcd}
0\arrow{r} & \Omega_{\partial^{r,o} \mathcal M_{GHB}, p} \arrow[dotted]{d}{\sigma^{\flat}_r} \arrow{r}{u}& \Omega_{\mathcal M^{\chi, n, \epsilon, ad}_{HB, X_r}, \tilde p} \cong \mathbb H^1(\mathcal C^{\vee}_{\bullet})\arrow[bend right=50]{d}{\sigma^{\flat}}\arrow{r}{j}& \Omega_{A_r, e} \cong H^0(X_r, T_{X_r})^{\vee}\arrow{r} & 0\\
0 & T_{\partial^{r,o} \mathcal M_{GHB}, p} \arrow{l}& T_{\mathcal M^{\chi, n, \epsilon, ad}_{HB, X_r}, \tilde p}\cong \mathbb H^1(\mathcal C_{\bullet})\arrow{l}{v}\arrow[bend right=50]{u}{\omega^{\#}} & T_{A_r, e} \cong H^0(X_r, T_{X_r})\arrow{l}{i}\arrow[dotted]{u}{B_{r}} & 0\arrow{l}
\end{tikzcd}
\end{equation}

Notice that $v\circ \sigma^{\flat}\circ u=\sigma^{\flat}_r$ and $j\circ \omega^{\#}\circ i= B_r$.

\begin{lema}\label{rank} Let $(X_r, \mathcal E, \phi)$ be a stable Gieseker-Higgs bundle. Let $X_r=\cup_{i\in \Lambda} U_i$ be an open cover of $X_r$ such that the vector bundle $\mathcal E$ and $\omega_{X_r}$ are trivial over every $U_i$. Let us denote the co-cycle (with respect to the cover $\{U_i\}$) of $\mathcal E$ by $\{ A_{ij} \}$. Let $\{\phi_i\}_{i\in \Lambda}$ denote the collection of Higgs fields on the cover $\{U_i\}_{i\in \Lambda}$ which glue to give the global Higgs field $\phi$. Then
\begin{enumerate}
\item the morphism $i$ is given by $i(\psi)=([X_{\psi}, \phi_i], [X_{\psi}, A_{ij}])$.
\item the morphism $\omega^{\#}$ is given by $\omega^{\#}(\{\alpha_i\},\{\eta_{ij}\})=(\{-\alpha_i\}, \{\eta_{ij}\})$.
\item the composite $j\circ \omega^{\#}\circ i=0$.
\end{enumerate}
\end{lema}

\begin{proof}
\textsf{proof of (1).} Let $\psi \in H^0(X_r, T_{X_r})$. It is, by definition, an isomorphism
$$
\psi: X_r[\epsilon]\rightarrow X_r[\epsilon]
$$
which on the sheaf of rings can be described as follows.
\begin{equation}
\psi^{\#}:\mathcal O_{X_r}[\epsilon]\rightarrow \mathcal O_{X_r}[\epsilon]
\end{equation}
\hspace{3.5cm}given by \hspace{1cm} $f+\epsilon \cdot g\mapsto f+\epsilon\cdot (g+X_{\psi}(df))$,

where $X_{\psi}$ is the vector field on $X_r$ corresponding to the infinitesimal automorphism $\psi$.

We denote by $\mathcal E[\epsilon]$ the trivial deformation of $\mathcal E$ over $X_{r}[\epsilon]$. We want to write the co-cycle of $\psi^*\mathcal E[\epsilon]$ in terms of the cocycle of $\mathcal E$. The transition functions of $\psi^*(\mathcal E[\epsilon])$ are given by $A_{ij}+\epsilon B_{ij}$ for some $B_{ij}$ which fits into the following commutative diagram.

\begin{equation}
\begin{tikzcd}
(\mathcal O_{U_{ij}}[\epsilon])^{\oplus n}\arrow{r}{A_{ij}+\epsilon \cdot 0}\arrow{d}{(1+\epsilon X_{\psi})} & (\mathcal O_{U_{ij}}[\epsilon])^{\oplus n} \arrow{d}{(1+\epsilon X_{\psi})}\\
(\mathcal O_{U_{ij}}[\epsilon])^{\oplus n}\arrow[dotted]{r}{A_{ij}+\epsilon B_{ij}} & (\mathcal O_{U_{ij}}[\epsilon])^{\oplus n}
\end{tikzcd}
\end{equation}

where the map 
\begin{equation}
(\mathcal O_{U_{ij}}[\epsilon])^{\oplus n}\xrightarrow{(1+\epsilon X_{\psi})} (\mathcal O_{U_{ij}}[\epsilon])^{\oplus n}
\end{equation}
is given by 
$$
\{(f_i+\epsilon g_i)\}^n_{i=1}\mapsto \{(1+\epsilon X_{\psi})(f_i+\epsilon g_i)\}^n_{i=1}. 
$$
 
Since the above diagram commutes we have
\begin{align*} 
&(A_{ij}+\epsilon B_{ij})\circ (1+\epsilon X_{\psi})=(1+\epsilon X_{\psi})\circ A_{ij}\\
& \implies B_{ij}=[X_{\psi}, A_{ij}]
\end{align*}

It can be easily checked that $B_{ik}=A_{ij}B_{jk}+B_{ij}A_{jk}$ for any $i,j,k$ and hence $\{B_{ij}\}$ defines an element of $H^1(\mathcal End \mathcal E)$.

Consider the Higgs field $\phi+\epsilon \cdot 0: \mathcal E[\epsilon]\rightarrow \mathcal E[\epsilon]\otimes \omega_{X_r}$ over $X_r[\epsilon]$. It can be expressed as Higgs fields $\{\phi_i\}$ over each $U_i$ satisfying the following

\begin{equation}\label{comm121}
A_{ij}\phi_j A_{ij}^{-1}=\phi_i, ~~\forall i,j.
\end{equation}

Similarly, the Higgs field $\psi^*\phi$ can be expressed as $\{\phi_i+\epsilon \phi'_i\}$ which fits into the following commutative diagram.
  
\begin{equation}
\begin{tikzcd}
\mathcal E[\epsilon]|_{U_i[\epsilon]} \arrow{r}{\phi_{i}+\epsilon \cdot 0}\arrow{d}{(1+\epsilon X_{\psi})\otimes 1} & ((\mathcal E\otimes \omega_{X_r})[\epsilon])|_{U_{i}[\epsilon]} \arrow{d}{(1+\epsilon X_{\psi})\otimes 1}\\
(\psi^*\mathcal E[\epsilon])|_{U_i[\epsilon]}\arrow[dotted]{r}{\phi_{i}+\epsilon \phi'_i} & ((\psi^*\mathcal E[\epsilon])\otimes \omega_{X_r}[\epsilon])|_{U_{i}[\epsilon]}
\end{tikzcd}
\end{equation}

Since we have $(\phi_i+\epsilon \phi'_i)\circ ((1+\epsilon X_{\psi})\otimes 1)=((1+\epsilon X_{\psi})\otimes 1)\circ \phi_i $, it follows that $\phi'_i=[X_{\psi}, \phi_i]$. It can be easily verified that $\phi'_iA_{ij}-A_{ij}\phi'_j=B_{ij}\phi_j -\phi_iB_{ij}$ for all $i,j$. Hence $(\{\phi'_j\}, \{B_{ij}\})$ defines an element of $\mathbb H^1(\mathcal C_{\bullet})$.

Therefore, we conclude that $i(\psi)=(\{[X_{\psi}, \phi_j]\}, \{[X_{\psi}, A_{ij}]\})$.

\textsf{proof of (2).} From the description of the morphism of complexes $\mathcal C^{\vee}_{\bullet}\rightarrow \mathcal C_{\bullet}$, we see that $\omega^{\#}(\{\alpha_j\},\{\eta_{ij}\})=(\{-\alpha_j\}, \{\eta_{ij}\})$.

\textsf{proof of (3).} We have $\omega^{\#}\circ i(\psi)=(\{-[X_{\psi}, \phi_j]\}, \{[X_{\psi}, A_{ij}]\})$. Since the map $j$ is just the dual of the morphism $i$, we have 
\begin{align*}
(j\circ \omega^{\#}\circ i)(\psi)(\psi')
&=Trace([X_{\psi}, A_{ij}]\circ [X_{\psi'}, \phi_j]-[X_{\psi}, \phi_i]\circ [X_{\psi'}, A_{ij}]) ~\hspace{2cm}~~(\eqref{Pyaar1}, \eqref{Pyaar2}, ~~\text{and remark}~~ \ref{Liou441})
\end{align*}
The vector fields $X_{\psi}$ and $X_{\psi'}$ are elements of $H^0(X_r, T_{X_r})$. Since $X_0$ is a stable curve, the vector fields have support only along $R[r]$, the chain of $\mathbb P^1$'s. Moreover, it is not difficult to see that 
\begin{equation}
H^0(X_r, T_{X_r})\cong \oplus^r_{i=1} H^0(R[r]_i, T_{R[r]_i}(-x^+_i-x^-_i)),
\end{equation}
where $x^+_i$ and $x^-_i$ denote the two nodes of $R[r]_i$. Therefore, we see that the vector fields $X_{\psi}$ and $X_{\psi'}$ have support on $R[r]$ and they vanish at the nodes. Moreover, over any particular $\mathbb P^1$, say $R[r]_i$, two such vector fields only differ by a scalar. In order to compute the term $Trace([X_{\psi}, A_{ij}]\circ [X_{\psi'}, \phi_j]-[X_{\psi}, \phi_i]\circ [X_{\psi'}, A_{ij}])$ we will  concentrate and compute it on every $\mathbb P^1$.  

Let us concentrate on $R[r]_k$ for any $k$. Let us denote by $x^+_k$ and $x^-_k$ the two nodes on $R[r]_k$. Then $R[r]_k=(R[r]_k\setminus x^+_k)\cup (R[r]_k\setminus x^+_k)$ is an open cover. Let $A$ denote the transition function $R[r]_k\setminus \{x^+_k, x^-_k\}\rightarrow GL_n$, and $\phi^+_k$ and $\phi^-_k$ denote the Higgs fields on $R[r]_k\setminus \{x^-\}$ and $R[r]_k\setminus \{x^+\}$, respectively corresponding to the Higgs field $\phi|_{R[r]_k}$. From \eqref{comm121}, it follows that 
\begin{equation}\label{comm212}
A\circ \phi^+_k=\phi^-_k\circ A.
\end{equation}

We recall that $\mathcal E|_{R[r]_k}\cong \mathcal O(1)^{a_k}\oplus \mathcal O^{b_k}$, for some positive integer $a_k$ and some non-negative integer $b_k$ such that $a_k+b_k=n$. Therefore the matrix function $A$ has the following form 

\begin{equation}
z\mapsto \left[ 
\begin{array}{c|c} 
 B & C \\ 
 \hline 
 0 & D 
\end{array} 
\right], 
\end{equation}

where $B=\frac{1}{z}\cdot I_{a_k}$ and $D=I_{b_k}$. Similarly, $\phi^{\pm}_k$ is also of the following form 

\begin{equation}
\left[ 
\begin{array}{c|c} 
 \phi^{\pm}_1 & \phi^{\pm}_3 \\ 
 \hline 
 0 & \phi^{\pm}_2 
\end{array} 
\right]. 
\end{equation}

We want to compute $Trace([X_{\psi}, A]\circ [X_{\psi}, \phi^+_k]-[X_{\psi}, \phi^-_k]\circ [X_{\psi}, A])$. It follows from the description above that 

\begin{equation}
Trace([X_{\psi}, A]\circ [X_{\psi}, \phi^+_k]-[X_{\psi}, \phi^-_k]\circ [X_{\psi}, A])=Trace([X_{\psi}, B]\circ [X_{\psi}, \phi^+_1]-[X_{\psi}, \phi^-_1]\circ [X_{\psi}, B])
\end{equation}

Now let $\overrightarrow{f}\in \Gamma(R[r]_k\setminus \{x^+_k, x^-_k\}, \mathcal O^{n}_{R[r]_k})$. Then 
\begin{equation}
[X_{\psi}, B](\overrightarrow{f})=(X_{\psi}\circ B)(\overrightarrow{f})-(B\circ X_{\psi})(\overrightarrow{f})=(X_{\psi}B)(\overrightarrow{f})+B(X_{\psi}\overrightarrow{f})-B(X_{\psi}\overrightarrow{f})=(X_{\psi}B)(\overrightarrow{f}).
\end{equation} 

Similarly, 
\begin{equation}
[X_{\psi}, \phi^{\pm}_1]=X_{\psi}\phi^{\pm}_1.
\end{equation}

Therefore, 
\begin{align*} 
&Trace([X_{\psi}, B]\circ [X_{\psi}, \phi^+_1]-[X_{\psi}, \phi^-_1]\circ [X_{\psi}, B])\\
&=Trace((X_{\psi}B)\circ (X_{\psi} \phi^+_1)-(X_{\psi} \phi^-_1) \circ(X_{\psi}, B))\\
&=\frac{1}{rank~~B}\cdot Trace(X_{\psi}B)\cdot Trace(X_{\psi} \phi^+_1-X_{\psi} \phi^-_1) ~~~~~~~~(\text{Since}~~ X_{\psi}B~~\text{is a diagonal matrix})
\end{align*} 

Now from \eqref{comm212}, we have $\phi^+_k=A^{-1}\circ \phi^-_k\circ A$. Therefore, $Trace(\phi^+_k)=Trace(\phi^-_k)$ and $Trace(\phi^+_1)=Trace(\phi^-_1)$. Since $X_{\psi}(Trace(\phi^{\pm}_1))=Trace(X_{\psi}\phi^{\pm}_1)$, we have 

\begin{equation}
Trace(X_{\psi}\phi^+_1)=Trace(X_{\psi}\phi^-_1).
\end{equation}

Therefore, 
\begin{align*} 
&Trace([X_{\psi}, B]\circ [X_{\psi}, \phi^+_1]-[X_{\psi}, \phi^-_1]\circ [X_{\psi}, B])\\
&=\frac{1}{rank~~B}\cdot Trace(X_{\psi}B)\cdot Trace(X_{\psi} \phi^+_1-X_{\psi} \phi^-_1)\\
&=0
\end{align*} 

Therefore, we conclude that $(j\circ \omega^{\#}\circ i)(\psi)(\psi')=0$.
\end{proof}

\begin{rema}
Notice that the statement (1) in lemma \ref{rank} gives an alternative proof of lemma \ref{Tesimal}.
\end{rema}

\begin{thm}\label{Foliation11}
The stratification of the Poisson variety $\mathcal M_{GHB}$ given by the successive degeneracy loci of the Poisson structure \eqref{Pstrat} is the same as the stratification given by the successive singular loci \eqref{strat2055}. Moreover, $\partial^r \mathcal M_{GHB}\setminus \partial^{r+1}\mathcal M_{GHB}$ is a smooth Poisson subvariety of dimension $2(n^2(g-1)+1)-r$ with constant Poisson rank $2(n^2(g-1)+1)-2r$. In particular, the most singular locus $\partial^n \mathcal M_{GHB}$ is a smooth Poisson variety of dimension $2(n^2(g-1)+1)-n$ with constant Poisson rank $2(n^2(g-1)+1)-2n$.
\end{thm}
\begin{proof}
Follows from the fact that $B_r=0$ for all $r=0,\dots, n$.
\end{proof}

\begin{rema}\label{Sepa}
Let us denote by $\mathcal M^{\epsilon_{\bullet}}_{PVB}$ the moduli of the parabolic vector bundles of rank $n$ and degree $d-n$ over $\tilde{X_0}$ with full-flags at the two points $x^+$ and $x^-$ and semi-stable with respect to sufficiently small and generic choice of parabolic weights $\epsilon_{\bullet}$. In \cite{5}, it is shown that the most singular locus $\partial^n \mathcal M_{GVB}$ is isomorphic to $\mathcal M^{\epsilon_{\bullet}}_{PVB}$. Let us denote by $\mathcal M^{\epsilon_{\bullet}}_{PHB}$ the moduli of non-strongly parabolic-Higgs bundles over the curve $\tilde{X_0}$ of rank $n$ and degree $d-n$ with full-flagged parabolic structures at the two pre-images $x^+$ and $x^-$ and with sufficiently small and generic choice of parabolic weights $\epsilon_{\bullet}$. One can show that the most singular locus $\partial^n \mathcal M_{GHB}$ is isomorphic to the closed Poisson sub-scheme of $\mathcal M^{\epsilon_{\bullet}}_{PHB}$ consisting of non-strongly parabolic-Higgs bundles whose eigenvalues of the Higgs field at the two points $x^+$ and $x^-$ are the same. We will discuss this in a separate note.  
\end{rema}

\subsection{\textbf{The induced Poisson structure on the normalisation of $\mathcal M_{GHB}$}} Consider the normalization $q:\widetilde{\mathcal M}_{GHB}\rightarrow \mathcal M_{GHB}$. It is a smooth variety with normal-crossing divisor $q^{-1}(\partial \mathcal M_{GHB})$. The pullback of the log-symplectic form induces a log-symplectic structure on $\widetilde{\mathcal M}_{GHB}$. The variety $\widetilde{\mathcal M}_{GHB}$ has the following stratification.

\begin{equation}
\widetilde{\mathcal M}_{GHB}\supset \partial \widetilde{\mathcal M}_{GHB}:=q^{-1}(\partial \mathcal M_{GHB})\supset \partial^2\widetilde{\mathcal M}_{GHB}:=q^{-1}(\partial^2\mathcal M_{GHB})\supset \cdots
\end{equation}
It is straightforward to check that $\partial^{r+1}\widetilde{\mathcal M}_{GHB}$ is the singular locus of $\partial^r\widetilde{\mathcal M}_{GHB}$ for every $r\geq 1$. The sub-scheme $\partial^{r,o} \mathcal M_{GHB}$ is locally an intersection of $r$ connected components of $\mathcal M_{GHB}$. Therefore, locally the inverse image of $\partial^{r,o}\mathcal M_{GHB}$ is the disjoint union of $r$ sub-varieties, each of which is isomorphic (locally) to $\partial^{r,o}\mathcal M_{GHB}$. Therefore, we see that the Poisson rank at a point of the strata $\partial^{r,o}\widetilde{\mathcal M}_{GHB}(:=\partial^{r}\widetilde{\mathcal M}_{GHB}\setminus \partial^{r+1}\widetilde{\mathcal M}_{GHB})$ is the same as the Poisson rank at the image of this point under the normalisation map. We can compute the Poisson rank at a point using the same diagram \eqref{B}. 

\begin{rema} By lemma \ref{rank}, it follows that the bi-residues/magnetic terms \cite[Example 3.3.]{24} is $0$ for every stratum of the normal crossing divisor $\partial \widetilde{\mathcal M_{GHB}}$. Therefore, by \cite[proposition 3.6]{24}, it follows that the local normal form of the Poisson structure $\widetilde{\mathcal M_{GHB}}$ is stably equivalent to

\begin{equation}
\hspace{1cm} \omega=\sum^k_{j=1} dp_j\wedge \frac{dy_j}{y_j}, ~~~\text{for some integer}~~ k, 
\end{equation}

as in the example \ref{cotex}. 
\end{rema}

\section{\textbf{Description of the symplectic foliation}}

We recall from \S 6 that the moduli space $\mathcal M^{\chi, n,\epsilon, ad}_{HB, X_r}$ of $\epsilon$ stable admissible Higgs bundles of rank $n$ and Euler characteristic $\chi$ on the curve $X_r$ is a smooth variety. The torus $A_r:=Aut(X_r/X_0)$ acts freely on $\mathcal M^{\chi, n,\epsilon, ad}_{HB, X_r}$. The quotient is isomorphic to $\partial^{r,o}\mathcal M_{GHB}:=\partial^r \mathcal M_{GHB}\setminus \partial^{r+1}\mathcal M_{GHB}$. From theorem \ref{equisym}, it follows that $\mathcal M^{\chi, n,\epsilon, ad}_{HB, X_r}$ is a $A_r$-symplectic manifold. We will see that the action is Hamiltonian i.e., it has a momentum map. The Hamiltonian action of an algebraic group and the momentum map can be defined in the algebraic setting. We refer to \cite{L}, \cite{L I}, and \cite[Chapter II, \S 1, \S 2]{V II} for the details. Before describing the momentum map notice that $T_{A_r,e} \cong H^0(X_r, T_{X_r})\cong \oplus ^r_{i=1} H^0(R[r]_i, \mathcal O_{R[r]_i})$. 

\begin{thm}\label{Casimir}
\begin{enumerate}
\item The map 
\begin{equation}
\mu_r: \mathcal M^{\chi, n,\epsilon, ad}_{HB, X_r}\rightarrow (T_{A_r,e})^{\vee} 
\end{equation}
defined by 
\begin{equation}
\mu_r(\mathcal E, \phi)(X_{\psi})=\lambda(\imath(X_{\psi}))=Trace(\phi\circ \imath(X_{\psi})), ~~\text{for}~~X_{\psi}\in H^0(X_r, T_{X_r})
\end{equation}
is a momentum map, where 
\begin{enumerate}
\item $\imath: H^0(X_r, T_{X_r})\hookrightarrow \mathbb H^1(\mathcal C_{\bullet})$ denotes the differential of the orbit map $A_r\rightarrow \mathcal M^{\chi, n,\epsilon, ad}_{HB, X_r}$ at the point $(\mathcal E,\phi)$.

\item $\lambda$ denotes the symplectic potential on $\mathcal M^{\chi, n,\epsilon, ad}_{HB, X_r}$ (\ref{Liou440} and remark \ref{Liou441}).
\end{enumerate} 

\item $\mu_r(\mathcal E,\phi)=(Trace ~~\phi|_{\mathcal O_{R[r]_1}(1)^{\oplus a_1}}, \dots, Trace ~~\phi|_{\mathcal O_{R[r]_r}(1)^{\oplus a_r}})$, where $\mathcal E|_{R[r]_i}\cong \mathcal O_{R[r]_i}(1)^{\oplus a_i}\oplus \mathcal O_{R[r]_i}^{\oplus b_i}$ for every $i=1,\dots, r$. 

\item The coordinate functions of $\mu_r$ are the Casimir functions of $\mathcal M^{a_{\bullet}}_{GHB}$ \eqref{admi445}. In particular, the variety $\mu^{-1}_r(0)\cap \mathcal M^{a_{\bullet}}_{GHB}$ is a symplectic leaf of $\mathcal M^{a_{\bullet}}_{GHB}$ containing $\Omega_{\mathcal M^{a_{\bullet}}_{GVB}}$. Moreover, it consists of triples $(X_r, \mathcal E, \phi)$ such that the trace of $\phi|_{\mathcal O_{R[r]_i} (1)^{\oplus a_i}}: \mathcal O_{R[r]_i} (1)^{\oplus a_i}\rightarrow \mathcal O_{R[r]_i} (1)^{\oplus a_i}$ is zero for all $i=1,\dots , r$.
\end{enumerate}
\end{thm}
\begin{proof}
\textsf{proof of (1).} Similar to the case of a smooth curve, the symplectic form on $\mathcal M^{\chi, n,\epsilon, ad}_{HB, X_r}$ is exact, i.e., there exists a $1$-form $\lambda$ such that the symplectic form is given by $-d\lambda$. In the literature, such a form is called the symplectic potential. In our case, the symplectic potential is an extension of the Liouville $1$-form, as in the case of smooth curves. The symplectic potential form $\lambda$ can be described similarly as in \ref{Liou440} (remark \ref{Liou441}). 

It is well-known that any $G$-variety $Z$ equipped with an equivariant symplectic potential $\lambda$ has an equivariant momentum map.

\begin{equation}
\mu: Z\rightarrow \mathfrak g^{\vee},~~~\mathfrak g:=\text{Lie algebra of}~~G
\end{equation}
given by 
$$
\mu(z)(\xi)=\lambda(z)(\imath(\xi)),
$$
where 
\begin{enumerate}
\item $z\in Z$, and $\xi\in \mathfrak g$,
\item $\imath : \mathfrak g\rightarrow T_z Z$, denotes the differential of the orbit map at the point $z$.
\end{enumerate} 

Therefore, it is enough to check that the symplectic potential in our case is preserved by the action of the torus $A_r$. From \ref{Liou440}, it follows that 
\begin{equation}
\lambda(\mathcal E,\phi)(\{s_{ij}\}, \{t_j\})=\{Trace(\phi\circ s_{ij})\},
\end{equation}
where
\begin{enumerate}
\item $(\mathcal E, \phi)\in \mathcal M^{\chi, n,\epsilon, ad}_{HB, X_r}$, and 
\item $(\{s_{ij}\}, \{t_i\})\in \mathbb H^1(\mathcal C_{\bullet})$, the tangent space of $\mathcal M^{\chi, n,\epsilon, ad}_{HB, X_r}$ at the point $(\mathcal E, \phi)$ (remark \ref{co}).
\end{enumerate}
Therefore, we see that
\begin{align*}
(f^*\lambda)(\mathcal E,\phi)(s_{ij}, t_i)
&=Trace(f^{\#}\circ \phi \circ (f^{\#})^{-1}\circ f^{\#}\circ s_{ij}\circ(f^{\#})^{-1})\\
&=Trace(\phi\circ s_{ij})\\
&=\lambda(\mathcal E, \phi)(s_{ij}, t_i),
\end{align*}
where $f\in A_r$ is an automorphism of $X_r$, and $f^{\#}$ denotes the induced morphism $f^*\mathcal E\rightarrow \mathcal E$. This completes the proof of (1).

\textsf{proof of $(2)$ and $(3)$.} Since $\mu_r$ is $A_r$-invariant map, it descends to $\mathcal M^{a_{\bullet}}_{GHB}$. 
\begin{claim}
The morphism $\mu_r$ is a smooth morphism. 
\end{claim}
\begin{claimproof}
To prove this, it is enough to show that the morphism $d\mu_r: T_{\mathcal M^{\chi, n,\epsilon, ad}_{HB, X_r}}\rightarrow (T_{A_r,e})^{\vee} 
$ is surjective at every point of $\mathcal M^{\chi, n,\epsilon, ad}_{HB, X_r}$. Let $(\mathcal E, \phi)$ be a point in $\mathcal M^{\chi, n,\epsilon, ad}_{HB, X_r}$. Recall that the tangent space of $\mathcal M^{\chi, n,\epsilon, ad}_{HB, X_r}$ is isomorphic to $\mathbb H^1(\mathcal C_{\bullet})$. Since $\mu_r$ is a momentum map, the morphism
\begin{equation}
d\mu_r: \mathbb H^1(\mathcal C_{\bullet})\rightarrow T_{A_r,e}^{\vee}
\end{equation}
is the same as $j\circ \omega^{\#}$ (diagram \ref{B}). By lemma \ref{Tesimal} and remark \ref{Free}, the morphism $j$ is surjective. Since $\omega^{\#}$ is an isomorphism, the morphism $d\mu_r$ is also surjective. Hence, $\mu_r$ is a smooth morphism.
\end{claimproof}

Therefore, $\mu_r^{-1}(0)\cap \mathcal M^{a_{\bullet}}_{GHB}$ is a symplectic leaf of $\mathcal M^{a_{\bullet}}_{GHB}$. Notice that $\mu_r$ is the quotient map in the following short exact sequence  

\begin{equation}
0\rightarrow \pi^*\Omega_{\mathcal M^{a_{\bullet}}_{GVB}}\rightarrow \Omega_{\mathcal M^{\chi, n, \epsilon, a_{\bullet}}_{VB, X_r}}\rightarrow \mathcal O_{\mathcal M^{\chi, n, \epsilon, a_{\bullet}}_{VB, X_r}}\otimes \Omega_e A_r\rightarrow 0
\end{equation}

Therefore, it follows that $\mu_r^{-1}(0)\cap \mathcal M^{a_{\bullet}}_{GHB}$ contains $\Omega_{\mathcal M^{a_{\bullet}}_{GVB}}$. Here, by abuse of notation, we denote both the sheaf and its total space by the same notation $\Omega_{\mathcal M^{a_{\bullet}}_{GVB}}$.

Now let $(\mathcal E, \phi)$ be a Gieseker-Higgs bundle on $X_r$ and $X_{\psi}$ be an element of $H^0(R[r]_i, \omega_{X_r}^{\vee})$. Let us denote by $x^+_i$ and $x^-_i$ the two nodes on $R[r]_i$. Then $R[r]_i=(R[r]_i\setminus x^+_i)\cup (R[r]_i\setminus x^+_i)$ is an open cover. Let $A$ denote the transition function $R[r]_i\setminus \{x^+_i, x^-_i\}\rightarrow GL_n$. We recall that $\mathcal E|_{R[r]_i}\cong \mathcal O(1)^{a_i}\oplus \mathcal O^{b_i}$, for some positive integer $a_i$ and some non-negative integer $b_i$ such that $a_i+b_i=n$. Therefore the matrix function $A$ has the following form 

\begin{equation}
z\mapsto \left[ 
\begin{array}{c|c} 
 B & C \\ 
 \hline 
 0 & D 
\end{array} 
\right], 
\end{equation}

where $B=\frac{1}{z}\cdot I_{a_i}$ and $D=I_{b_i}$. Similarly, $\phi$ is also of the following form 

\begin{equation}
\left[ 
\begin{array}{c|c} 
 \phi_1 & \phi_3 \\ 
 \hline 
 0 & \phi_2 
\end{array} 
\right]. 
\end{equation}

We easily see that 
\begin{equation}
Trace(\phi\circ[X_{\psi}, A])=Trace(\phi_1\circ [X_{\psi}, B]).
\end{equation}
Let $\overrightarrow{f}\in \Gamma(R[r]_i\setminus \{x^+_i, x^-_i\}, \mathcal O^{n}_{R[r]_i})$. Then 
\begin{equation}
[X_{\psi}, B](\overrightarrow{f})=X_{\psi}(\frac{1}{z}\cdot \overrightarrow{f})-\frac{1}{z}\cdot X_{\psi}(\overrightarrow{f})=\overrightarrow{f}\cdot X_{\psi}(\frac{1}{z}).
\end{equation} 

Also, notice that $X_{\psi}(z)$ is some scalar multiple of $z\cdot \frac{d}{dz}$. Therefore, 
\begin{equation}
Trace(\phi_1\circ [X_{\psi}, B])=Trace(\phi_1\circ (-\frac{1}{z}\cdot I))=-\frac{1}{z}\cdot Trace(\phi_1).
\end{equation}

Now using the identification $\Omega_{R[r]_i}(x^+_i+x^-_i)\cong \mathcal O_{R[r]_i}$, we can identify $-\frac{1}{z}\cdot Trace(\phi_1)$ with $Trace(\phi_1)$. Hence the $i$-th component of $\mu_r(\mathcal E, \phi)$ is $Trace~~\phi|_{\mathcal O_{R[r]_i}(1)^{\oplus a_i}}$.
\end{proof}

\begin{rema}
Let us denote by $\mathcal M^{\epsilon_{\bullet}}_{SPHB}$ the closed subscheme of $\mathcal M^{\epsilon_{\bullet}}_{PHB}$ (remark \ref{Sepa}) consisting of parabolic Higgs bundles whose eigen-values of the Higgs field at the two points $x^+$ and $x^-$ are all $0$. It follows from proposition \ref{Casimir} that $\mathcal M^{\epsilon_{\bullet}}_{SPHB}$ is a symplectic leaf of the most singular locus $\partial^n\mathcal M_{GHB}$.
\end{rema}

\begin{rema}
When the nodal curve is reducible as in remark \ref{Reducible}, the moduli space $\mathcal M_{GHB}$ is the union of two log-symplectic manifolds transversally intersecting along a smooth divisor (\cite[\S 9, remark 9.3]{3} and \cite{4}). It follows from remark \ref{Sepa} that the divisor is isomorphic (as a Poisson scheme) to the moduli space of stable (with respect to parabolic weights determined by the polarisation on the moduli space \cite[Lemma 3.4.2]{4}) parabolic-Higgs bundles with the same eigenvalues (of the Higgs field) at the two pre-images of the node. It follows from proposition \ref{Casimir} that the moduli space of strongly-parabolic Higgs bundles is a symplectic leaf of the divisor.
\end{rema}

\section{\textbf{Algebraically completely integrability}}

\subsection{\textbf{The Hitchin map and its general fibres}} There is a Hitchin map on the moduli of Gieseker-Higgs bundles which is defined as follows.

\begin{equation}
h: \mathcal M_{GHB}\rightarrow B:=\oplus_{i=1}^n H^0(X_0, \omega^{\otimes i}_{X_0})
\end{equation}
given by
\begin{center}
$(X_r, \mathcal E, \phi: \mathcal E\rightarrow \mathcal E\otimes \pi^*_r\omega_{X_0})\mapsto (\text{Trace}~~\phi, \dots, (-1)^{i-1}\text{Trace}~~(\wedge^i \phi),\dots, (-1)^{n-1}\text{Trace}~~(\wedge^n \phi))$
\end{center}
Notice, using properties \eqref{ProPro123}, it follows that 
\begin{equation}
\oplus_{i=1}^n H^0(X_r, \omega^{\otimes i}_{X_r})\cong \oplus_{i=1}^n H^0(X_0, \omega^{\otimes i}_{X_0}).
\end{equation} 
It is shown in \cite{3}, that the Hitchin map $h$ is proper.

For a general element $\xi\in B$, \cite[Section 7]{3} constructs a spectral curve $X_{\xi}$, which is an irreducible vine curve, ramified outside the nodes, such that there is the following correspondence.

\begin{equation}
\bigg\{\text{line bundles on the curve}~~X_{\xi} ~~\text{of degree}~~\delta \bigg\}\leftrightarrow\left\{
\begin{array}{@{}ll@{}}
\text{Gieseker-Higgs bundles on} X_0\\
\text{of rank}~~n~~\text{and degree}~~\delta-n(n-1)(g-1) \\
\text{with characteristic polynomial}~~ \xi
\end{array}\right\}.
\end{equation}

Therefore the subvariety consisting of the objects on the right is isomorphic to the Picard $Pic^{\delta}_{X_{\xi}}$ of the vine curve $X_{\xi}$, which is a semiabelian variety. The full Hitchin-fiber $h^{-1}({\xi})$ is a compactification of this semi-abelian variety with normal crossing singularity. Moreover, the smooth locus of $h^{-1}(\xi)$ is precisely $Pic^{\delta}_{X_{\xi}}$. For the precise statements we refer to \cite[Theorem 8.16((Quasi-abelianization))]{3}.

\subsection{\textbf{Completely integrability}}
\begin{defe}\label{Lagrangian}
Let $X$ be a variety with normal crossing singularity with a log-symplectic form i.e., a non-degenerate closed section of $\wedge^2 \Omega_X(log~~\partial X)$. An irreducible subvariety $Y\subset X$ is co-isotropic (resp. Lagrangian) if it is generically a co-isotropic(resp. Lagrangian ) subvariety of a symplectic leaf; i.e., Y is contained in the closure $\overline{S}$ of a symplectic leaf $S\subset X$ and the intersection $Y\cap S$ is a co-isotropic (resp. Lagrangian ) subvariety of $S$.
\end{defe}

\begin{defe}\label{ACIS}
A Poisson structure (may not be of uniform rank) on a variety $X$ (possibly singular) is an algebraically completely Integrable system structure on $h: X\rightarrow B$ if $h$ is a Lagrangian fibration over the complement of some properly closed subvariety of $B$.
\end{defe}

From Theorem \ref{Foliation11}, it follows that $\mathcal M_{GHB}\setminus \partial \mathcal M_{GHB}$ is the maximal symplectic leaf (dense open) of $\mathcal M_{GHB}$. From the description of the fibre of the Hitchin map \cite[Theorem 8.16((Quasi-abelianization))]{3}, it follows that the fibre of $h: \mathcal M_{GHB}\setminus \partial \mathcal M_{GHB}\rightarrow B$ over a general point $\xi$ is the Picard variety $J_{\xi}$ of the vine curve $X_{\xi}$. We will now show that the general fibre is a Lagrangian in $\mathcal M_{GHB}\setminus \partial \mathcal M_{GHB}$ following the strategy of \cite{23} and \cite{22}.

\begin{lema}
Let $f:X\rightarrow Y$ be a morphism between two projective nodal curves of degree $d$ and unramified along the nodes. Then
\begin{enumerate}
\item we have a short exact sequence
\begin{equation}\label{Eq1001}
0\rightarrow f^*\omega_Y\rightarrow \omega_X \rightarrow \omega_{X/Y} \rightarrow 0,
\end{equation}
where $\omega_X$ and $\omega_Y$ are the sheaf of differentials of the nodal curves with logarithmic poles along the nodes and $\omega_{X/Y}$ denote the sheaf of relative logarithmic differentials.
\item $\omega_{X/Y}\cong \mathcal O_R$, where $R$ is the divisor $R:=\sum_{_{p\in X}} \text{length}~~(\omega_{X/Y})_p\cdot p$.
\item $2g_X-2=d(2g_Y-2)+\text{deg}~~R$, where $g_X$ and $g_Y$ are the arithmetic genus of $X$ and $Y$, respectively.
\end{enumerate}
\end{lema}
\begin{proof}
Since $f$ is unramified at the node by the functorial property of the sheaf of logarithmic differentials, we have the following inclusion of rank $1$-locally free sheaves $f^*\omega_Y\rightarrow \omega_X$. The locus, where the inclusion is not an isomorphism, is, by definition, the ramification divisor $R$. Therefore $\mathcal O_R\cong \omega_{X/Y}$. The statement $(3)$ follows from \ref{Eq1001}.
\end{proof}

The following lemma is a straightforward generlisation of \cite[\S 5]{15}, \cite[Remark 3.7]{6} and \cite[\S 4.3]{16}. One can also find a proof in \cite[Proposition 4.1]{9}.
\vspace{-.2em}
\begin{lema}
Let $L$ be a line bundle on the vine curve $X_{\xi}$ such the push-forward of $L\rightarrow L\otimes f^* \omega_{X_0}$ is a Gieseker-Higgs bundle $\phi: \mathcal E\rightarrow \mathcal E\otimes \omega_{X_0}$, where the map $L\rightarrow L\otimes f^* \omega_{X_0}$ is given by the multiplication by the canonical section of $f^*\omega_{X_0}$ on $X_{\xi}$. Suppose also that the spectral curve $f: X_{\xi}\rightarrow X_0$ is ramified along a divisor $R\subset X_{\xi}$, which does not map to any of the nodes. Then we have the following exact sequence over $X_{\xi}$
\begin{equation}\label{Spectral1}
0\rightarrow L(-R)\rightarrow f^*\mathcal E\rightarrow f^*\mathcal E\otimes \pi^*_{\xi}\omega_{X_0}\rightarrow L\otimes f^*\omega_{X_0}\rightarrow 0.
\end{equation}
\end{lema}

\begin{prop}\label{Inter420}
Let $p:=(X_0, \mathcal E, \phi)$ be a Gieseker-Higgs bundle in $h^{-1} ({\xi})$. Then we have the following short exact sequence
\begin{equation}\label{Hype}
0\rightarrow H^1(f_*\mathcal O_{X_{\xi}})\rightarrow \mathbb H^1(\mathcal C_{\bullet})\rightarrow H^0(f_*\omega_{X_{\xi}})\rightarrow 0
\end{equation}
Moreover, $T_{J_{\xi},p} \cong H^1(f_*\mathcal O_{X_{\xi}})$ and $\mathcal N_{J_{X_{\xi}}/\mathcal M_{GHB},p}\cong H^0(f_*\omega_{X_{\xi}})$.
\end{prop}

\begin{proof}
Tensoring the sequence \eqref{Spectral1} with $L^{-1}\otimes \mathcal O(R)$ we get
\begin{equation}\label{Spectral2}
0\rightarrow \mathcal O_{_{X_{\xi}}}\rightarrow \pi^*E\otimes L^{-1}\otimes \mathcal O(R)\rightarrow \pi^*E\otimes \pi^*\omega_{_{X_0}}\otimes L^{-1}\otimes \mathcal O(R)\rightarrow \pi^*\omega_{_{X_0}}\otimes \mathcal O(R) \rightarrow 0
\end{equation}
Notice that the morphism $\pi:X_{\xi}\rightarrow X_0$ is a finite cover. Using Riemann-Hurwitz formula, we see that the push forward (under $\pi$) of the above exact sequence is the same as
\begin{equation}\label{Spectral3}
0\rightarrow \pi_*\mathcal O_{_{X_{\xi}}}\rightarrow \mathcal End E\xrightarrow{[\bullet, \phi]} \mathcal End E\otimes \omega_{_{X_0}}\rightarrow \pi_*\omega_{_{X_{\xi}}} \rightarrow 0
\end{equation}
Using this exact sequence, we can form the following short exact sequence of chains:

\begin{equation}\label{Spectral4}
\begin{tikzcd}
& 0\arrow{d} & 0\arrow{d} \\\
0 \arrow{r} & \pi_*\mathcal O_{_{X_{\xi}}}\arrow{r}\arrow{d} & 0\arrow{r}\arrow{d} & 0\\
0 \arrow{r} & {\mathcal End}~~ E \arrow{r}{[\bullet, \phi]} \arrow{d}{[\bullet, \phi]} & {\mathcal End}~~ E\otimes \omega_{_{X_0}}\arrow{r}\arrow{d} & 0\\
0\arrow{r}& {Im}~~ ({[\bullet, \phi]}) \arrow{r}\arrow{d} & {\mathcal End}~~ E\otimes \omega_{_{X_0}}\arrow{r}\arrow{d} & 0\\
& 0 & 0
\end{tikzcd}
\end{equation}
The last chain is quasi-isomorphic to the chain 
\begin{equation}\label{eq20225}
0\rightarrow 0\rightarrow \pi_*\omega_{_{X_{\xi}}} \rightarrow 0.
\end{equation}
This follows from the following short exact sequence of complexes:

\begin{equation}
\begin{tikzcd}
0\arrow{r} & {Im}~~({[\bullet, \phi]})\arrow["="]{r}\arrow["="]{d}& {Im}~~({[\bullet, \phi]})\arrow{r}\arrow{d}& 0\arrow{r}\arrow{d} & 0\\
0\arrow{r} & {Im}~~({[\bullet, \phi]})\arrow{r}& {\mathcal End}~~E\otimes \omega_{_{X_0}}\arrow{r}& \pi_*\omega_{_{X_{_{\xi}}}}\arrow{r} & 0
\end{tikzcd}
\end{equation}
Since the first vertical complex from the left is quasi-isomorphic to $0$, the vertical complex in the middle and the first vertical complex from the right are quasi-isomorphic.

Now from \ref{Spectral4} and using the long exact sequence of hypercohomology, we get equation \ref{Hype}.
The tangent space $T_{_{p}} J_{_{\xi}}\cong H^1(X_{_{\xi}}, \mathcal O_{_{X_{\xi}}})\cong H^1(X_{_{0}}, \pi_*\mathcal O_{_{X_{\xi}}})$. Therefore, $\mathcal N_{_{J_{_{X_{\xi}}}|\mathcal M_{_{{GHB}}}}}\cong H^0(\pi_*\omega_{_{X_{_{\xi}}}})$.
\end{proof}
\vspace{-.2em}
\begin{thm}
The generic fiber $h^{-1} (\xi)$ is Lagrangian in a symplectic leaf for the log-symplectic structure on $\mathcal M_{GHB}$. Therefore the Hitchin map $h: \mathcal M_{GHB}\rightarrow B:=\oplus_{i=1}^n H^0(X_0, \omega_{X_0}^{\otimes i})$ is an algebraically completely integrable system (\ref{ACIS}). 
\end{thm}
\vspace{-.2em}
\begin{proof}
By definition \ref{Lagrangian}, it is enough to show that the isomorphism $\sigma^{\flat}: \mathbb H^1(\mathcal C_{_{\bullet}}^{\vee})\rightarrow \mathbb H^1(\mathcal C^{}_{_{\bullet}})$ maps $\mathcal N^{^{\vee}}_{_{J_{_{X_{\xi}}}|\mathcal M_{_{{GHB}}}}}$ to $T_{_{p}} J_{_{\xi}}$.
To see this, consider the following diagram:
\begin{equation}
\begin{tikzcd}
0\arrow{r} & H^1(\pi_*\mathcal O_{_{X_{\xi}}}) \arrow{r} & \mathbb H^1(\mathcal C_{_{\bullet}})\arrow{r} & H^0(\pi_*\omega_{_{X_{\xi}}})\arrow{r} & 0\\
0\arrow{r} & H^0(\pi_*\omega_{_{X_{\xi}}})^* \arrow{r} & \mathbb H^1(\mathcal C^{^{\vee}}_{_{\bullet}})\arrow["\sigma^{\flat}"]{u}\arrow{r} & H^1(\pi_*\mathcal O_{_{X_{\xi}}})^*\arrow{r} & 0
\end{tikzcd}
\end{equation}
Since $\mathcal N_{_{J_{_{X_{\xi}}}|\mathcal M_{_{{GHB}}}}}\cong H^0(\pi_*\omega_{_{X_{\xi}}})$ and $T_{_{p}} J_{_{\xi}} \cong H^1(\pi_*\mathcal O_{_{X_{\xi}}})$ it is enough to show that $\sigma^{\flat}(H^0(\pi_*\omega_{_{X_{\xi}}})^*)\subseteq H^1(\pi_*\mathcal O_{_{X_{\xi}}})$. Since the horizontal short exact sequences are exact, it is enough to show that the composite map 

\begin{equation}\label{comp420}
H^0(\pi_*\omega_{_{X_{\xi}}})^*\rightarrow H^0(\pi_*\omega_{_{X_{\xi}}})
\end{equation}

is $0$. Notice that $H^0(\pi_*\omega_{_{X_{\xi}}})^*\cong H^1(\pi_*\mathcal O_{X_{\xi}})$ and also recall from the proof of proposition \ref{Inter420} that the complex
\begin{equation}
\pi_*\mathcal O_{X_{\xi}}\rightarrow 0
\end{equation}
is quasi-isomorphic to the complex \eqref{eq20225}.
Now it is clear that the morphism \eqref{comp420} is the same as the composite of the morphisms induced on the $\mathbb H^1$'s of the following morphism of complexes.
\begin{equation}
\begin{tikzcd}
\pi_*\mathcal O_{_{X_{\xi}}}\arrow{r} \arrow{d} & \mathcal End E \arrow{d}{-[\bullet, \phi]}\arrow["\mathbb 1"]{r} & \mathcal End E\arrow{d}{[\bullet, \phi]}\arrow{r} & {Im}([\bullet, \phi])\arrow{d}\\
0\arrow{r} & \mathcal End E\otimes \omega_{_{X_{_0}}} \arrow["-\mathbb 1\otimes \mathbb 1"]{r} & \mathcal End E\otimes \omega_{_{X_{_0}}} \arrow{r}& \mathcal End E\otimes \omega_{_{X_{_0}}}
\end{tikzcd}
\end{equation}

Since the composition is $0$, therefore the morpshim $H^0(\pi_*\omega_{_{X_{\xi}}})^*\rightarrow H^0(\pi_*\omega_{_{X_{\xi}}})$ is also $0$.
\end{proof}

\end{document}